\documentclass[final]{amsart}
\usepackage{graphicx} 
\usepackage{amsmath}
\usepackage{mathtools}
\usepackage{amssymb}
\usepackage{amsthm}
\usepackage[colorlinks]{hyperref}
\usepackage[nameinlink]{cleveref}
\usepackage[notref,notcite,color]{showkeys}
\usepackage{enumitem}
\usepackage{bbm}
\usepackage{mathrsfs}

\newtheorem{ass}{Assumption}[]
\newtheorem{fass}{Functional Assumption}[]
\newtheorem{sass}{SPDE Assumption}[]

\newtheorem{thm}{Theorem} 
\newtheorem{deff}[thm]{Definition}
\newtheorem{lem}[thm]{Lemma}
\newtheorem{cor}[thm]{Corollary}

\newtheorem{rem}[thm]{Remark}

\numberwithin{thm}{section} 
\numberwithin{equation}{section} 

\newcommand{\E}{\mathbb{E}}
\newcommand{\F}{\mathcal{F}}
\newcommand{\G}{\mathcal{G}}
\newcommand{\U}{\mathcal{U}}

\newcommand{\Prb}{\mathbb{P}}
\newcommand{\mcM}{\mathcal{M}}

\newcommand{\Pp}{\mathcal{P}}
\newcommand{\N}{\mathbb{N}}
\newcommand{\R}{\mathbb{R}}
\newcommand{\Id}{\mathbbm{1}}
\newcommand{\inv}{\mathcal{M}_+}
\newcommand{\tinv}{\tilde{\mathcal{M}}_+}

\newcommand{\D}{\mathbb{D}}
\newcommand{\Dm}{\mathcal{D}}
\newcommand{\Dme}{\mathcal{D}^{ext}}
\newcommand{\tDme}{\tilde{\mathcal{D}}^{ext}}
\newcommand{\K}{\mathcal{K}}

\newcommand{\Ll}{\mathcal{L}}

\newcommand{\ip}[2]{\langle #1, #2 \rangle}
\newcommand{\T}{\mathbb{T}}
\newcommand{\tr}{\operatorname{tr}}

\newcommand{\B}{\mathcal{B}}
\newcommand{\X}{\mathcal{X}}
\newcommand{\Y}{\mathcal{Y}}

\newcommand{\cdom}{C(\overline D)_+}

\author[J. F{ö}ldes]{Juraj F{ö}ldes}
\address{Juraj F{ö}ldes\footnote{Juraj F\"oldes was partly supported under the grants: NSF DMS-1816408,  
Simons foundation SFI-MPS-TSM-00013942, U.S. DOE DE-NA0004269
}, Dept. of Mathematics, University of Virginia, Kerchof Hall,
Charlottesville, VA 22904-4137}
\email{foldes@virginia.edu}

\title{Stochastic Persistence in Infinite Dimensions}
\author{Declan Stacy}
\address{Declan Stacy, Institute of Mathematics, EPFL, MA C1 615 (Bâtiment MA), 
Station 8,
1015 Lausanne, Switzerland}
\email{declan.stacy@gmail.com}

\date{January 2026}

\begin{document}

\begin{abstract}
    Motivated by infinite-dimensional ecological and biological models such as reaction-diffusion SPDEs and stochastic functional differential equations, we develop a general criteria for stochastic persistence (coexistence) in terms of an average lyapunov function, which was previously known only in  finite dimensions. To apply our results to SPDEs we  analyze the projective process, and we employ a combination of mild (stochastic convolution) and variational (lyapunov function) techniques. Our analysis also requires some nontrivial well-posedness and nonnegativity results for reaction-diffusion SPDEs, which we state and prove in great generality, extending the known results in the literature. Finally, we present  several examples including ecological models (Lotka-Volterra), an epidemic model (SIR), and a model for turbulence. Notably we show that, as in the SDE case, coexistence in the Lotka-Volterra model is determined by the invasion rates. 


\end{abstract}

\maketitle 

\section{Introduction}
A prominent feature of many dynamical systems is the existence of invariant subsets of the state or phase space. For instance, in physical systems, level sets of energy or momentum are invariant under the dynamics.
Moreover, often in practice and theoretical analysis, the system is affected by random influences which requires mathematical modeling via stochastic dynamical systems. As in the deterministic case, there can be a stochastically invariant set, which we denote by $\mcM_0$. For example, in biology $\mcM_0$ might represent the subset where a species is extinct, and therefore we refer to $\mcM_0$ as an extinction set.

To assess the relevance of such an extinction set $\mcM_0$ based on its occurrence in observations, experiments, or numerical computation, one investigates the asymoptotic stability or instability of $\mcM_0$. The former means that if the the initial condition is close to $\mcM_0$, then the system converges, as $t \to \infty$, to $\mcM_0$. The latter, which is the focus of our paper, captures the idea that trajectories remain bounded away from the extinction set $\mcM_0$.

Stochastic persistence (instability) is a central concept in the study of long-term behavior of random dynamical systems. Originally developed in the context of finite-dimensional systems, the theory (and its complement, stochastic extinction) has proven to be a powerful tool for studying population dynamics, epidemiology, turbulence, and chemostat models (\cite{persistence, extinction}). However, many models of interest in biology, physics, and other applied sciences are naturally infinite-dimensional. 
As articulated in \cite{deterministic-SIR}, ``the spread of disease may be faster in one part of the population and slower in the other part ... the spatial effect cannot be ignored." Thus, models incorporating spatial dependency (SPDEs) can be more realistic than classical SDE models. The same applies for models where future behavior depends on (part of) the past, not just on the present state. As articulated in \cite{functional}, ``the delays or past dependence are unavoidable in natural phenomena
 and dynamical systems; the framework of stochastic functional differential equations is more realistic, more effective, and more general for the population dynamics in real life than a stochastic differential equation counterpart." 
 
 In the deterministic setting, there is a vast literature analyzing the persistence and stability of reaction-diffusion PDEs (see for example \cite{carrying-capacity, allee-effect, deterministic-two-species, biology-waves, reac-diff-thesis, reac-diff-thesis-2, pattern-formation-deterministic-computational, deterministic-lotka-volterra}). Although  recently there has been some treatment of the stochastic case for various examples, see \cite{space-time, stochastic-SIR-analytical, traveling-waves-mult-noise, mult-noise, stochastic-SIR-NN, stochastic-SEIRS, stochastic-lotka-volterra-SPDE},  there is no unified theory. 
The main goal of our paper is to 
develop a general framework for infinite dimensional stochastic dynamical systems that is readily applicable to 
stochastic partial differential equations (SPDEs), which incorporate spatial dependency, and stochastic functional differential equations, which account for delayed effects. 
We stress that
 several crucial arguments used in the finite-dimensional case no longer apply in infinite dimensions, and therefore
 we propose a new approach to investigate infinite dimensional stochastic persistence.

More specifically, we prove under general  conditions that infinite-dimensional stochastic systems are persistent if an  ``average Lyapunov exponent" is negative and give its formula
(see below for precise definitions).  For stochastic systems, the average Lyapunov exponent is often hard to compute explicitly, so we also show the following perturbation result: if the average Lyapunov exponent of the corresponding deterministic system is negative, then the stochastic system is persistent if the random forcing is sufficiently weak. 
Afterwards we apply our main abstract result to stochastic functional (delayed) differential equations, where our results considerably simplify arguments in \cite{functional} and allow for deeper understanding of the results.

The analysis of SPDEs is quite challenging due to different functional frameworks, notions or solutions, and where the change of one sign can lead to blow-up solutions. 
We devote our attention to SPDEs with
nonnegative solutions, which are the most relevant for biological applications. We combine the Lyapunov function method with powerful tools from semigroup theory to prove 
persistence results for a general class of SPDEs under moderate assumptions. Then, in many examples, we provide criteria that are sharp at the heuristic level; their sharpness will be established rigorously through a complementary extinction theory developed in a forthcoming paper.
 Additionally, our conditions on SPDEs allow for noise which is much closer to white noise compared to the usual Hilbert-Schmidt assumption.  However, in order to treat such stochastic forcing,  we have to prove  well-posedness, non-negativity, and certain It\^{o} formulas for mild solutions, which extend the works  \cite{Cerrai2003, nhu-positivity, feller-stoch-reac-diff} and are 
interesting in their own right. We remark that proofs of many technical and abstract results (especially in \Cref{sec:spde-well-posed}) can be omitted on the first reading. The reader who is less concerned with the dynamics and primarily interested in general SPDE theory can find the contemporary analysis of stochastic reaction-diffusion equations by reading \Cref{sec:spde-well-posed} and the preceding subsections in \Cref{section-SPDE}.

At the conceptual level, 
 the contribution of the present manuscript is twofold. First, we introduce new techniques for obtaining asymptotic  stability of infinite-dimensional systems, which are necessary for accurate analysis of models, for example in epidemics.
 Second, we further the development of the average Lyapunov function approach, which has proven to be optimal for analyzing the stability of fixed points and, more generally, invariant subsets of (random) dynamical systems (see the discussion in the introduction of \cite{extinction}). From a broad perspective, the average Lyaupov function method  allows us to linearize the problem about the invariant subset $\mcM_0$. The linearized problems are easier to analyze, as their behavior depends only upon a single quantity: the average Lyapunov exponent. Despite its abstract definition as $\sup_{\mu \in P_{inv}(\mcM_0)} \mu H$ (see \Cref{as-V}\ref{as-alpha}), the average Lyapunov exponent is a familiar concept in concrete contexts. For example, in spatially homogeneous (finite dimensional) ecological models the average Lyapunov exponent has a natural interpretation as an ``invasion rate" (\cite{persistence, functional, ecologicalContinuous, ecologicalDiscrete, ecologicalGeneral}).  Our method allows us to extend the notion of the invasion rate to spatially dependent (infinite-dimensional) contexts, which is higly non-trivial (see \Cref{rem:on-inv-rates}). For deterministic PDEs with a fixed point, the average Lyapunov exponent is the principal eigenvalue which we express as the maximum of a Rayleigh quotient.  Thus our results in \Cref{SPDE-main-results} can be viewed as a generalization of the classical notion of the principal eigenvalue to the stochastic setting (see \Cref{rem:rayleigh-quotient}). Overall, our results allow us to leverage the power of linearization to analyze the long-term behavior of more realistic applied models compared to the known procedures in the literature. 

Compared to the traditional Lyapunov function approach, which require certain conditions on the generator of the Markov process to hold at every point, average Lyapunov functions only look at averages with respect to invariant measures on the subset $\mcM_0$. In applications, average Lyapunov functions are thus much easier to construct and are less dependent on the specific form of the system. On the other hand, establishing properties of 
stochastic dynamical systems under assumptions on averaged rather than traditional Lyapunov functions is much more challenging and it is the purpose of the present paper. The essential step 
 to prove stochastic persistence is to formulate conditions which guarantee that a certain family of random measures $\mu_t$ (the empirical occupation measures) are tight almost surely, and that a certain unbounded function $H$ satisfies $\int Hd\mu_t \to \int Hd\mu$ if $\mu_t \to \mu$ (see  precise definitions below). The key conditions in  \cite{persistence, extinction} require the existence of functions $W,W'$ and a constant $K > 0$ which satisfy, among other assumptions:

\begin{enumerate}[label=(\roman*)]
    \item $\Ll W \leq K - W'$, where $\Ll$ is the infinitesimal generator.
    \item $W'$ has compact sublevel sets.
    \item $W'$ is a lot bigger than $|H|$.
\end{enumerate}

For finite-dimensional systems, finding such $W,W'$ in applications is usually quite natural, however, it is impossible in infinite-dimensional Banach spaces, where continuous functions with compact sublevel sets do not exist. In addition, SPDEs and SDEs with delay present their own additional challenges for proving the tightness and convergence results (for $\mu_t$). In particular, (i) is very hard to establish for SPDEs defined on the unit sphere in a Banach space (Lyapunov exponent methods naturally introduce a polar change of coordinates, which leads to equations on sphere), while (ii)-(iii) are difficult to prove for SDEs with delay.

To track the history for SDEs with delayed effects up to $r > 0$ units of time into the past, the natural state space is $C([-r,0], \R^n)$, the set of all continuous functions from $[-r,0]$ to $\R^n$. One criterion for compactness in $C([-r,0], \R^n)$, see for example Arzel\` a-Ascoli theorem,  is a uniform bound on the H\"{o}lder continuity of $[T-r,T] \ni t \mapsto X_t$. Inspired by \cite{functional}, instead of requiring (ii), we require the process to belong a compact set on the event that $W(X_t)$ is uniformly bounded for a suitably long period of time, a condition, which is easily verified via Kolmogorov continuity criterion. We also replace (iii) with the assumption that $\int_0^r |H|$ is a lot smaller than $\int_0^r W'$. Then, the proof of tightness of $\mu_t$ and convergence of $\int H d\mu_t$ resembles \cite[Section 3]{functional}, but
 by adapting the arguments in \cite{persistence}
we avoid the 13-14 pages of arguments in \cite[Section 4]{functional}. Also, our results are in some sense more fundamental, because they also apply to SPDEs.

For parabolic SPDEs we consider systems of $m$ equations such that the nonlinearity is dissipative. For these equations, one can often use energy estimates to verify (i); for example, take $W$ to be some norm (for example, $W(u) = \|u\|_{L^2}^2$) and $W'$ to be a stronger norm (like $\|u\|_{H^1}^2$).
For our applications in physics or biology (and to avoid complicated geometrical and topological arguments) we 
take $\mcM_0$ to be a subspace where some components $1, \ldots, d$, with $d < n$ are equal to zero. Then polar coordinates in the first $d$ components are natural as they keep track of the distance of the process to $\mcM_0$.  
Specifically, if we denote $u_t$ the projection of our process on the first $d$ coordinates, we define the norm of our process $u_t$ given by $r_t = \|u_t\|$ and the angle $v_t = u_t/\|u_t\|$. However, since the polar transformation is degenerate at the origin (if $r = 0$, then the angle can be arbitrary), we must also ``blow up" $\mcM_0$ so that $\{u=0\}$ is replaced by the sphere $\{r=0\}$ ($v$ arbitrary) -- otherwise, it is not possible to define a certain function ($H$, the generator applied to our average Lyapunov function) continuously. After enlarging our state space in this way, the energy estimates no longer give us (i).
Indeed, since $\|v_t\|$ is constant, it does not satisfy a dissipative problem.

The problem of obtaining tightness of empirical occupation measures for the projective process $(r_t,v_t)$ is quite difficult \cite{projective-process-1, projective-process-2}, in particular measures supported on $r = 0$. For example, consider the heat equation $u_t = \Delta u$. Since there is an invariant measure (for $v_t$) supported on any eigenfunction direction,
 it is easy to see that the set of all invariant measures is not compact. On the other hand, the assumptions for finite dimensional problems in \cite{persistence, extinction} (especially assumption (i) above) imply that the set of invariant measures is compact. Thus, it is not possible to use the same techniques as in the finite-dimensional case to handle these projective processes. Even in this manuscript, our assumptions imply compactness of the set of all invariant measures for the projective process on $\{r =0\}$, which for a general SPDE would require some non-degeneracy assumption on the noise.

 We avoid these issues by looking only at the cone of nonnegative functions. By recalling our example of the heat equation from above, we note that for nonnegative solutions there is only one invariant measure for the projective process (corresponding to the largest eigenvalue of the laplacian), because eigenfunctions corresponding to higher frequencies are not nonnegative.  Since non-negativity is a natural assumption in biological models (negative population sizes are usually not considered), this restriction covers many important applications.

 Even with this positivity assumption, however, showing compactness is quite hard and involves the combination of several Lyapunov functions (a variational method) with novel techniques involving the mild formulation of SPDEs (semigroup methods). Our results improve upon those in \cite{space-time, stochastic-SIR-analytical} because in addition to proving the required tightness, we also provide an explicit formula for the Lyapunov exponent which guarantees persistence, as opposed to just bounds. 
 
 Also, our general Lyapunov functions apply to variety of SPDEs and in this paper we present details for the following problems:
\begin{itemize}
    \item A logistic growth model with harvesting on a domain (see \Cref{example-logistic}) with multiplicative noise, which is a stochastic version of the deterministic model analyzed in \cite{carrying-capacity}. We formulate criteria for the survival of all species.
    \item A model for turbulence with (multiplicative) space-time white noise on a torus \cite{space-time, physics-interpretation} (see \Cref{example-torus}). We show that when the noise is small enough, there are multiple invariant measures.
    \item A SIR epidemic model on a domain with multiplicative noise \cite{stochastic-SIR, stochastic-SIR-analytical, stochastic-SIR-NN, nhu-positivity} (see \Cref{example-SIR}). We formulate criteria that guarantee that the disease is endemic. Whereas the extinction set in the previous examples was the fixed point $0$, this extinction set is more complicated, further illustrating the power of our theory. Additionally, unlike in the cited papers, we obtain almost-sure results about persistence of the disease, as opposed to bounds on certain expectations.
    \item Competitive Lotka-Volterra models on a domain with a variety of noise options. Whereas \cite{stochastic-lotka-volterra-SPDE} focuses on extinction, we are able to prove coexistence results. We encourage the reader to read \cite[7.3. Discussion]{stochastic-lotka-volterra-SPDE}, which outlines the difficulties in applying the ``growth rate method" or the ``eigenvalue method" to SPDEs. In this paper, we are able to perform both.
\end{itemize}

 We do not provide examples for stochastic functional differential equations, as there are many in \cite[Section 5]{functional} and although our method provides a shorter proof, it does not improve known results in these cases.

We remark that under the positivity preserving assumption one can introduce the Cayley-Hilbert metric on a simplex
 and show contractivity for Markov semigroups corresponding to linear or homogeneous problems
 as in \cite{BenaimLobrySari25, Bushell86, cohenfausti2024}. Thus, one could try to prove $\mu_t \to \mu$ by generalizing these methods, although it is not clear how to establish the contractivity for the non-homegeneus, non-linear problem in infinite dimensions that we analyze here. We opted for a different approach which is more in line with \cite{persistence} and generalizes also to the stochastic functional differential equations case. 
 
To summarize, the present manuscript  pushes the boundary of the current theory of stochastic persistence by providing a unified framework for treating a wide variety of infinite-dimensional systems. After formulating our setup and \Cref{as1}--\ref{as-compact} in \Cref{deff-and-as}, we state and prove our general results in \Cref{results} and \Cref{sec:proofs}. In \Cref{sec:funct} we show how to verify \Cref{as1}--\ref{as-compact} for stochastic functional differential equations. In \Cref{section-SPDE}, we consider a general stochastic partial differential equation with nonnegative solutions and formulate assumptions which guarantee that \Cref{as1}--\ref{as-compact} hold, as well as prove novel results relating to well-posedness and the weak maximum principle. This is where we construct our newly developed Lyapunov functions and ``upgrade" the control that  we obtain using the semigroup (mild) approach. Finally, in \Cref{examples} we illustrate how our general results apply to SPDE models for population dynamics, turbulence, and epidemics.
The reader primarily interested in examples can go straight to \Cref{examples} after reading the introduction of \Cref{section-SPDE}.

\section{Definitions and Assumptions}\label{deff-and-as}
Let $(\mcM,d)$ be a Polish (complete and separable) metric space which is not necessarily locally compact, $\mcM_0 \subset \mcM$ be a closed set, and $\inv \coloneqq \mcM_0^c$ (where $\mcM_0^c$ denotes the complement of $\mcM_0$). The set $\mcM_0$ can be viewed as an ``extinction" set and $\inv$ as a ``persistence'' set.
 We endow $\mcM$ with the Borel $\sigma$- algebra. For each $x \in \mcM$, let $X_t^x$ be a homogeneous Markov process defined on $\mcM$ which has cadlag sample paths. By this we mean that there is a filtered probability space $(\Omega, \mathcal{F}, \{\mathcal{F}_t\}_{t \geq 0}, \Prb)$ (which we assume is complete and right-continuous) and a family of $\mcM-$valued random variables $\{X_t^x\}_{x \in \mcM, t \geq 0}$ such that:

\begin{itemize}
    \item $X_0^x = x$ and $t \mapsto X_t^x$ is cadlag (right continuous with left limits) a.s.
    \item $X^x_\cdot$ is adapted to $\{\mathcal{F}_t\}_{t \geq 0}$, meaning $X_t^x$ is $\mathcal{F}_t$ measurable for each $t \geq 0$.
    \item \label{semigroup} For all bounded measurable functions $f:\mcM \to \mathbb{R}$, the map
    $$
    [0,\infty) \times \mcM \ni (t,x) \mapsto \Pp_t f(x) \coloneqq \E[f(X_t^x)] 
    $$ 
    is measurable and for any $s, t \geq 0$ we assume (homogeneity)
    that
    $$
    \Pp_s f(X_t^x) = \E[f(X_{t + s}^x)|\mathcal{F}_t] \,.
    $$    
\end{itemize}

It is standard to prove that $(\mathcal{P}_t)_{t \geq 0}$ defines 
a semigroup: for any $s, t \geq 0$ it holds that $\Pp_{s+t}f = \Pp_s \Pp_tf$. Also, the definition of $\Pp_tf$ makes sense as long as $f: \mcM \to \R$ is measurable and bounded from below (or above), possibly attaining the value $\infty$ (or $-\infty$).

\subsection{Invariant Sets and the Feller Property }
Next, we list two basic assumptions on our process $X_t^x$ and the sets $\mcM_0, \inv$. First, we suppose that the sets $\mcM_0$ and $\inv$ are invariant, that is, 
if an initial condition is in $\mcM_0$ (respectively $\inv$) then $X_t^x$ belongs to $\mcM_0$ (respectively $\inv$) for all $t \geq 0$. 
In general, we have the following definition.

\begin{deff}
\label{inv-set}
    A measurable set $A \subset \mcM$ is invariant if $x \in A$ implies that almost surely $X_t^x \in A$ for all $t \geq 0$.
\end{deff}

\begin{ass} \label{as1}
    $\mcM_0$ and $\mcM_+$ are invariant.
\end{ass}

Second, we assume a standard continuity on the law of $X_t^x$ as a function of the initial condition $x$ and time $t$. Such a condition is commonly known as the ``Feller" property.
Let $C_b(\mcM)$ denote the space of bounded continuous functions on $\mcM$ endowed with the supremum norm $\|f\| = \sup_{x \in \mcM} |f(x)|$. The following continuity assumption is vital for our analysis:

\begin{ass} \label{as2}
    Assume that $X_t^x$ is Feller in the sense that for any $f \in C_b(\mcM)$ and $t \geq 0$, $\Pp_tf \in C_b(\mcM)$, and also $\Pp_tf \to f$ pointwise as $t \downarrow 0$.
\end{ass}

\begin{rem}
Note that the requirement of right-continuity in time ($\Pp_tf \to f$ pointwise as $t \downarrow 0$) is automatically satisfied since $X^x_t$ is assumed to have cadlag paths.
\end{rem}

 \subsection{Two Important Martingales} \label{generator-stuff}
In this section we introduce important definitions that are used in \Cref{lyap}, as well as two martingales that are crucial to our  analysis.

Two central operators in the study of Feller processes are the generator $\mathcal{L}$ and Carre du Champ $\Gamma$ of the Markov semigroup $\Pp_t$, which are discussed rigorously in \cite{extinction}. For a suitably nice function $f: \mcM \to \R$,  $\mathcal{L}f(x)$ represents the average rate of change of $\E[f(X_t^x)]$ at $t = 0$, so that $M_t^f$ defined as :
\begin{equation} \label{Mtf}
     M_t^f(x) \coloneqq f(X_t^x) - f(x) - \int_0^t \mathcal{L}f(X_s^x)ds   
    \end{equation}
is a martingale. Analogously,     
$\Gamma f(x)$ represents the average rate of change of the quadratic variation of \eqref{Mtf} at $t = 0$, so that the following process is a martingale:
    \begin{equation}\label{Mtf-quad-var}
        (M^f_t(x))^2 - \int_0^t \Gamma f(X_s^x)ds \,.
    \end{equation}

For readers familiar with It\^o calculus, we remark that $\mathcal{L}f$ (resp. $\Gamma f$) is the
 coefficient of the ``$dt$" part of ``$df$" (resp. ``$(df)^2$").  This interpretation is useful when computing terms like $\Ll f^k$ in the context of SPDEs, where one can utilize It\^o's formula.

\begin{deff}
\label{D+}
    Let $A \subset \mcM$ be an open invariant set (see \Cref{inv-set}). We let $\Dme_+(A)$ be the set of all measurable functions $f:A \to [0,\infty)$ for which there exists a measurable function $\mathcal{L}f: A \to \R \cup \{-\infty\}$ such that $M^f_{\cdot}(x)$
   in \eqref{Mtf} is a cadlag local martingale for all $x \in A$.
\end{deff}

\begin{deff}
    \label{D2}
   Let $A \subset \mcM$ be an open invariant set (see \Cref{inv-set}). We let $\Dme_2(A)$ be set of all measurable functions $f:A \to \R$ for which there exist measurable functions $\mathcal{L}f: A \to \R \cup \{-\infty\}$ and $\Gamma f: A \to [0,\infty)$ such that $M^f_{\cdot}(x)$
   in \eqref{Mtf} is a cadlag square integrable martingale and the stochastic process in \eqref{Mtf-quad-var} is a martingale for all $x \in A$.
\end{deff}

\begin{rem}
    Note that in \Cref{D+} and \Cref{D2} we allow $\Ll f$ take the value $-\infty$, but it is implicit in the definitions that $s \mapsto \Ll f(X_s^x)$ must be in $L^1([0,t])$ almost surely for all $t > 0$. Allowing infinite values is crucial for our study of SPDEs, where $f$ is a norm and $\Ll f$ includes a term which is the negative of some stronger norm.
\end{rem}

\subsection{Lyapunov Functions}
\label{lyap}
In this section we provide the assumptions necessary for our main results. In \Cref{as-W} we suppose that there is a Lyapunov function which is large near spatial infinity. In conjunction with \Cref{as-compact}, it allows us to prove tightness of appropriate measures. \Cref{as-V} is crucial as it gives the existence of an ``average Lyapunov function" that forces $X_t^x$ to (on average) move away from $\mcM_0$ if $X_t^x$ is close to $\mcM_0$. \Cref{as-U} guarantees that the variances of the mentioned Lyapunov functions evaluated at $X_t^x$ have a controlled growth in time.
     
Before proceeding, we precisely define what we mean if we write that one function is ``smaller" than another. 

\begin{deff}\label{lesssim}
    For an arbitrary set $A$ and functions $f,g: A \to [-\infty,\infty]$, we write $f \lesssim g$ if there is some $C > 0$ such that $f(x) \leq Cg(x)$ for all $x \in A$.
\end{deff}

\begin{deff}\label{<<}
    For an arbitrary set $A$ and functions $f,g: A \to [-\infty,\infty]$, we write $f << g$ if $\sup_{\{x \in A \mid |g(x)| \leq M\}} |f(x)| < \infty$ for all $M > 0$.
\end{deff}
\begin{rem}
    For example, if $g$ is proper (continuous with compact sublevel sets) and $f$ is continuous then $f << g$. Also, if $f \lesssim g^k$ for some $k > 0$, then $f << g$.
\end{rem}

\begin{ass}
        \label{as-W} There exist measurable functions $W: \mcM \to [1,\infty)$, $W': \mcM \to [1,\infty]$, and a constant $K > 0$ such that: \begin{enumerate}
            \item  $W \in \Dme_2(\mcM)$ (see \Cref{D2}).
            \item $\mathcal{L}W \leq K - W'$.
            \item $W << W'$ (see \Cref{<<}).
        \end{enumerate}
\end{ass}

\begin{deff}\label{fxn-vanish}
    For an arbitrary set $A$ and functions $f: A \to [-\infty,\infty]$ and $g:A \to [0,\infty]$, we say $f$ vanishes over $g$ if one of the following equivalent conditions holds: \begin{enumerate}
    
        \item $\forall \epsilon > 0, \exists M > 0$ such that $|f|\Id_{|f|>M} \leq \epsilon g$
        \item $\forall \epsilon > 0, \exists N > 0$ such that $|f| \leq \epsilon g + N$
    \end{enumerate}
\end{deff}
\begin{rem}
    Note that \Cref{fxn-vanish} is different than the definition given in \cite[Definition 2.18]{extinction}.
\end{rem}

We define a similar notion for collections of random variables:
\begin{deff}\label{rv-vanish}
    For collections of random variables $\{X_t\},\{Y_t\}$ defined on the same probability space and $Y_t > 0$ a.s., we say $\{X_t\}$ vanishes over $\{Y_t\}$ if one of the following equivalent conditions hold:
    \begin{enumerate}
        \item $\forall \epsilon > 0, \exists M > 0$ such that almost surely $|X_t|\Id_{|X_t|>M} \leq \epsilon Y_t$ for all $t$.
        \item $\forall \epsilon > 0, \exists N > 0$ such that almost surely $|X_t| \leq N + \epsilon Y_t$ for all $t$.
    \end{enumerate}
    
\end{deff}

\begin{deff}
\label{meas}
    Let $P(\mcM)$ denote the set of all Borel probability measures on $\mcM$. 
    For any $\mu \in P(\mcM)$, $S \subset \mcM$ measurable, and $f:S \to \mathbb{R}$ measurable, we define
$$
\mu f := \int_S f d\mu \,.
$$
    We endow $P(\mcM)$ with topology of weak convergence, that is, $\mu_n \to \mu$ if for all $f \in C_b(\mcM)$, $\mu_nf \to \mu f$. 
\end{deff}

\begin{deff} \label{inv-meas}
    For $F \subset \mcM$ an invariant set (see \Cref{inv-set}), let $P_{inv}(F) \subset P(\mcM)$ denote the set of invariant measures on $F$, that is, $\mu \in P_{inv}(F)$ if $\mu(F) = 1$ and $\mu \Pp_tf = \mu f$ for all $f \in C_b(\mcM)$ and $t \geq 0$.
\end{deff}
\begin{deff}
\label{occ-meas}
    For any $x \in \mcM$ and $t > 0$ let $\mu_t^x$ denote the empirical occupation measure 
    \begin{equation}\label{eq:docm}
    \mu_t^x(\omega) \coloneqq \frac{1}{t}\int_0^t \delta_{X_s^x(\omega)}ds,    
    \end{equation}
    where $\delta_y$ is the Dirac  measure concentrated at $y$, that is, $\delta_y f = f(y)$. Note that $\omega \in \Omega$, and therefore $\mu_t^x$ is a random measure which depends on $\omega$, but often this dependence is suppressed.
\end{deff}

\begin{deff}\label{x-y-def}

Let $\X$ be the set of all $f : \mcM \to [0,\infty]$ such that there exists $K > 0$ such that for all $x \in \mcM$,
        $$\limsup_{t \to \infty} \mu_t^x f \leq K \quad a.s.$$
Let $\Y$ be the set of all $g : \mcM \to [0,\infty]$ such that there exists $K > 0$ such that for all $x \in \mcM$,
        $$\limsup_{t \to \infty} \E[\mu_t^x g] \leq K$$
\end{deff}

\begin{ass}
        \label{as-V} There is a nonnegative continuous function $V \in \Dme_2(\inv)$ such that:
    \begin{enumerate}[label=(\roman*)]
        \item \label{as-alpha} $\Ll V$ extends to a continuous function $H: \mcM \to \R$ and there is a constant $\Lambda < 0$ such that $\sup_{\mu \in P_{inv}(\mcM_0)} \mu H \leq \Lambda < 0$.
        \item \label{as-vanish} There is a positive finite measure $\mu$ on $[-r,0]$ for some $r \geq 0$ and $f \in \X$ such that for every $x \in \mcM$, $\{H(X_t^x)\}_{t > r, x \in \inv}$ vanishes over the function $\{\int f(X_{t+s}^x)d\mu(s)\}_{t > r, x \in \inv}$ (see \Cref{rv-vanish}). 
    \end{enumerate}
\end{ass} 
\begin{rem}
    It is shown below in \Cref{H-in-L1} that under our assumptions, $H$ is $\mu$-integrable for all $\mu \in P_{inv}(\mcM)$, so the condition $\mu H \leq \Lambda$ in \Cref{as-V}\ref{as-alpha} is well defined.
\end{rem}
\begin{rem}
    \Cref{as-V}\ref{as-vanish} is a generalization of \cite[Assumption 4(ii)]{extinction} (see also \cite[Hypothesis 4(ii)]{persistence}), which corresponds to $\mu = \delta_0$. Indeed, it will be shown in \Cref{muW'-K} that $W' \in \X$. This generalization is necessary for treating models with delayed effects such as those analyzed in \cite{functional}.
\end{rem}

\Cref{as-V}\ref{as-alpha} is the most important assumption. It implies that $V$ decreases (on average) when $X_t$ is close to $\mcM_0$. Since $V$ is nonnegative, this implies that $X_t$ cannot spend too much time near $\mcM_0$.

The following \Cref{as-U} ensures that $M_t^V$ and $M_t^W$ (see \eqref{Mtf}) satisfy the strong law of large numbers for martingales, as stated in \Cref{strong-law}. In fact, \Cref{as-U} can be replaced by \Cref{strong-law}. However, in practice \Cref{as-U} often follows from \Cref{as-W}, and therefore, we assume \Cref{as-U} rather than \Cref{strong-law}.

\begin{ass}\label{as-U}
    Let $W$ and $V$ be respectively as in \Cref{as-W} and \Cref{as-V}. Assume that 
there are measurable functions $U: \mcM \to [0,\infty)$, $U': \mcM \to [0,\infty]$, a constant $K > 0$, and some $g \in \Y$ such that:
   \begin{enumerate}[label=(\roman*)]
       \item \label{5.1} $U \in \Dme_+(\mcM)$ (see \Cref{D+}).
       \item \label{5.2} $\mathcal{L}U \leq K - U'$.
       \item \label{5.3} $\Gamma W \leq g$.
       \item \label{5.4} $\Gamma V \leq g$.
   \end{enumerate}
\end{ass} 

\begin{lem}\label{u'-in-y}
    $U' \in \Y$. In particular, we may replace $g$ with $KU'$ in \Cref{as-U}.
\end{lem}
\begin{proof}
    See \cite[Lemma 4.1]{extinction}.
\end{proof}

A direct consequence of \Cref{as-U} is the following law of large numbers: 
\begin{lem}\label{strong-law}
Let $W,V,U$ be as in \Cref{as-W}, \Cref{as-V}, \Cref{as-U}.
For any $x \in \mcM$ we have 
\begin{equation}
    \lim_{t \to \infty} \frac{M_t^W(x)}{t} = 0 \qquad \textrm{a.s.}
\end{equation}
and for any $x \in \inv$ we have 
\begin{equation}
    \lim_{t \to \infty} \frac{M_t^V(x)}{t} = 0 \qquad \textrm{a.s.} \,,
\end{equation}
where $M^W_t$ and $M^V_t$ were defined in \eqref{Mtf}. 
\end{lem}
\begin{proof}
    See the arguments in \cite[Lemma 2.12]{extinction} and \cite[Corollary 4.2]{extinction}.
\end{proof}

\begin{cor}\label{muW'-K}
$W' \in \X$ (see \Cref{as-W}, \Cref{x-y-def}).
    
\end{cor}
\begin{proof}
    By \Cref{strong-law}, $\frac{-W(X_t^x)}{t} + \mu_t^x(\mathcal{L}W) \to 0$ a.s. as $t \to \infty$. Moreover, from $W \geq 0$ and $\mathcal{L}W \leq K - W'$ (see \Cref{as-W}) it follows that
    \begin{align*}
        0 &\leq \liminf_{t \to \infty} \frac{W(X_t^x)}{t} = \liminf_{t \to \infty} \mu_t^x(\mathcal{L}W) 
        \leq \liminf_{t \to \infty} K - \mu_t^x W'     \,.   
    \end{align*}

\end{proof}

If $W'$ has compact sublevel sets, as assumed in \cite{extinction} and \cite{persistence}, then \Cref{as-W} -- \ref{as-U} would be enough to guarantee stochastic persistence of $X_t$ (see \Cref{main} below). However, in most infinite-dimensional contexts such as stochastic functional equations arising from biological systems with delay (\Cref{sec:funct}) and stochastic PDEs (see \Cref{examples}), the compactness of sublevel sets cannot be obtained, and replace it by the following weaker assumption.

\begin{ass}
    \label{as-compact} There is some $f \in \X$ (\Cref{x-y-def}) such that for all $\epsilon > 0$ there exists $T > 0$ such that for all $M$ sufficiently large there exists a compact set $\K \subset \mcM$ which satisfies (regadless of initial condition)
    
    \begin{equation}\label{eq:as-compact}
        \Prb\Big(\mu_T(\K^c) > \epsilon \Big) \leq \Prb\Big(\int_0^Tf(X_t)dt + \sup_{t \leq T}W(X_t) > M \Big) + \epsilon \,.
    \end{equation}
    
\end{ass}

In other words, if $f(X^x_t)$ and $W(X^x_t)$ are small for a long time, then with high  probability, $X_t$ spent most of that time in the compact set $\K$.

\begin{rem}\label{compact-sublevel}
    If there exists some $f\in \X$ which has compact sublevel sets, then \Cref{as-compact} is satisfied. Indeed,  for any $M$ and any $T$, $\K = \{W' \leq M/(T\epsilon)\}$ is compact and on the event $\{\int_0^T W'(X_t) dt \leq M\}$ we have
    $$
    \mu_T(\K^c) = \frac{1}{T}\int_0^T \Id_{f(X_t) > M/(T\epsilon)}dt \leq \int_0^T \frac{\epsilon}{M}f(X_t)dt 
    \leq \epsilon\,.
    $$
    Hence, 
    \begin{equation}
      \Prb\Big(\mu_T(\K^c) > \epsilon \Big) \leq \Prb\Big(\int_0^Tf(X_t)dt  > M \Big)  
    \end{equation}
    and \eqref{eq:as-compact} follows from $W \geq 0$.
\end{rem}

Along with \Cref{muW'-K}, the following lemma is instrumental for proving the existence of the $f \in \X$ required for \Cref{as-V}\ref{as-vanish} and \Cref{as-compact}. It is proved below in \Cref{sec:proofs}.

\begin{lem}\label{fn-in-X}
    Let $f_0 \coloneqq (W')^{1/2}$ (see \Cref{as-W}) and $f_1,\dots,f_n$, $h_1,\dots,h_n$ be functions satisfying the following. For $i = 1,\dots,n$, some $T > 0$, and all $x \in \mcM$:
    \begin{itemize}
        \item $h_i$ vanishes over $f_i$.
        \item $\E\Big[\int_0^t f_i(X_t^x)^2dt\Big] < \infty$ for all $t \geq T$.
        \item
        \begin{equation}\label{eq-int-f-h-bounds}
            \begin{aligned}
                 \E\Big[\int_0^T f_i(X_t^x)dt\Big] &\leq h_i(x) + f_{i-1}(x) + \E\Big[\int_0^T f_{i-1}(X_t^x) + h_i(X_t^x)dt\Big] \\
                \E\Big[\int_0^T f_i(X_t^x)^2dt\Big] &\leq h_i(x) + f_{i-1}(x)^2 + \E\Big[\int_0^T f_{i-1}(X_t^x)^2 + h_i(X_t^x)^2dt\Big] \,.
            \end{aligned} 
        \end{equation}
    \end{itemize}
    Then $f_n \in \X$ and $f_n^2 \in \Y$.
\end{lem}
\begin{rem}
    It will clear from the proof of \Cref{fn-in-X} that if we had started with $f_0 = U'$ (see \Cref{as-U}) and assumed only the first inequality of \eqref{eq-int-f-h-bounds} then we could conclude that $f_n \in \Y$. However, this will not be needed for our examples.
\end{rem}

\section{Main General Results}\label{results}
\begin{thm} \label{main}
    Under our \Cref{as1}--\ref{as-compact}, the following statements hold: \begin{enumerate}[label=(\roman*)]
        \item For all $x \in \inv$, almost surely all limit points $\mu$ of $\mu_t^x$ as $t \to \infty$ satisfy $\mu \in P_{inv}(\inv)$.
        \item $X_t$ is stochastically persistent, meaning for all $\epsilon > 0$ there exists a compact set $\K_\epsilon \subset \inv$ such that for all $x \in \inv$ $$\Prb(\liminf_{t \to \infty} \mu_t^x(\K_\epsilon) \geq 1 - \epsilon) = 1 \,.$$
        \item $P_{inv}(\inv)$ is nonempty.
    \end{enumerate}
\end{thm}

\begin{rem}\label{rem:weaker-thm}
    We believe (iii) and a subsequential version of (i) would still hold under weaker assumptions. In particular, in \Cref{as-V}\ref{as-vanish} and \Cref{as-compact}, consider replacing the space $\X$ with $\Y$. For $g \in \Y$, Fatou's Lemma gives that $\liminf_{t \to \infty} \mu_t^xg < \infty$, and then all of our proofs should go through by working on a subsequence $t_n$ which attains the $\liminf$.
\end{rem}

The proof is based on two key lemmas which are proved in \Cref{sec:proofs} below. 

\begin{lem} \label{tightness}
    Under the assumptions of \Cref{main}, almost surely $(\mu_{t_n}^x)_{n \in \N}$ is tight for all $t_n \to \infty$.
\end{lem}

\begin{lem} \label{muH-converge}
    Under the assumptions of \Cref{main}, there is $A > 0$ such that, almost surely, $t_n \to \infty$ and $\mu_n \coloneqq \mu_{t_n}^x \to \mu \in P(M)$ implies $\mu H = \lim_{n \to \infty} \mu_n H$ and $\mu |H| \leq A$.
\end{lem}

We also have the following theorem, which is a version of \cite[Theorem 3.11]{extinction} in our setting. It states that the average Lyapunov exponent $\Lambda$ is (semi-)continuous in the parameters governing the process $X_t$.

\begin{thm}\label{robust}
    Suppose $\mcM$ (resp. $\mcM_0$) decomposes as $\Theta \times \tilde{\mcM}$ (resp. $\Theta \times \tilde{\mcM_0}$), where $\Theta$ is a compact metric space such that $\{\theta\} \times \tilde{\mcM}$ is an invariant set (\Cref{inv-set}) for every $\theta \in \Theta$.
    
    Fix some $\theta_0 \in \Theta$. Suppose \Cref{as1}--\ref{as-compact} hold, except \Cref{as-V}\ref{as-alpha} is replaced with: ``$\Ll V$ extends to a continuous function $H: \mcM \to \R$ and there is a constant $\Lambda < 0$ such that $\sup_{\mu \in P_{inv}(\{\theta_0\} \times \mcM_0)} \mu H \leq \Lambda < 0$." Then for all $\theta$ sufficiently close to $\theta_0$, the results of \Cref{main} hold on $\{\theta\} \times \mcM$. To be more specific, let $\{Y_t^x\}_{t \geq 0, x \in \tilde{\mcM}}$ be the Markov process on $\tilde{\mcM}$ such that $(\theta, Y_t^x) = X_t^{(\theta,x)}$ (this is well-defined because $\{\theta\} \times \tilde{\mcM}$ is an invariant set). Then we have:

    \begin{enumerate}[label=(\roman*)]
        \item For all $x \in \tinv \coloneqq \tilde{\mcM}_0^c$ we have that, almost surely, all limit points $\mu$ of $\tilde{\mu}_t^x \coloneqq (1/t)\int_0^t \delta_{Y_s^x(\omega)}ds$ as $t \to \infty$ satisfy $\mu(\tinv) = 1$.
        \item $Y_t$ is stochastically persistent, meaning for all $\epsilon > 0$ there exists a compact set $\K_\epsilon \subset \tinv$ such that for all $x \in \tinv$ $$\Prb(\liminf_{t \to \infty} \tilde{\mu}_t^x(\K_\epsilon) \geq 1 - \epsilon) = 1 \,.$$
        \item $P_{inv}(\{\theta\} \times \tinv)$ is nonempty.
    \end{enumerate}
\end{thm}

\begin{proof}
It is enough to show that $$\limsup_{\theta \to \theta_0} \sup_{ \mu \in P_{inv}(\{\theta\} \times \tilde{\mcM_0})} \mu H \leq \sup_{ \mu \in P_{inv}(\{\theta_0\} \times \tilde{\mcM_0})} \mu H \,.$$ The proof is exactly the same as the proof of \cite[Theorem 3.11]{extinction}, with \Cref{pinv-compact} and \Cref{H-in-L1} used in place of the first and second (respectively) applications of \cite[Lemma 2.22]{extinction}.
\end{proof}

\section{Proofs}\label{sec:proofs}

In this section, we assume that \Cref{as1}--\ref{as-compact} hold.

\subsection{Proof of \Cref{fn-in-X}}
Recall the set $\X$ defined in \Cref{x-y-def}. First we show some preliminary lemmas which will be important for the proofs below in \Cref{sec:proofs}.

\begin{lem}\label{time-equivalence}
    For any $T > 0$, $x \in \mcM$, and $f: \mcM \to [0,\infty]$ satisfying $\int_0^Tf(X_t^x)dt < \infty$ a.s., we have 
\begin{align*}
    \limsup_{t \to \infty} \mu_t^x f &
    = \limsup_{n \to \infty} \frac{1}{T}\int_0^T \frac{1}{nT} \int_0^{r + nT} f(X_s^x)dsdr \\
    &= 
    \limsup_{t \to \infty} \frac{1}{t}\int_0^t \frac{1}{T} \int_0^T f(X_{r+s}^x)drds 
     \,.
\end{align*}
  In addition, if   $\E[\int_0^Tf(X_t^x)dt] < \infty$, then
\begin{align*}  
    \limsup_{t \to \infty} \E[\mu_t^x f] &= 
    \limsup_{t \to \infty} \frac{1}{t}
    \E\Big[\int_0^t \frac{1}{T} \int_0^T f(X_{r+s}^x)drds\Big] \\
    &= \limsup_{n \to \infty} \frac{1}{T}\E\Big[ \int_0^T \frac{1}{nT} \int_0^{r + nT} f(X_s^x)dsdr\Big] \,.
\end{align*}

\end{lem}
\begin{proof}
    The claims follow from Tonelli's theorem. Indeed,
    \begin{align*}
        \frac{1}{T}\int_0^T \frac{1}{nT} \int_0^{r + nT} f(X_s^x)dsdr &= \frac{1}{nT^2} \int_0^{(n+1)T} \int_{(s-nT)\vee 0}^Tf(X_s^x)dr ds \\
        &\leq \frac{1}{nT} \int_0^{(n+1)T}f(X_s^x)ds = \frac{n+1}{n} \mu_{(n+1)T}^xf 
    \end{align*}
and by passing $n \to \infty$, we have 
    \begin{align*}
  \limsup_{n \to \infty} \frac{1}{T}\int_0^T \frac{1}{nT} \int_0^{r + nT} f(X_s^x)dsdr 
        \leq 
        \limsup_{n \to \infty} \frac{n+1}{n} \mu_{(n+1)T}^xf 
        \leq \limsup_{t \to \infty} \mu_{t}^xf \,.
    \end{align*}
On the other hand, let $(t_n)_{n = 1}^\infty$ be a sequence such that 
$$
\limsup_{n \to \infty} \mu_{t_n}^xf = \limsup_{t \to \infty}  \mu_{t}^xf .
$$ 
For each $n \geq 1$ let $k_n$ be such that 
$(k_n -1)T \leq t_n < k_n T$. 
Consequently, 
    \begin{align*}
        \frac{1}{k_n T^2} \int_0^{(k_n+1)T} \int_{(s-k_nT)\vee 0}^Tf(X_s^x)dr ds 
        &\geq \frac{1}{k_nT} \int_0^{k_nT}f(X_s^x)ds \\
        &\geq \frac{1}{k_nT} \int_0^{t_n}f(X_s^x)ds  
        =  \frac{t_n}{k_nT} \mu_{t_n}^xf
    \end{align*}
    and by passing $n \to \infty$ we obtain
\begin{multline*}
  \limsup_{n \to \infty} \frac{1}{T}\int_0^T \frac{1}{nT} \int_0^{r + nT} f(X_s^x)dsdr \\
  \geq 
 \limsup_{n \to \infty}  \frac{1}{k_n T^2} \int_0^{(k_n+1)T} \int_{(s-k_nT)\vee 0}^Tf(X_s^x)dr ds 
        \geq 
         \limsup_{t \to \infty} \mu_{t}^xf 
\end{multline*}
and the first equality follows. To prove the second equality, we 
use the substitution $z = s+t$ and Tonelli's theorem to obtain
\begin{align*}
\frac{1}{tT}
\int_0^t\int_0^T f(X_{r+s}^x)drds &= \frac{1}{tT} \int_0^T zf(X_z^x)dz + \frac{1}{t}\int_T^t f(X_z^x)dz
\\ &\qquad + \frac{1}{tT} \int_t^{t + T} (t+T-z)f(X_z^x)dz \,,
\end{align*}

Since $z \leq T$ in the first integral, then $f \geq 0$  implies
$$
0 \leq \frac{1}{tT} \int_0^T zf(X_z^x)dz \leq 
\frac{1}{t} \int_0^T f(X_z^x)dz \,. 
$$
Since $z \in [t, t+T]$ in the third integral, then $f \geq 0$ yields
\begin{align*}
  0 &\leq \frac{1}{tT} \int_t^{t + T} (t+T-z)f(X_z^x)dz 
 \leq
  \frac{1}{t} \int_t^{t + T} f(X_z^x)dz \,.
\end{align*}
Thus, 
\begin{align*}
  \frac{1}{t}\int_T^t f(X_z^x)dz \leq  \frac{1}{tT}
\int_0^t\int_0^T f(X_{r+s}^x)drds \leq 
\frac{1}{t}\int_0^{t+T} f(X_z^x)dz \,,
\end{align*}
or equivalently, 
\begin{align*}
 \mu_t^xf - \frac{1}{t}\int_0^T f(X_z^x)dz \leq  \frac{1}{tT}
\int_0^t\int_0^T f(X_{r+s}^x)drds \leq 
\frac{T+t}{t} \mu_t^xf 
\end{align*}
and the first assertion follows after passing $t \to \infty$ (in $\limsup$) and using that $\int_0^T f(X_t^x)dt < \infty$. 

The second statement follows analogously.

\end{proof}

\begin{cor}\label{int-f-smaller-than-M}
    For all $f \in \X$, $\epsilon > 0$, $T > 0$, there exists $M > 0$ such that for all $x \in \mcM$
    $$\limsup_{t \to \infty} \frac{1}{t}\int_0^t \Id_{\{\int_0^T f(X_{r+s}^x)dr > M\}}ds \leq \epsilon \quad a.s.$$
\end{cor}

\begin{proof}
    With $K$ as in \Cref{x-y-def}, by \Cref{time-equivalence},
    $$\limsup_{t \to \infty} \frac{1}{t}\int_0^t \int_0^T f(X_{r+s}^x)drds \leq KT \quad a.s.$$
    Since
    $$\Id_{\{\int_0^T f(X_{r}^x)dr > M\}} \leq \frac{1}{M} \int_0^T f(X_{r}^x)dr \,,$$
    the claim follows with $M = KT/\epsilon$.
\end{proof}

\begin{rem}
    Recall that  $\D([0,T],\mcM)$ is the Skorokhod space (see \cite[Lemma 5.2]{extinction} and the discussion preceding it).
\end{rem}

\begin{lem}\label{comparison-result}
    Fix $T > 0$ and assume $f,g: \D([0,T], \mcM) \to [0,\infty]$ satisfy for all $x \in \mcM$:
    \begin{gather}\label{eq-compare-exp}
        \E[f(\{X_s^x\}_{s \leq T})] \leq \E[g(\{X_s^x\}_{s \leq T})] \,,
    \\
    \label{eq-compare-var}
        \limsup_{t \to \infty} \E\Big[\frac{1}{t}\int_0^t f(\{X_r^x\}_{s\leq r \leq s+T})^2 + g(\{X_s^x\}_{s \leq r \leq s+T})^2ds\Big] < \infty \,.
    \end{gather}
    Then
    \begin{equation}\label{eq:comparison-result-1}
        \limsup_{t \to \infty} \frac{1}{t}\int_0^t f(\{X_r^x\}_{s\leq r \leq s+T})ds \leq \limsup_{t \to \infty} \frac{1}{t}\int_0^t g(\{X_r^x\}_{s\leq r \leq s+T})ds \,.
    \end{equation}

    Furthermore, if $f$ has the form $f(\phi) = \int_0^T \hat f(\phi(s))ds$ where $\hat f: \mcM \to [0,\infty]$, $g_2: \D([0,T],\mcM) \to [0,\infty)$ satisfies $g(\phi) \leq \frac{T}{2}\hat f(\phi(0)) + g_2(\phi)$, and \eqref{eq-compare-var} also holds with $g$ replaced by $g_2$, then
    $$\limsup_{t \to \infty} \frac{1}{t}\int_0^t \hat f(X_s^x)ds \leq \frac{2}{T}\limsup_{t \to \infty} \frac{1}{t}\int_0^t g_2(\{X_r^x\}_{s\leq r \leq s+T})ds \,.$$

\end{lem}

\begin{proof}
    First we introduce the following notation for $t \geq 0$, $t_0 \in [0,T]$, $k,n \in \N$, and any function $h: \D([0,T],\mcM) \to [0,\infty]$:
    \begin{align*}
        t_k &\coloneqq t_0 + kT \\
        A_t(h) &\coloneqq \int_0^t h(\{X_r^x\}_{s\leq r \leq s+T})ds \\
        \Delta_k^{t_0}(h) &\coloneqq A_{t_k} - A_{t_{k-1}} = \int_{t_{k-1}}^{t_k} h(\{X_r^x\}_{s\leq r \leq s+T})ds \\
        \G_k^{t_0} &\coloneqq \F_{t_k} \\
        M_n^{t_0}(h) &\coloneqq \sum_{k=1}^n \Delta_{2k-1}^{t_0} - \E[\Delta_{2k-1}^{t_0} |\G^{t_0}_{2k-2}] \\
        N_n^{t_0}(h) &\coloneqq \sum_{k=1}^n \Delta_{2k}^{t_0} - \E[\Delta_{2k}^{t_0} |\G^{t_0}_{2k-1}]
    \end{align*}

    By \Cref{time-equivalence},
    \begin{equation}\label{eq:limsup-A-equals-h}
        \limsup_{t \to \infty} \frac{1}{t}\int_0^t h(\{X_r^x\}_{s\leq r \leq s+T})ds = \limsup_{n \to \infty} \frac{1}{T}\int_0^T \frac{1}{2nT} A_{t_0 + 2nT}(h)dt_0 \,.
    \end{equation}

    Thus, our claim is equivalent to
    \begin{equation}\label{eq-quiv-claim-for-comparison-result}
        \limsup_{n \to \infty} \frac{1}{2n}\int_0^T A_{t_0 + 2nT}(f)dt_0 \leq \limsup_{n \to \infty} \frac{1}{2n}\int_0^T A_{t_0 + 2nT}(g)dt_0 \,.
    \end{equation}

     We write
    \begin{equation}\label{eq:break-up-A}
        A_{t_0 + 2nT} = A_{t_0} + \sum_{k=1}^{2n} \E[\Delta_k^{t_0} |\G^{t_0}_{k-1}] + M_n^{t_0} + N_n^{t_0}
    \end{equation}
     and bound each term separately. In what follows, we suppose $h \in \{f,g\}$.

      \textit{Step 1:}
      By \eqref{eq-compare-var}
     \begin{equation*}
         \E\Big[\int_0^T A_{t_0}(h) dt_0\Big] \leq T\E[A_T(h)] < \infty \,,
     \end{equation*}
     and so
     $$\lim_{n \to \infty} \frac{1}{2n}\int_0^T A_{t_0}(h)dt_0 = 0 \quad a.s.$$

    \textit{Step 2:} To control $M_n^{t_0} + N_n^{t_0}$, we first note that 
       \begin{align*}
         \E[|M_n^{t_0}|^2 + |N_n^{t_0}|^2] &= \sum_{k=1}^{2n} \E\Big[\Big|\Delta_k^{t_0} - \E[\Delta_k^{t_0} |\G^{t_0}_{k-1}]\Big|^2\Big] \\
         &\lesssim \sum_{k=1}^{2n} \E[|\Delta_k^{t_0}|^2] \\
         &\lesssim \E\Big[\int_{t_0}^{t_0 + 2nT} h(\{X_r^x\}_{s \leq r \leq s+T})^2ds\Big] \,,
     \end{align*}  
and therefore
     \begin{equation}\label{eq-disc-mg-var}
        \int_0^T \E[|M_n^{t_0}(h)|^2 + |N_n^{t_0}(h)|^2]dt_0 \lesssim \int_0^{(2n+1)T} \E[h(\{X_r^x\}_{t \leq r \leq t+T})^2]dt \,.
     \end{equation}
     By \eqref{eq-compare-var} and \eqref{eq-disc-mg-var}, for almost all $t_0$ the discrete-time process $M_n^{t_0}(h)$ is a $L^2$ martingale adapted to $\G_{2n}^{t_0}$ and $\limsup_{n \to \infty} \frac{1}{n}\int_0^T\E[|M_n^{t_0}|^2]dt_0 < \infty$. By Doob's inequality,
     $$\E\Big[\sup_{m \leq n}\Big(\int_0^T M_m^{t_0}dt_0\Big)^2\Big] \lesssim \E\Big[\int_0^T \sup_{m \leq n}|M_m^{t_0}|^2dt_0\Big] \lesssim \int_0^T \E[|M_n^{t_0}|^2]dt_0 \,.$$
     Then by a similar Borel-Cantelli argument as in \cite[Lemma 2.12]{extinction} we conclude
     $$\lim_{n \to \infty} \frac{1}{n}\int_0^T M^{t_0}_n(h)dt_0 = 0  \qquad \textrm{a.s.} \,,$$
     and the same exact argument works for $N_n^{t_0}$ as well.
     
     \textit{Step 3:} Finally, we note by \eqref{eq-compare-exp} and the Markov property that
     $$\E[\Delta_k^{t_0}(f)|\G_{k-1}^{t_0}] \leq \E[\Delta_k^{t_0}(g)|\G_{k-1}^{t_0}] \,,$$
     and so \eqref{eq-quiv-claim-for-comparison-result} follows from \textit{Step 1, Step 2,} and \eqref{eq:break-up-A}. Thus, we have shown \eqref{eq:comparison-result-1}.

     For the final claim, note that above we showed that
     \begin{equation}\label{eq:general-comparison}
         \int_0^{t_0+2nT} \int_s^{s+T}\hat f(X_r^x)drds= A_{t_0 + 2nT}(f) \leq A_{t_0 + 2nT}(g) + c(n,t_0) \,,
     \end{equation}
     where the first equality follows from the definition and $c(n,t_0)$ is a term satisfying $\lim_{n \to \infty} \frac{1}{2n}\int_0^T c(n,t_0)dt_0 = 0$. Then we notice that (by Tonelli's theorem)
     \begin{multline*}
         \int_0^T\int_0^{t_0+2nT} \int_s^{s+T}\hat f(X_r^x)drdsdt_0 = \int_0^T\int_0^{t_0+2nT+T} \int_{(r-T)\vee0}^{(t_0 + 2nT) \wedge r}\hat f(X_r^x)dsdrdt_0 \\
    \begin{aligned}
         &\geq T\int_0^T\int_T^{t_0+2nT} \hat f(X_r^x)drdt_0 \\
         &= T\int_0^T\int_0^{t_0+2nT} \hat f(X_r^x)drdt_0 - T^2\int_0^T\hat f(X_r^x)dr \,.
     \end{aligned}
     \end{multline*}
     Integrating \eqref{eq:general-comparison} from $0$ to $T$, applying the inequality above, using that $g(\phi) \leq \frac{T}{2}\hat f(\phi(0)) + g_2(\phi)$, and then subtracting $\frac{T}{2}\int_0^T\int_0^{t_0+2nT} \hat f(X_r^x)drdt_0$ from both sides gives
     \begin{multline*}
     \frac{T}{2}\int_0^T\int_0^{t_0+2nT} \hat f(X_r^x)drdt_0 - T^2\int_0^T \hat f(X_r^x)dr\\
     \leq \int_0^TA_{t_0 + 2nT}(g_2)dt_0 + \int_0^T c(n,t_0)dt_0 \,.
     \end{multline*}
     Dividing by $2nT^2$, taking the $\limsup$ as $n \to \infty$, and applying \eqref{eq:limsup-A-equals-h} gives
     $$\limsup_{n \to \infty} \frac{1}{4nT} \int_0^T\int_0^{t_0+2nT} \hat f(X_r^x)drdt_0 \leq         \limsup_{t \to \infty} \frac{1}{t}\int_0^t g_2(\{X_r^x\}_{s\leq r \leq s+T})ds \,.$$
     This finishes the proof by \Cref{comparison-result}.
\end{proof}

We conclude this section with the proof of \Cref{fn-in-X}:
\begin{proof}
    First note that
    $$\limsup_{t \to \infty} \E\Big[\frac{1}{t}\int_0^t f_0(X_t^x)^2 ds\Big] = \limsup_{t \to \infty} \E\Big[\frac{1}{t}\int_0^t W'(X_t^x) ds\Big] < \infty$$
    follows from \cite[Lemma 4.1]{extinction}, so $f_0^2 \in \Y$. Next,  we  inductively show that
    \begin{equation}\label{eq-var-f-finite}
        \limsup_{t \to \infty} \E\Big[\frac{1}{t}\int_0^t f_i(X_t^x)^2 ds\Big] < \infty \,,
    \end{equation}
    which proves $f_n^2 \in \Y$.
    Indeed, since $h_i$ vanishes over $f_i$, there is 
    $M > 0$ so that $h_i^2 \leq \epsilon f_i^2 + M$, where $\epsilon > 0$ will be chosen later. Plugging this into \eqref{eq-int-f-h-bounds}, we obtain
    $$(1-\epsilon)\E\Big[\int_0^T f_i(X_t^x)^2dt\Big] \leq \epsilon f_i(x)^2 + M(T+1) + f_{i-1}(x)^2 + \E\Big[\int_0^T f_{i-1}(X_t^x)^2dt\Big] \,.$$
    (We are justified with subtracting  $\epsilon \E\Big[\int_0^T f_i(X_t^x)^2dt\Big]$ from both sides because it is finite by assumption.) By Tonelli's theorem,
    $$\E\Big[\int_0^t\int_s^{s+T} f_i(X_r^x)^2drds\Big] \geq T\E\Big[\int_0^t f_i(X_r^x)^2dr\Big] - T\E\Big[\int_0^T f_i(X_r^x)^2dr\Big]$$
    (for $t \geq T$). By Markov property and the above we conclude that
    \begin{multline*}
    [(1-\epsilon)T - \epsilon]\E\Big[\int_0^t f_i(X_r^x)^2dr\Big] \\
    \leq (1-\epsilon)T\E\Big[\int_0^T f_i(X_r^x)^2dr\Big] + Mt(T+1) + (T+1)\E\Big[\int_0^{t+T} f_{i-1}(X_r^x)^2dr\Big] \,.
    \end{multline*}
    (Again, we are justified with the subtraction of $\epsilon \E\Big[\int_0^t f_i(X_r^x)^2dr\Big]$ from both sides by assumption.) \eqref{eq-var-f-finite} follows by induction, where $\epsilon$ is chosen such that $(1-\epsilon)T - \epsilon > 0$.

    The claim that $f_i \in \X$ is also true by induction. Indeed, the base step $f_0 \in \X$ holds by \Cref{muW'-K}. Since $h_i$ vanishes over $f_i$, by \eqref{eq-int-f-h-bounds} there is $M' > 0$ such that
    \begin{equation}\label{eq-the-equation-above}
        \E\Big[\int_0^T f_i(X_t^x)dt\Big] \leq 2f_{i-1}(x) +\frac{T}{2}f_i(x)+ 2\E\Big[\int_0^T f_{i-1}(X_t^x)dt\Big] + M' \,.
    \end{equation}
    Then the inductive step follows by \Cref{comparison-result} with $f(\phi) = \int_0^T f_i(\phi(t))dt$, $g_2(\phi) = 2f_{i-1}(\phi(0)) + 2\int_0^T f_{i-1}(\phi(t))dt + M'$, and $g(\phi) = \frac{T}{2}f_i(\phi(0)) + g_2(\phi)$ (\eqref{eq-compare-exp} is \eqref{eq-the-equation-above}, \eqref{eq-compare-var} holds by \eqref{eq-var-f-finite} and \Cref{time-equivalence}, and $\limsup_{t \to \infty} \frac{1}{t}\int_0^t g_2(\{X_r^x\}_{s\leq r \leq s+T})ds$ is bounded by \Cref{comparison-result} and our inductive hypothesis).
\end{proof}

\subsection{Proof of \Cref{tightness}}
Next we focus on using \Cref{as-compact} to show that the empirical occupation measures are tight almost surely.

\begin{lem} \label{W-small-M}
For all $m, T > 0$
 \begin{equation*}
        \lim_{M \to \infty} \inf_{W(x) \leq m} \Prb\Big(\sup_{t \leq T}W(X_t^x) \leq M \Big) = 1 \,.
    \end{equation*}
\end{lem}

\begin{proof}
    Given $x \in \{W \leq m\}$, define a martingale (see \Cref{D2})
    $$
    M_t \coloneqq W(X_t^x) - W(x) - \int_0^t \Ll W(X_s^x)ds \,.
    $$
     By \Cref{as-W}, 
     $$
     W(X_t^x) = W(x) + \int_0^t \Ll W(X_s^x)ds + M_t \leq m + Kt + M_t \,,
     $$ 
     so that for any $M > 0$, 
     $$
     \Prb\Big(\sup_{t \leq T}W(X_t^x) \leq M \Big) \geq \Prb\Big(\sup_{t \leq T} M_t \leq M - KT - m\Big) \,.
     $$

    Since $W \geq 0$,  $M_t$ is bounded below by $-m - KT$ for $t \leq T$. If we define $\tau \coloneqq \inf\{t \geq 0 \mid M_t > A\}$, then 
     any $A \geq 0$ we have  
    $$
    0 = \E[M_{\tau \wedge T}] \geq A\Prb(\tau \leq T) + (-m - KT)\Prb(\tau > T) \,.
    $$ 
     Rearranging the inequality and using $\Prb(\tau \leq T) = 1 - \Prb(\tau > T)$ gives $$\Prb(\tau > T) \geq \frac{A}{A + m + KT} \,.$$
If $M$ is sufficiently large, by 
    setting $A = M - KT - m > 0$ and noting that $$\Prb\Big(\sup_{t \leq T} M_t \leq M - KT - m\Big) \geq \Prb(\tau > T) \,,$$ we obtain that $$\Prb\Big(\sup_{t \leq T}W(X_t^x) \leq M \Big) \geq \frac{M - KT - m}{M} \to 1 \text{ as } M \to \infty \,.$$ Since the above bounds were independent of $x \in \{W \leq m\}$, the claim follows.
    
\end{proof}

\begin{cor} \label{W-small-m}
    For all $\epsilon > 0$, there exists $m > 0$ such that for any $x \in \mcM$, $$\limsup_{t \to \infty} \mu_t^x(\{W > m\}) \leq \epsilon \quad a.s.$$
\end{cor}

\begin{proof}
    By \Cref{muW'-K},
    $$
    \limsup_{t \to \infty} \mu_t^x W' \leq K \quad a.s.
    $$
It follows from Chebyshev inequality that $\mu_t^x(\{W' > K/\epsilon\}) \leq (\epsilon / K) \mu_t^x W'$, and therefore  
$$
\limsup_{t \to \infty} \mu_t^x(\{W' > K/\epsilon\}) \leq \epsilon \quad a.s.
$$
    Since $W << W'$ (see \Cref{as-W} and \Cref{<<}), $$m \coloneqq \sup_{W'(y) \leq K/\epsilon} W(y) < \infty \,,$$ and then the claim follows from $$\{W > m\} \subset \{W' > K/\epsilon\} \,.$$
\end{proof}

\begin{cor}\label{W-small-m-2}
    For all $T > 0$, $\epsilon > 0$ there exists $M > 0$ such that for all $x \in \mcM$
    $$\limsup_{t \to \infty} \frac{1}{t}\int_0^t \Id_{\{\sup_{s \leq r \leq s+T} W(X_r^x) > M\}}ds \leq \epsilon \quad a.s.$$
\end{cor}

\begin{proof}
Fix $x \in \mcM$. 
    By \Cref{W-small-m}, there is $m > 0$ such that
    $$\limsup_{t \to \infty} \mu_t^x(\{W > m\}) \leq \frac{\epsilon}{2} \quad a.s.$$
    By \Cref{W-small-M} there is $M > 0$ such that (for all $x \in \mcM$)
    $$\E\Big[\Id_{\{\sup_{t \leq T} W(X_t^x) > M\}}\Big] \leq \E\Big[\frac{\epsilon}{2} + \Id_{\{W(X_0^x) > m\}}\Big] \,.$$
    The claim follows from \Cref{comparison-result} 
    with $f((X_t)) = \Id_{\{\sup_{t \leq T} W(X_t^x) > M\}} \in [0, 1]$ and $g((X_t)) = \frac{\epsilon}{2} + \Id_{\{W(X_0^x) > m\}} 
    \in [\frac{\epsilon}{2}, \frac{\epsilon}{2} + 1]$, hence 
      \eqref{eq-compare-var} clearly holds. 
\end{proof}

\begin{lem} \label{occ-is-pretty-tight}
    For all $\epsilon > 0$ there exists a compact set $\K \subset \mcM$ such that for any $x \in \mcM$ $$\liminf_{t \to \infty} \mu_t^x(\K) \geq 1 - 4\epsilon \quad a.s.$$ In particular, \Cref{tightness} holds.
\end{lem}

\begin{proof}
Fix $\epsilon > 0$. 
Let $f,T$ be as in \Cref{as-compact} and choose $M$ large enough such that
\begin{equation}\label{eq-with-two-ids}
     \begin{aligned}
     \limsup_{t \to \infty} \frac{1}{t}\int_0^t \Id_{\{\int_0^T f(X_{r+s}^x)dr > M/2\}}ds &\leq \epsilon \quad a.s. \\
     \limsup_{t \to \infty} \frac{1}{t}\int_0^t \Id_{\{\sup_{s \leq r \leq s+T} W(X_r^x) > M/2\}}ds &\leq \epsilon \quad a.s.
 \end{aligned}
\end{equation}
which is possible by \Cref{int-f-smaller-than-M} and \Cref{W-small-m-2}. By \Cref{as-compact}, there exists a compact set $\K \subset \mcM$ such that

  \begin{align*}
      \E\Big[\frac{1}{T}\int_0^T \Id_{\K^c}(X_t^x)dt\Big] 
      &\leq \epsilon + \Prb \Big[\frac{1}{T}\int_0^T \Id_{\K^c}(X_t^x)dt > \epsilon \Big] \\
      &\leq \E\Big[\epsilon + \Id_{\{\int_0^Tf(X_t)dt + \sup_{t \leq T}W(X_t) > M\}} + \epsilon\Big] \\
      &\leq \E\Big[\Id_{\{\int_0^T f(X_{r}^x)dr > M/2\}} + \Id_{\{\sup_{r \leq T} W(X_r^x) > M/2\}} + 2\epsilon\Big]
  \end{align*}
  By \Cref{comparison-result} and \eqref{eq-with-two-ids},
  $$\limsup_{t \to \infty} \frac{1}{t}\int_0^t \frac{1}{T} \int_0^T \Id_\K^c(X_{r+s})drds \leq 4\epsilon \,,$$
  which by \Cref{time-equivalence} proves the claim.

\end{proof}

We finish this subsection by using \Cref{occ-is-pretty-tight} to show that the sets of invariant measures on $\mcM$ and $\mcM_0$ are nonempty and compact. 

\begin{lem} \label{lim-is-inv}
   If $x \in \mcM$ (respectively $x \in \mcM_0$), then almost surely the following holds: for every sequence $t_n \to \infty$ such that $\mu_{t_n}^x(\omega)$ converges to some $\mu \in P(\mcM)$, then $\mu \in P_{inv}(\mcM)$ (respectively $\mu \in P_{inv}(\mcM_0)$).\end{lem}
   
   \begin{proof}
   The claim for $\mcM$ is proven  in \cite[Theorem 2.2ii]{persistence}.

For any $x \in \mcM_0$, additionally note that the Portmanteau theorem and $\mcM_0$ being invariant imply that if $\mu_{t_n}^x \to \mu$, then $\mu(\mcM_0) \geq \limsup_{n \to \infty} \mu_{t_n}^x(\mcM_0) = 1$ a.s.
\end{proof}

Recall Birkhoff's ergodic theorem:

\begin{thm}\label{Birk}
If $\mu \in P_{inv}(\mcM)$ is ergodic then for $\mu$-a.e. $x \in \mcM$ $$\mu_t^x \to \mu \quad a.s.$$
\end{thm}

\begin{cor}\label{pinv-compact}
    The sets $P_{inv}(\mcM)$ and $P_{inv}(\mcM_0)$ are nonempty and compact.
\end{cor}
\begin{proof}
First, note that $P_{inv}(\mcM)$ is closed in the set of all probability measures on $\mcM$ with the topology of weak convergence because $X_t$ is Feller.

    By \Cref{occ-is-pretty-tight}, there is a sequence of compact sets $\K_n$ such that for any $x \in \mcM$ $$\liminf_{t \to \infty} \mu_t^x(\K_n) \geq 1 - \frac{1}{n} \quad a.s.$$
    If $\mu = \lim_{t \to \infty} \mu_t^x$, then by Portmanteau theorem $\mu(\K_n) \geq \limsup_{t \to \infty} \mu_t^x(\K_n)$. Then we obtain tightness of $P_{inv}(\mcM)$ from \Cref{Birk} and ergodic decomposition, and compactness from Prokhorov theorem. Also, $P_{inv}(\mcM_0)$ is closed in $P_{inv}(\mcM)$ by Portmanteau theorem, because $\mcM_0$ is closed.
    
    The nonemptiness is a consequence of \Cref{tightness} and \Cref{lim-is-inv}.    \end{proof}

\subsection{Proof of \Cref{muH-converge}} Recall that our goal is to use \Cref{as-V}\ref{as-vanish} to show $\mu_nH \to \mu H$ if $t_n \to \infty$ and $\mu_n \coloneqq \mu_{t_n}^x \to \mu$.

\begin{proof}
For $M > 0$, let $H_M \coloneqq ((H \wedge M) \vee (-M))$. Note that $\mu_n H_M \to \mu H_M$ as $n \to \infty$ and $|H - H_M| \leq |H|\Id_{|H| > M}$.
Thus, to prove $\mu_nH \to \mu H$ it suffices to show that almost surely there is $T > 0$ such that \begin{equation}\label{eq:uniform-convergence}
    \lim_{M \to \infty} \sup_{t \geq T} \mu_t^x |H|\Id_{|H| > M} = 0 \,.
\end{equation}

Using the first condition in \Cref{rv-vanish} and \Cref{as-V}\ref{as-vanish}, there is is a finite positive measure $\mu$ on $[-r,0]$ and $f \in \X$ such that 
 $\forall \epsilon > 0, \exists M > 0$ such that almost surely $$|H(X_t^x)|\Id_{|H(X_t^x)|>M} \leq \epsilon \int_{-r}^0  f(X^x_{t+s})d \mu(s)$$ 
for all $t > r$. 

 Setting $C =  \mu([-r,0])$ we have by Fubini-Tonelli theorem and $f \geq 0$ that 
\begin{equation*}
        \int_r^t |H(X_s^x)|\Id_{|H(X_s^x)|>M}ds \leq \epsilon C \int_0^{t}  f(X_s^x)ds = \epsilon Ct \mu_{t}^x  f\,.
\end{equation*}

Let $C'(M) = \int_0^r |H(X_s)|\Id_{|H(X_s)|>M}ds$, so that for $t > r$ $$\mu_t^x |H|\Id_{|H| > M} = \frac{1}{t}C'(M) + \frac{1}{t}\int_r^t |H(X_s)|\Id_{|H(X_s)|>M}ds \leq \frac{C'(M)}{t} + \epsilon C \mu_{t}^x f \,.$$

Letting $K$ be as in \Cref{x-y-def}, almost surely there is  $T(\omega) =: T > r$, independent of $M$,  so that $$\frac{t+r}{t}\mu_{t+r}^x f \leq 2K$$ for all $t \geq T$. Putting these inequalities together yields 
\begin{equation} \label{eq:depse}
  \sup_{t \geq T} \mu_t^x |H|\Id_{|H| > M} \leq \frac{C'(M)}{T} + 2\epsilon KC \,.  
\end{equation}
By the continuity of $H$, $\int_0^r |H(X_s^x)|ds < \infty$, 
 and the fact that $X_t^x$ has cadlag paths, the dominated convergence theorem implies 
 $C'(M) \to 0$ as  $M \to \infty$) and \eqref{eq:uniform-convergence} follows.

By the same argument as above, $\mu_n |H| \to \mu |H|$, and thus $$\mu |H| \leq M + \liminf_{n \to \infty} \mu_n |H|\Id_{|H| > M} \leq M + \liminf_{n \to \infty} \frac{C'(M)}{t_n} + 2\epsilon KC \,.$$ The claim follows with $A = M + 2\epsilon KC$.
\end{proof}

\begin{cor}\label{H-in-L1}
For any $\mu \in P_{inv}(\mcM)$ we have $\mu |H| \leq A$, where $A$ is as in \Cref{muH-converge}.  In addition, $\mu \mapsto \mu H$ is a continuous 
function on $P_{inv}(\mcM)$ and thus $\sup_{\mu \in P_{inv}
(\mcM_0)}\mu H < 0$ is equivalent to $\mu H < 0$ for all ergodic 
$\mu \in P_{inv}(\mcM_0)$.
\end{cor}
\begin{proof}
   The claim $\mu |H| \leq A$ follows from \Cref{muH-converge} and \Cref{Birk}. From \eqref{eq:depse} and \Cref{Birk}, we obtain that for all $\epsilon > 0$ there is $M > 0$ such that $\mu |H|\Id_{|H| > M} \leq 2\epsilon KC$ for all $\mu \in P_{inv}(\mcM)$. Thus, $\mu \mapsto \mu H$ is a uniform limit of the continuous functions $\mu \mapsto \mu ((H\vee(-M)) \wedge M)$ and so it is continuous as well. The last claim follows from the compactness of $P_{inv}(\mcM)$ (see \Cref{pinv-compact}). \end{proof}
\subsection{Proof of \Cref{main}}
First, we verify that \cite[Proposition 4.6]{persistence} holds in our setting:
\begin{lem}\label{muH-neg}
     \begin{itemize}
         \item For all $\mu \in P_{inv}(\mcM)$, we have $\mu H \leq 0$. Also, $\mu H = 0$ if and only if  $\mu \in P_{inv}(\inv)$.
         \item Since $P_{inv}(\inv)$ is compact, for all $\epsilon > 0$ there is a compact set $\K \subset \inv$ such that $\mu(\K) \geq 1 - \epsilon$ for all $\mu \in P_{inv}(\inv)$.
     \end{itemize}
\end{lem}

\begin{proof}
    As in the proof of \cite[Proposition 4.6i]{persistence}, we note that if $\mu \in P_{inv}(\inv)$ is ergodic then, by \Cref{strong-law} and \Cref{Birk}, for $\mu$-a.e. $x$ almost surely 
    \begin{equation}\label{eq:ervt}
        \lim_{t \to \infty} \frac{V(X_t^x)}{t} = \lim_{t \to \infty} \mu_t^x H = \mu H \,.
    \end{equation}
    Since $\inv = \bigcup_{M > 0} \{V \leq M\}$, there is some $M > 0$ for which $\mu(\{V \leq M\}) \geq \frac{1}{2}$.

Since $\mu$ is ergodic,     
 by \Cref{Birk},  $X_t$  enters $\{V \leq M\}$ infinitely often, and therefore $\lim_{t \to \infty} V(X_t^x)/t = 0$ for $\mu$ almost every $x$. Then from \eqref{eq:ervt} follows that 
 $\mu H = 0$ for each ergodic $\mu \in P_{inv}(\inv)$. The rest of the proof follows as in \cite[Proposition 4.6i]{persistence} (express each invariant measure as a convex combination of ergodic invariant measures $\mu$ supported on $\inv$ and $\mcM_0$ and notice that by \Cref{as-V}\ref{as-alpha} the ones supported on $\mcM_0$ satisfy $\mu H \leq \Lambda < 0$).

 To show that $P_{inv}(\inv)$ is compact, we use the first part of the lemma which implies $P_{inv}(\inv) = \{\mu \in P_{inv}(\mcM) \mid \mu H = 0\}$. Since $P_{inv}(\mcM)$ is compact (\Cref{pinv-compact}) and $\mu \mapsto \mu H$ is a continuous function on $P_{inv}(\mcM)$ (\Cref{H-in-L1}), then $P_{inv}(\inv)$ is a closed subset of a compact set, and thus it is compact. The tightness of $P_{inv}(\inv)$ in $P(\inv)$ follows from Prokhorov's theorem. 
 
 Note that by Portmanteau theorem the subspace topology of $P(\inv)$ inherited from $P(\mcM)$ is the same as the topology of weak convergence in $P(\inv)$. Indeed, $\mu_n \to \mu$ in $P(\inv)$ if and only if $\liminf_{n \to \infty} \mu_n(\U) \geq \mu(\U)$ for all sets $\U$ which are open in $\inv$. In addition, $\U$ being open in $\inv$ is the same as $\U$ being open in $\mcM$ because $\inv$ is open in $\mcM$. Thus, if $\mu_n, \mu \in P(\inv)$ and $\mu_n \to \mu$ in $P(\mcM)$ then $\mu_n \to \mu$ in $P(\inv)$. Conversely, if $\mu_n \to \mu$ in $P(\inv)$ then for all open $\U \subset \mcM$ we have $\liminf_{n \to \infty} \mu_n(\U) = \liminf_{n \to \infty} \mu_n(\U \cap \inv) \geq \mu(\U \cap \inv) = \mu(\U)$.

\end{proof}

\begin{rem}
We could not directly apply the proof given in \cite[Proposition 4.6ii]{persistence} because it uses a sequence of compact sets $K_m \subset \inv$ such that $M_0 = \cap \overline{K_m^c}$, which is equivalent to $\inv = \cup K_m^\circ$. Without local compactness, we cannot construct such a sequence. For the same reason, we must give a slightly different proof of \cite[Theorem 4.4ii]{persistence} below.
\end{rem}

Next, we prove \Cref{main}:
\begin{proof}
    The proof of (i) is exactly the same as the proof of \cite[Theorem 4.4i]{persistence}, and (iii) follows from (ii) and \Cref{lim-is-inv}.

    To prove (ii), we define open neighborhoods $\mcM_\delta \coloneqq \{x \in \mcM \mid d(x,\mcM_0) < \delta\}$. We claim  that for all $\epsilon > 0$ there exists $\delta_\epsilon > 0$ such that  for all $x \in \inv$
    $$\Prb(\limsup_{t \to \infty} \mu_t^x(\mcM_{\delta_\epsilon}) \leq \epsilon) = 1 \,.$$
For a contradiction, suppose that there is some $\epsilon > 0$ and $x_n \in \inv$ such that for each $n \geq 1$
$$\Prb(\limsup_{t \to \infty} \mu_t^{x_n}(\overline{\mcM_{1/n}}) > \epsilon) \geq \Prb(\limsup_{t \to \infty} \mu_t^{x_n}(\mcM_{1/n}) > \epsilon) > 0 \,.$$
By \Cref{tightness} and  \Cref{main} part (i), almost surely for each $n$ there is $(t_m^n)_{m=1}^\infty$ and $\mu_n \in P_{inv}(\inv)$ with $t^n_m \to \infty$ and $\mu_{t_m^n}^{x_n} \to \mu_n$ as $m \to \infty$, and 
$$
\lim_{m \to \infty} \mu_{t_m^n}^{x_n}(\overline{\mcM_{1/n}}) =\limsup_{t \to \infty} \mu_t^{x_n}(\overline{\mcM_{1/n}}) \,.
$$ 
It follows from Portmanteau theorem that 
$\mu_n(\overline{\mcM_{1/n}}) \geq \epsilon$ on the positive probability event $\{\limsup_{t \to \infty} \mu_t^{x_n}(\overline{\mcM_{1/n}}) > \epsilon\}$. Thus, we conclude that there exist $\mu_n \in P_{inv}(\inv)$ with $\mu_n(\overline{\mcM_{1/n}}) \geq \epsilon$. By compactness of $P_{inv}(\inv)$ (\Cref{muH-neg}), there is a convergent subsequence $\mu_{n_k} \to \mu \in P_{inv}(\inv)$. Since $\overline{\mcM_{1/n}}$ are nested, by Portmanteau theorem we have $\mu(\overline{\mcM_{1/n}}) \geq \epsilon$ for all $n$, and thus $\mu(\mcM_0) \geq \epsilon$, contradicting $\mu(\inv) = 1$.

Finally, we use \Cref{occ-is-pretty-tight}
 to conclude that for all $\epsilon > 0$ there exists $\K \subset \mcM$ compact such that  for all $x \in \mcM$
    $$\Prb(\liminf_{t \to \infty} \mu_t^x(\K) \geq 1 - \epsilon) = 1 \,.$$
Thus, the second claim of \Cref{main} follows by setting $\K_{2\epsilon} = \K \cap \mcM_{\delta_\epsilon}^c$.
 \end{proof}

 \begin{rem}\label{K-contained-in-strictly-positive}
     Suppose there is an open set $\U \subset \inv$ for which $\mu(\U) = 1$ for all $\mu \in P_{inv}(\inv)$. Then repeating the argument above with $\mcM_0$ replaced with $\U^c$ shows that the compact sets $\K_\epsilon$ in \Cref{main} may be assumed to be subsets of $\U$, because for all $\epsilon > 0$ there exists a $\delta_\epsilon > 0$ such that for all $x \in \inv$ $$\Prb(\limsup_{t \to \infty} \mu_t^x(\{d(\cdot, \U^c) < \delta_\epsilon\}) \leq \epsilon) = 1 \,.$$
 \end{rem}

\section{Application 1: Stochastic Functional Kolmogorov Equations}\label{sec:funct}

The goal of this section is to show that \Cref{as1}--\ref{as-compact}  are satisfied for stochastic functional Kolmogorov equations considered in \cite{functional}. We assume that our Markov process $(\Phi)_{t\geq 0}$ evolves in $\mcM \coloneqq C([-r,0], [0,\infty)^n)$, where $r, n > 0$ are fixed and $C(X,Y)$ denotes the space of continuous functions from $X$ to $Y$ equipped with the supremum norm:
\begin{equation}
    \|\phi\| \coloneqq \sup_{s \in [-r,0]} |\phi(s)| \qquad 
    \textrm{for any} \, \phi \in \mcM \,,
\end{equation}
where $|\cdot|$ denotes the standard Euclidean norm on $\R^n$.

To differentiate between functions in $\mcM$ and points in $[0,\infty)^n$, we use $x$ to denote $\phi(0)$ for $\phi \in \mcM$
and in general $X^\phi(t) \coloneqq \Phi_t^\phi(0)$ is 
the state of the process at time $t \geq -r$. Note that 
the process $\Phi_t^\phi$ on $\mcM$ contains the information 
on $X^\phi(x)$ on the past interval $s \in [t-r, r]$. Consequently, $\Phi_t^\phi(s) = X^\phi(t+s)$ for any $t \geq 0$ and $s \in [-r, 0]$. Also, denote $\pi_i: [0,\infty)^n \to [0,\infty)$ the projection on the $i$th coordinate and define $X_i^\phi(t) \coloneqq \pi_i(\Phi_t^\phi(0))$.
 We assume that the process $\Phi_t^\phi$ satisfies the following system of stochastic functional differential equations with initial condition 
 $\Phi_0 = \phi$: 
 \begin{equation}\label{eq:func-eq}
    dX_i(t) = X_i(t)f_i(\Phi_t)dt + X_i(t)g_i(\Phi_t)dE_i(t) \qquad 
    i = 1, \cdots, n \,,
\end{equation} 
where $f_i,g_i: \mcM \to \R$ are continuous functions which are Lipschitz on the bounded sets $\{\|\cdot\| \leq M\}$ for all $M > 0$ and $E_i$ is a of independent Brownian motions with covariance $\E[E_i(1)E_j(1)] = \sigma_{ij}$. The process $X_i(t)$ can be thought of as the population of some species $i$ at time $t$, and its evolution depends on the history of the ecosystem starting from time $t-r$. 

We consider that extinction occurred if 
 at least one species vanish, that is, if the process reaches the extinction set:
$$\mcM_0 \coloneqq \{\phi \in \mcM \mid \phi(0)_i = 0 \text{ for some } i \leq n\} \,.$$ 
The coexistence set is then
$$\inv \coloneqq \mcM_0^c = \{\phi \in \mcM \mid \phi(0)_i > 0 \text{ for all } i \leq n\} \,.
$$
The significance of \eqref{eq:func-eq} and the applicability of 
\Cref{fass1}--\ref{fass3}, formulated below, are discussed thoroughly in \cite{functional} in the context of Lotka-Volterra competitive or predator-prey models, stochastic replicator equations in evolutionary game theory, or SIR epidemic models, all with a delay.

In order to verify our general assumptions \Cref{as-W}, \Cref{as-U}, and \Cref{as-compact}, we use the following condition formulated in \cite[Assumption 2.1 (3)]{functional}.

\begin{fass}\label{fass1}
There exist $c_i, \gamma_b, \gamma_0, A_0, M > 0$, $A_1 > A_2 > 0$, a continuous function $h: \R^n \to [1,\infty)$, and a probability measure $\mu$ on $[-r,0]$ such that for any $\phi \in \mcM$, \begin{equation*}
    F(\phi) - \frac{1}{2}G(\phi) + \gamma_b J(\phi) \leq A_0\Id_{\{|x| < M\}} - \gamma_0 - A_1h(x) + A_2 \int_{-r}^0 h(\phi(s))d\mu(s) \,,
\end{equation*}
where $x \coloneqq \phi(0)$ and
\begin{align*}
    F(\phi) &\coloneqq \frac{\sum_{i=1}^nc_ix_if_i(\phi)}{1 + \sum_{i=1}^nc_ix_i} \\
    G(\phi) &\coloneqq \frac{\sum_{i,j = 1}^n \sigma_{ij}c_ic_jx_ix_jg_i(\phi)g_j(\phi)}{(1 + \sum_{i=1}^nc_ix_i)^2} \\
    J(\phi) &\coloneqq \sum_{i=1}^n |f_i(\phi)| + g_i^2(\phi) \,.
\end{align*}
\end{fass}

Additionally, the following assumption originating in \cite[Assumption 2.2]{functional} helps us verify \Cref{as-V} \ref{as-vanish}. 

\begin{fass}\label{fass2}
    One of the following holds for $J$ as in \Cref{fass1}:
    \begin{enumerate}[label=(\roman*)]
        \item If $h$ is as in \Cref{fass1}, then for any $\phi \in \mcM$ we have 
        \begin{equation}\label{eq:fass-1}
            J(\phi) \lesssim h(x) + \int_{-r}^0 h(\phi(s))d\mu(s)\,.
        \end{equation}
        \item There exists a function $h_1: \R^n \to [1,\infty)$ and a probability measure $\mu_1$ on $[-r,0]$ such that 
        for any $\phi \in \mcM$
        \begin{equation}\label{eq:fass-2}
            h_1(x) \lesssim J(\phi) \lesssim h_1(x) + \int_{-r}^0 h_1(\phi(s))d\mu_1(s)\,.
        \end{equation}
    \end{enumerate}
    Here, $x := \phi(0)$ and $\lesssim$ is as in \Cref{lesssim}.
    
\end{fass}

Finally, to verify \Cref{as-V}\ref{as-alpha} we need the following condition from \cite[Assumption 2.3]{functional}.
\begin{fass}\label{fass3}
    For any $\mu \in P_{inv}(\mcM_0)$,
    $$\max_{i \leq n} \lambda_i(\mu) > 0 \,,$$ where
    $$\lambda_i(\mu) \coloneqq \mu\Big(f_i - \frac{\sigma_{ii}g_i^2}{2}\Big) \,.$$
\end{fass}

It follows from standard theory (see the arguments in \cite{functional}) that, under our (local) Lipschitz assumptions on $f_i, g_i$ stated above and \Cref{as-W}, the problem \eqref{eq:func-eq} has a unique global solution $\Phi_t^\phi$ for each initial condition $\phi \in \mcM$, the Markov process $\Phi_t^\phi$ is ($C_b$-)Feller, and $\mcM_0, \inv$ are invariant sets. In particular, \Cref{as1}--\ref{as2} hold.

We aim to use \Cref{main} to prove the following theorem.

\begin{thm}\label{functional-thm}
    Under \Cref{fass1}--\ref{fass3}, $(X_t)_{t\geq 0}$ is stochastically persistent in the sense that for all $\delta > 0$ there are $0 < b < B < \infty$ such that for any nonnegative initial condition $\phi \in \inv$, 
    $$
    \liminf_{T \to \infty} \frac{1}{T}m(\{t \leq T \mid b \leq \min_{i \leq n} (X_i)_t \leq \max_{i \leq n} (X_i)_t \leq B\}) \geq 1 - \delta
    $$ 
    almost surely, where $m=m_T$ denotes the Lebesgue measure on $[0, T]$.
\end{thm}

In other words, for most of the time (a proportion of $1 - \delta$) the ecosystem supports at least a small portion ($b > 0$) of every species.

Following \cite{functional}, 
for any $\varphi \in \mcM$ and $s \geq 0$ we define $\varphi \in \mcM$ as
$\varphi_s(t) = \varphi((s+t)\wedge 0)$ for any $t \in [-r, 0]$. 
Then
for any (sufficiently smooth) function $V : \mcM \to \R$ and $\phi \in \mcM$ we define 
\begin{equation}
    \partial_t V (\varphi) = \lim_{s \to 0} \frac{V(\varphi_s) - V(\varphi)}{s}
\end{equation}
and for any $i = 1, \cdots, n$, we define
$\partial_i V (\varphi)(t) = 0$  for $t \neq 0$ and 
\begin{equation}
    \partial_i V (\varphi)(0) = \lim_{s \to 0} \frac{V(\varphi(0) + se_i) - V(\varphi(x))}{s} \,,
\end{equation}
where $(e_j)_{j = 1}^n$ is the standard basis of $\R^n$. Similarly, we define higher order derivatives. 
The It\^{o} formula for stochastic functional differential equations \eqref{eq:func-eq}, see \cite[(3.3)]{functional},  has the form
\begin{equation}
    dV(X_t) = \Ll V(X_t) dt + \sum_{i = 1}^n X_i(t)g_i(X) \partial_iV(X_t) dE_i(t) \,,
\end{equation}
where $\Ll$ is \eqref{eq:func-eq} as in \cite[(3.2)]{functional}: 
\begin{equation}\label{eq:lske}
    \Ll V(\varphi) = \partial_tV(\varphi) + 
    \sum_{i = 1}^n \varphi_i(0) f_i(\varphi) \partial_iV(\varphi) + 
    \frac{1}{2} \sum_{i, j = 1}^n \varphi_i(0)\varphi_j(0)
    \sigma_{ij} g_i(\varphi)g_j(\varphi) \partial_{ij}V(\varphi) \,.
\end{equation}
Then  carre du champ operator $\Gamma$ is 
$$
\Gamma V(\phi) \coloneqq \Ll (V^2)(\phi) - 2V(\phi)\Ll V(\phi) = \sum_{i,j = 1}^n x_ix_j\sigma_{ij}g_i(\phi)g_j(\phi)\partial_i V(\phi)\partial_j V(\phi) \,.
$$ 
It follows from the functional It\^{o} formula and a similar argument to \cite[Lemma 4.1, Corollary 4.2]{extinction} that all of the functions defined below belong to the appropriate domains (\Cref{D+} or \Cref{D2}).

\begin{lem}\label{func-WU}
    Assume \Cref{fass1}, fix $\gamma = \gamma_b/2$ and define
    \begin{equation}\label{Psi-functional}
        \Psi(\phi) \coloneqq \ln\Big({1 + \sum_{i=1}^n c_ix_i}\Big) + A_2\int_{-r}^0\int_s^0 e^{\gamma (u-s)} h(\phi(u))dud\mu(s) \,,
    \end{equation}
    where $x_i = \phi_i(0)$ and $h, \mu$ are as in \Cref{fass1}. 
Then there are $p > 0, K > 0$ so that $W \coloneqq e^{p\Psi}$ and $U \coloneqq e^{2p\Psi}$ satisfy: \begin{enumerate}[label=(\roman*)]
        \item \label{func-lw} $\Ll W \leq K - W'$,
        \item \label{func-lu} $\Ll U \leq K - U'$,
        \item \label{func-gamma} $J + \Gamma W \lesssim U'$, where $J$ is as in \Cref{fass1},
    \end{enumerate}
    where
\begin{equation}
    \begin{aligned}
        W'(\phi) &\coloneqq pW(\phi)\Big(\gamma_0 + Ah(x) + \frac{A_2\gamma}{2}\int_{-r}^0\int_s^0 e^{\gamma (u-s)} h(\phi(s))dud\mu(s) + \frac{\gamma_b}{2}J(\phi)\Big) \\
        U'(\phi) &\coloneqq 2pU(\phi)\Big(\gamma_0 + Ah(x) + \frac{A_2\gamma}{2}\int_{-r}^0\int_s^0 e^{\gamma (u-s)} h(\phi(s))dud\mu(s) + \frac{\gamma_b}{2}J(\phi)\Big)
    \end{aligned}
\end{equation}
and $A = A_1 - A_2 \int e^{-\gamma s}d\mu(s) > 0$ (see \Cref{fass1}).
\end{lem}

\begin{rem}
    Note that in \cite[Lemma 3.1]{functional}  our $\Psi$ is called $U$, but we renamed it to avoid confusion with $U$ in \Cref{as-U}.
\end{rem}

\begin{proof}
    The proof of \ref{func-lu} is exactly the same as the proof of \ref{func-lw} with $p$ replaced by $2p$, so we only prove \ref{func-lw} and \ref{func-gamma}. 
    Note that $\partial_t$ and $\partial_i$ of the respectively first and the second term in $\Phi$ vanish, and therefore by standard manipulations (see also \cite[Lemma 3.1]{functional}), we have 
    \begin{equation}
        \begin{aligned}
            \Ll \Psi(\phi) &= F(\phi) - \frac{1}{2}G(\phi) + A_2 h(x)\int_{-r}^0 e^{-\gamma s}d\mu(s) - A_2 \int^0_{-r} h(\phi(s))d\mu(s) \\
            &\qquad - A_2 \gamma \int_{-r}^0\int_s^0 e^{\gamma (u-s)} h(\phi(u)) dud\mu(s) \,,
        \end{aligned}
    \end{equation}
where $F$ and $G$ are as in \Cref{fass1}. 
Using \Cref{fass1}, we obtain \begin{equation}
    \begin{aligned}
        \Ll \Psi(\phi) &\leq -\gamma_bJ(\phi) + A_0\Id_{\{|x| < M|\}} - \gamma_0 - A_1h(x) \\
        &\qquad + A_2 \int^0_{-r} h(\phi(s))d\mu(s)+ A_2 h(x)\int^0_{-r} e^{-\gamma s}d\mu(s) \\
        &\qquad - A_2 \int^0_{-r} h(\phi(s))d\mu(s)
            - A_2 \gamma \int_{-r}^0\int_s^0 e^{\gamma (u-s)} h(\phi(u))dud\mu(s) \\
        &= -\gamma_bJ(\phi) + A_0\Id_{\{|x| < M\}} - \gamma_0 - Ah(x) \\
        &\qquad - A_2 \gamma \int_{-r}^0\int_s^0 e^{\gamma (u-s)} h(\phi(u))dud\mu(s) \,.
    \end{aligned}
\end{equation}
Again since $\partial_i$ of the second term in \eqref{Psi-functional} is equal to zero, we obtain
\begin{equation}\label{gamma-Psi-functional}
    \Gamma \Psi(\phi) = \Gamma \ln\Big({1 + \sum_{i=1}^n c_ix_i}\Big) = G(\phi)\,.
\end{equation}
Since $G \lesssim J$, we can choose $p > 0$ small enough so that $p\Gamma \Psi \leq \gamma_b J$. Then the usual It\^{o} formula or direct computation implies
\begin{equation}
    \begin{aligned}
        \Ll W(\phi) &= \Ll e^{p\Psi}(\phi) = pe^{p\Psi(\phi)}\Big(\Ll \Psi(\phi) + \frac{p}{2}\Gamma \Psi(\phi)\Big) \\
        &\leq pe^{p\Psi(\phi)}\Big(-\frac{\gamma_b}{2}J(\phi) + A_0\Id_{\{|x| < M\}} - \gamma_0 - Ah(x) \\
        &\qquad - A_2 \gamma \int_{-r}^0\int_s^0 e^{\gamma (u-s)} h(\phi(u))dud\mu(s)\Big) \\
        &= pe^{p\Psi(\phi)}\Big(A_0\Id_{\{|x| < M\}} - \frac{A_2\gamma}{2}\int_{-r}^0\int_s^0 e^{\gamma (u-s)} h(\phi(u))dud\mu(s)\Big) \\
        &\qquad - W'(\phi) \,.
    \end{aligned}
\end{equation}

To prove \ref{func-lw}, it suffices to show that for some $K > 0$
\begin{equation}\label{eq:bound-lw-k}
    pe^{p\Psi(\phi)}\Big(A_0\Id_{\{|x| < M\}} - \frac{A_2\gamma}{2}\int_{-r}^0\int_s^0 e^{\gamma (u-s)} h(\phi(u))dud\mu(s)\Big) \leq K \,.
\end{equation}  
If 
$$
\frac{A_2\gamma}{2}\int_{-r}^0\int_s^0 e^{\gamma (u-s)} h(\phi(u))dud\mu(s) \geq A_0 \quad \text{or} \quad |x| \geq M \,,
$$ 
then \eqref{eq:bound-lw-k} is bounded above by $0$. Otherwise $$\Psi(\phi) \leq \ln\Big(1 + |c|M\Big) + \frac{2A_0}{\gamma} \,,$$ where $c = (c_1,\dots,c_n)$, and so \ref{func-lw} holds with $K = \ln(1 + |c|M) + 2A_0/\gamma$.

To prove \ref{func-gamma}, we use \eqref{gamma-Psi-functional} and the usual It\^{o} formula or direct computation to compute
$$\Gamma W(\phi) = p^2e^{2p\Psi(\phi)} \Gamma \Psi(\phi) = p^2U(\phi)G(\phi) \,,$$ 
and consequently $U \geq 1$ and $G \lesssim J$ imply $$J + \Gamma W \lesssim UJ \lesssim U' \,,$$
as desired.
\end{proof}

\begin{cor}\label{fass1-and-fass2-are-satisfied}
Under \Cref{fass1}, \Cref{as-W} and \Cref{as-U} hold with $W, W', U, U'$ as in \Cref{func-WU} and \begin{equation}\label{func-V}
    V(\phi) \coloneqq \sum_{i=1}^n \rho_i V_i(x_i) \,,
\end{equation}
where $\rho_i > 0$, $x \coloneqq \phi(0)$, and $V_i: (0,\infty) \to [0,\infty)$ is a smooth function such that 
\begin{equation}
    V_i(x) = \begin{cases}
        -\ln(x) & x \in (0, 1)\,, \\
        0 & x \geq 2 \,.
    \end{cases}
\end{equation}
\end{cor}

\begin{proof}
    The inequalities $\Ll W \leq K - W'$, $\Ll U \leq K - U'$, and $\Gamma W \lesssim U'$ are shown in \Cref{func-WU}.
    
    We overload notation and use $V_i$ to also denote the function defined for $\phi \in \inv$ as $\phi \mapsto V_i(x_i)$. We have for $\phi \in \inv$ and $x_i \in (0,1)$ that $$\Gamma V_i(\phi) = \frac{1}{x_i^2}\sigma_{ii}^2x_i^2g_i^2(\phi) = \sigma_{ii}^2g_i^2(\phi) \,,$$ and if $x_i \in [1,2]$, then 
    $$\Gamma V_i(\phi) = V_i''(x_i)x_i^2g_i^2(\phi) \lesssim g_i^2(\phi) \,,$$ 
    where we used that $V_i$ is smooth and thus $V_i''(x)$ is uniformly bounded on $[1,2]$. Since $\Gamma V_i(\phi) = 0$ if $x_i \in [2,\infty)$, putting these inequalities together yields $\Gamma V \lesssim J$ (it is easily checked that $\Gamma V \leq n\sum_{i=1}^n \rho_i^2 \Gamma V_i$). By \Cref{func-WU}, 
    $J \lesssim U'$, and thus \Cref{as-U} is satisfied.

    Finally, $W << W'$ in the sense of \Cref{<<} follows immediately from $W \leq (p\gamma_0)^{-1}W'$.
\end{proof}

\begin{lem}\label{func-vanish}
    Assume \Cref{fass1}, \Cref{fass2} and define
    \begin{equation}\label{func-H}
        H(\phi) \coloneqq \sum_{i=1}^n \rho_iH_i(\phi) \,,
    \end{equation}
   where for any $\phi \in \mcM$
   \begin{equation}
   H_i(\phi) = 
   \begin{cases}
   -f_i(\phi) + \frac{1}{2}\sigma_{ii}g_i^2(\phi) & \textrm{if }  x_i = \phi_i(0) =0 \\        
   V_i'(x_i)x_if_i(\phi) + \frac{1}{2}\sigma_{ii}^2V_i''(x_i)x_i^2g_i^2(\phi) & \textrm{otherwise} \,,
       \end{cases}
   \end{equation}
and $V$ is defined in \Cref{fass1-and-fass2-are-satisfied}. Then
\Cref{as-V}\ref{as-vanish} is satisfied with the measure $\mu + \delta_0$ and $W', V$, where $\mu$ is defined in \Cref{fass1} and $W'$ and $V$
 are as in \Cref{func-WU} and \Cref{fass1-and-fass2-are-satisfied} respectively. 
\end{lem}

\begin{proof}
From standard calculations and the definition of $\Ll V$ in \eqref{eq:lske} follows that $H(\phi) = \Ll V(\phi)$ if $\phi \in \mcM_+$. Clearly, $H$ is continuous on $\mcM$, and therefore
 $H$ is indeed the continuous extension of $\Ll V$ to $\mcM$.
  Also, $|H| \lesssim J$ by using similar arguments as in the proof of $\Gamma V_i \lesssim J$ in \Cref{fass1-and-fass2-are-satisfied}.
    
     To prove \Cref{as-V}\ref{as-vanish}  with the measure $\mu + \delta_0$, 
    from \Cref{rv-vanish} and the almost sure continuity of the paths of \eqref{eq:func-eq}, it suffices to show the following statement: for all $\epsilon > 0$ there is $N > 0$ such that for all continuous functions $\xi: [-2r, 0] \to [0,\infty)^n$ we have \begin{equation}\label{eq:jeawp}
        J(\phi_0) \leq N + \epsilon \Big(W'(\phi_0) + \int_{-r}^0 W'(\phi_s)d\mu(s)\Big) \,,
    \end{equation} 
    where $\phi_s$ is defined for $s,t \in [-r,0]$ by $\phi_s(t) = \xi(s+t)$.

    We give a proof of \eqref{eq:jeawp} under either the assumption \eqref{eq:fass-1} or \eqref{eq:fass-2}. First, suppose \eqref{eq:fass-1} holds. If $W(\phi) \geq 1/(\epsilon pA)$ for some $\phi \in \mcM$, then by the definition of $W'$ 
    and $x = \phi(0)$ we have 
    $$
    h(x) \leq \epsilon pAW(\phi) h(x) \leq \epsilon W'(\phi) \,.
    $$ 
    If $W(\phi) \leq 1/(\epsilon pA)$, then  
    by using $W \geq (1+\sum c_ix_i)^p$ we obtain
    $|x| \leq (\tilde{c}/(\epsilon pA))^{\frac{1}{p}}$, where $\tilde{c} > 0$ is a fixed constant only depending on $c_i$. Since $h$ is a continuous function, $|h| \leq N'$ 
    on the compact set $\{x: |x| \leq (\tilde{c}/(\epsilon pA))^{\frac{1}{p}}\}$. To summarize, we have 
    \begin{equation}\label{eq:srhb}
    h(x) \leq N' + \epsilon W'(\phi) \,,    
    \end{equation}
    and \eqref{eq:jeawp} follows from \eqref{eq:fass-1}.

    If \eqref{eq:fass-2} holds, then by the definition of $W'$ for any $x \in [0, \infty)^n$ we have
    $$
    h_1(x) \lesssim J(\phi_x) \lesssim \frac{W'(\phi_x)}{W(\phi_x)} \,,
    $$
    where $\phi_x \equiv x$. Since for every $\epsilon > 0$
    there is $M_\epsilon$ such that $W(\phi_x) \geq \frac{1}{\epsilon}$ if $|x| \geq M_\epsilon$ and 
    $J(\phi_x)$ is uniformly bounded for $\{x \in [0, \infty)^n  \mid |x| \leq M_\epsilon\}$, we obtain that \eqref{eq:srhb} holds and \eqref{eq:jeawp} follows from \eqref{eq:fass-2}. 
\end{proof}

To obtain the compactness required by \Cref{as-compact}, we use  similar argument as in \cite[Lemma 3.3]{functional}, which relies on the following version of the Kolmogorov continuity criterion.

\begin{thm}\label{Kolmogorov}
    Let $(X_t)_{t\in[0, T]}$ be a stochastic process on $\R^n$ for some $T > 0$. Suppose there exist constants $\alpha, \beta, \gamma, K, C > 0$ such that:
    \begin{itemize}
        \item $\E[|X_t - X_s|^\alpha] \leq K|t-s|^{1 + \beta}$ for all $t,s \in [0,T]$.
        \item $\E[\sup_{t \in [0,T] \cap D} |X_t|] \leq C$, where $D$ denotes the dyadic numbers $\frac{k}{2^n}$.
        \item $\gamma < \beta / \alpha$.
    \end{itemize}

    Then there is a modification $\tilde{X_t}$ of $X_t$ with almost surely $\gamma$-H\"{o}lder continuous sample paths. Additionally, there is a function $A: [0,\infty) \to [0,1]$ which depends only on $T,\alpha,\beta,\gamma,K,C$ (not on $X$) and satisfies: \begin{itemize}
        \item $\lim_{M \to \infty} A(M) = 0$.
        \item $\Prb(\|\tilde{X}\|_\gamma > M) \leq A(M)$, where $$\|\tilde{X}\|_\gamma \coloneqq |\tilde{X}_0| + \sup_{0 \leq s < t \leq T} \frac{|\tilde{X}_t - \tilde{X}_s|}{t-s} \,.$$
    \end{itemize}
\end{thm}

\begin{rem}
    The usual Kolmogorov criterion implies that for any process $X$ satisfying suitable bounds, it has a modification which is almost surely H\"{o}lder continuous, so the (random) H\"{o}lder constant $C_h$ is finite almost surely and thus $\Prb(C_h > M) \to 0$ as $M \to \infty$. In \Cref{Kolmogorov} we note that the rate at which $\Prb(C_h > M)$ converges to $0$ only depends on the process $X$ through certain constants. This observation is necessary for verifying \Cref{as-compact}, because we need our $\K$ to be independent of the initial condition.
\end{rem}

\begin{proof}
    The existence of such a modification is given in \cite[Theorem 2.9]{LeGall}. Examination of the proof reveals that there is a constant $C_\gamma$ (depending only on $\gamma$) such that $|X_t(\omega) - X_s(\omega)| \leq C_\gamma |t-s|^\gamma$ for all (dyadic) $s,t$ such that $|s-t| \leq 2^{-n^*(\omega)}$, where $n^*$ is an $\N$-valued random variable with $\Prb(n^* > n) \leq a_n$ for some sequence $a_n \to 0$ which depends only on $T,\alpha,\beta,\Gamma,K$. Thus, almost surely $$\|\tilde{X}\|_\gamma \leq (1 + 2\cdot2^{\gamma n^*})\sup_{t \in [0,T] \cap D}|X_t| + C_\gamma \,,$$ from which the existence of $A$ follows.
\end{proof}

\begin{lem}\label{functional-compact-holds}
    Under \Cref{fass1}, \Cref{as-compact} holds.
\end{lem}

\begin{proof}
    Recall our choice of $W$ in \Cref{func-WU}. If $W(\Phi_t) \leq M^p$ for all $t \in [0,T]$, then $|X_t| \leq \tilde{c}M$, where $\tilde{c} > 0$ is a fixed constant only depending on $c_i$ and $p$ (the same one as in the proof of \Cref{func-vanish}), and thus $\|\Phi_t\| \leq \tilde{c}M$ for all $t \in [r,T]$. 
    
    Thus, on the event $\{W(\Phi_t) \leq M^p \textrm{ for all } t \in [0,T]\}$, the process $\{\Phi_t\}_{t \in [r,T]}$ agrees with the truncated process $\tilde{\Phi}_t(s) = \tilde{X}(t+s)$, where $\tilde{X}_t$ satisfies the same stochastic functional differential equation as $X_t$ (see \eqref{eq:func-eq}) but with $f_i,g_i$ replaced by bounded Lipschitz functions $f_i^M, g_i^M$ which agree with $f_i,g_i$ on the set $\{\|\phi\| \leq \tilde{c}M \}$ (for example, such truncated functions are defined in the proof of \cite[Lemma 3.3]{functional}), and initial condition $\tilde{\Phi}_{r} = \Phi_r$. 
    
    It is standard to show that $\{\tilde{X}_t\}_{t \in [r,T]}$ satisfies the assumptions of \Cref{Kolmogorov} with $\alpha = 4$, $\beta = 1$, $\gamma = 1/8$, and $K,C$ depending only on $M,T$ and not on the initial condition $\Phi_r$ which satisfies $\|\Phi_r\| \leq \tilde{c}M$ almost surely. Thus, we have that \eqref{eq:as-compact} holds with $\K = \{\phi \in \mcM \mid \|\phi\|_\gamma \leq C_h \}$ for $C_h$ large enough depending only on $M,T,\epsilon$. The set $\K$ is compact by Arzela-Ascoli theorem, so \Cref{as-compact} holds.
\end{proof}

Finally we prove \Cref{functional-thm}.

\begin{proof}[Proof of \Cref{functional-thm}]
    At the beginning of this section we noted that \Cref{as1}--\ref{as2} hold. In \Cref{fass1-and-fass2-are-satisfied}, \Cref{func-vanish}, and \Cref{functional-compact-holds} we proved that all other assumptions except \Cref{as-V}\ref{as-alpha} hold.
    
    Although it is known (see \cite[Lemma 2.1]{ecologicalContinuous}) that $\mu H = -\sum_{i=1}^n \rho_i\lambda_i(\mu)$ for all $\mu \in P_{inv}(\mcM_0)$, we at least provide some intuition. Clearly, it suffices to show the identity for $\mu$ being ergodic. Since $\{x\mid x_i = 0\}$ and $\{x\mid x_i > 0\}$ are invariant for each $i$, each set has $\mu$ measure equal to 0 or 1. If for some $i$ we have $\mu(\{x\mid x_i = 0\}) = 1$, then $\mu$ almost surely $\mu H_i = -f_i + \frac{\sigma_{iig_i^2}}{2}$ and the claim follows. If $\mu(\{x\mid x_i = 0\}) = 0$, then as in the proof of \Cref{muH-neg} it holds that $\mu(H_i) = \lambda_i(\mu) = 0$.

    Choose $\rho_i > 0$ in \eqref{func-V} and \eqref{func-H} and $\Lambda < 0$ such that 
    $$
    \mu H = -\sum_{i=1}^n \rho_i\lambda_i(\mu) \leq \Lambda < 0
    $$
    for all $\mu \in P_{inv}(\mcM_0)$. This is possible by applying Hahn-Banach separation theorem to the set $\{(\lambda_1(\mu), \dots, \lambda_n(\mu)) \mid \mu \in P_{inv}(\mcM_0)\}$, which is compact by \Cref{pinv-compact} and \Cref{H-in-L1} and is disjoint from $\{x \in \R^n \mid x_i \leq 0 \text{ for all } i \leq n\}$ by \Cref{fass3}. Then \Cref{as-V}\ref{as-alpha} is satisfied.

    Since $\phi \mapsto \min_{i \leq n}\phi(0)_i$ is continuous and strictly positive on $\inv$, it has a strictly positive infimum on any compact subset of $\inv$. Also, $\phi \mapsto \max_{i \leq n}\phi(0)_i$ has a finite supremum on any compact set. Thus, the theorem follows from the second statement of \Cref{main}.
\end{proof}

\section{Application 2: Stochastic PDEs}\label{section-SPDE}

In this section we discuss general stochastic PDEs with positive solutions. Numerous examples are given in \Cref{examples}.
The considered (system of) stochastic PDE has the form \begin{equation}\label{SPDE}
    dx_t = (Ax_t + F(x_t,\theta))dt + G(x_t,\theta)dW_t \,.
\end{equation}

 Let us fix the notation which is used throughout \Cref{section-SPDE}. The precise definitions and assumptions on $F$ and $G$ are given below after a development of the linear theory.

\begin{itemize}
    \item 
    A standing assumption throughout this section is that there is a filtered probability space $(\Omega, \mathcal{F}, \{\mathcal{F}_t\}_{t \geq 0}, \Prb)$ to which all  stochastic processes are adapted.
    \item We fix $N\geq 1$ to be the dimension of the space and $m$ to be the dimension of the process. Specifically, we assume that 
    $x_t = (x^{(1)}_t, \dots,x^{(m)}_t)$ and  $x^{(j)}$ for each $j = 1, \dots, m$ is a real valued 
    random function on the domain $D \subset \R^N$, where 
    $D$ is either a bounded open connected subset of $\R^N$ with smooth boundary or $D = \T^N$ is the $N$-dimensional torus.
    \item The stochastic process $x_t$ takes values in the space $\cdom \coloneqq \{f \in B \mid f \geq 0\}$, where $f \geq 0$ is understood component-wise, and $B \coloneqq C(\overline D) \coloneqq C(\overline D, \R^m)$ is the space of all continuous functions $f: \overline D \to \R^m$ 
    equipped with the supremum norm. 
      \item $A = (A_1,\cdots,A_m)$ acts as $Ax_t = (A_1x^{(1)}_t, \cdots, A_mx_t^{(m)})$, where each $A_i$ is an elliptic second order differential operator with appropriate boundary conditions specified below.
    
    \item Fix $d \in \{1, \dots, m\}$ and write
     $x_t = (u_t,z_t)$, where $u_t$ and $z_t$ represent respectively the first $d$ and last $m-d$ components of $x_t$. The reader can think of $u_t$ as representing the population of $d$ species for which we investigate the survival/extinction and $z_t$ models the populations of other species or environmental factors. 
     In addition, 
  we use $\pi_u: \R^m \to \R^m$ (resp. $\pi_z: \R^m \to \R^m$) to denote the projections onto the first $d$ (resp. last $m-d$) coordinates. Unless otherwise specified, $x_u$ (resp. $x_z$) denotes the first $d$ (resp. last $m - d$) coordinates of the object $x$. For example, $A_u = (A_1,\dots,A_d)$ and   
 we could rewrite \eqref{SPDE} as
    \begin{align*}
        du_t &= [A_uu_t + F_u(u_t,z_t,\theta)]dt + G_u(u_t,z_t,\theta)dW_t \\
        dz_t &= [A_zz_t + F_z(u_t,z_t,\theta)]dt + G_z(u_t,z_t,\theta)dW_t 
    \end{align*}
or we use $x_u = u$ and $x_z = z$. 

\item
 Also, if $x,y \in \R^m$ we denote $x\cdot y \in \R^m$ the component-wise multiplication of $x$ and $y$. 
    
    \item Let $\Sigma$ be the Borel sigma-algebra on $D$ and $\mu$ be a measure admitting a smooth density with respect to Lebesgue measure which is positive and uniformly bounded away from zero.
    We use $L^p$ as a shorthand for $L^p(D, \Sigma, \mu)$, $L^p(D, \Sigma, \mu;\R^d)$, and $L^p(D, \Sigma, \mu;\R^m)$, which is  clear from the context. Note that due to our assumptions on $\mu$, 
    the norm on $L^p(D, \Sigma, \mu)$ is equivalent to $L^p(D, \Sigma, \lambda)$, where $\lambda$ is the Lebesgue measure. In particular, all Sobolev embeddings hold true for our $\mu$ and we use this observation below without notice. 
    \item $\theta$ is some (fixed) parameter which is an element of the compact metric space $\Theta$. In our examples, $\theta$ represents the strength of the noise.

    \item $F$ is a sufficiently smooth function from $B \times \Theta$ to $B$.
    \item $W_t$ is a cylindrical Brownian motion on a separable Hilbert space $\U$.
    
    \item $G$ is a sufficiently smooth map from $B \times \Theta$ to the space of bounded linear operators from $\U$ to $L^2$, denoted by $\Ll(\U,L^2)$.
\end{itemize}

Below, our assumptions guarantee that $\{(x_u)_0 = u_0 \equiv 0\}$ is an invariant set of \eqref{SPDE} and our goal is to develop criteria which imply that $x_t$ is stochastically persistent on $\cdom \setminus \{u \equiv 0\}$ in the sense of \Cref{main}. In ecological models, this indicates survival of the first $d$ species. The main applications are to stochastic reaction-diffusion equations coming from biology and models for turbulence.

Additionally, we will study the dependence of our results on $\theta$, which will be quite useful in the small noise regime. In particular, if $\Theta = [0,1]$ and $G(u,\theta) = \theta G(u)$, then for $\theta = 0$ the problem \eqref{SPDE} becomes a deterministic PDE. In \Cref{examples} this will allow us to say that persistence occurs in the small noise regime if an eigenvalue corresponding to this PDE is positive, and often we are able to explicitly calculate this eigenvalue.

\begin{rem}\label{we-can-also-handle-switching}
    All of our results also apply to models with regime-switching, meaning they remain valid if $\theta$ is allowed to depend on time in such a way that $(x_t,\theta_t)$ forms a Markov process. In this case, $\Theta$ need not be compact as long as $\theta_t$ is suitably stable (empirical occupation measures are tight, etc.). Since none of our examples have regime-switching, we decided to keep $\theta$ a constant  for brevity.
\end{rem}

To ease notation, throughout the following sections we often suppress the dependence of $F,G$ on $\theta$.

\subsection{Semigroups}\label{SPDE-setup}
We begin by imposing some assumptions on the operator $A$ and proving some properties of its semigroup and interpolation spaces.

\begin{deff}\label{sobolev-spaces}
    For $k \in \N$ and $1 < p < \infty$ we define the Sobolev spaces recursively by $W^{0,p} \coloneqq L^p$ and
    $$W^{k,p} \coloneqq \{u \in L^p \mid \partial_iu \in W^{k-1,p} \text{ for } i = 1,\dots, N\}$$
    endowed with the norm
    $$\|u\|_{W^{k,p}} \coloneqq \sum_{i=1}^N\|\partial_i u\|_{W^{k-1,p}} + \|u\|_{L^p} \,.$$
    In the definition, it is implicit that the weak derivatives $\partial_i u$ exist. For $s > 0$ (not necessarily an integer), we define $W^{s,p}$ via complex interpolation, see \cite[Theorem 6.4.5 (7)]{interpolation-simple} with the domain $\R^n$, but the results also hold for our domain $D$ via extension operators.

    If $D = \T^N$, we also define $W^{-s,p}$ to be the dual of $W^{s,p^*}$, where $1/p + 1/p^* = 1$.
\end{deff}

\begin{sass}\label{sass-A}
    Each operator $A_i$, $i = 1, \dots, m$ is of the form
    $$
    \sum_{1 \leq j, k \leq N} a^i_{jk}(x)\partial_j\partial_k + \sum_{j=1}^N b_j^i(x)\partial_j + c^i(x) \,,
    $$
    where $a^i_{jk}, b^i_j, c^i$, $j, k = 1, \dots, N$ are smooth functions on $\overline D$ and there is $\eta > 0$ such that for all $y \in \R^N$ and all $i = 1, \dots, m$
    $$
    \sum_{1 \leq j, k \leq N} a^i_{jk}(x)y_jy_k \geq \eta |y|^2 
    \qquad\qquad \textrm{(ellipticity).}
    $$
    We define the domain of the unbounded operator $A_i$ on $L^p$ for $1 < p < \infty$ based on the boundary conditions on $\partial D$:

    \begin{align*}
    \Dm(A_i) \coloneqq \begin{cases}
   \{u \in W^{2,p} \mid \frac{\partial u}{\partial \nu} \equiv 0 \text{ on } \partial D\} &  \text{Neumann Boundary Conditions}  \\
 W^{2,p} &  \text{Periodic Boundary Conditions} \,,
\end{cases}
    \end{align*}
where $\nu$ denotes the outward-pointing normal vector to $\partial D$.

We also assume the following assertions  for each $i = 1, \dots, m$:
\begin{enumerate}[label=(\roman*)]
    \item Every $z \in \mathbb{C}$ with $\Re z > -1$ lies in the resolvent set of $A_i$.
    \item $A_i$ is self-adjoint when $p = 2$.
    \item The solution $u$ to $\frac{\partial u(t,x)}{\partial t} = A_iu(t,x)$ is nonnegative if $u(0,\cdot) \geq 0$. In other words, the semigroup $S_i$ generated by $A_i$ (\Cref{semigrp-deff}) preserve positivity.
\end{enumerate}
\end{sass}

\begin{rem}\label{boundary-cond-is-arb}
    We restrict to Neumann or periodic boundary conditions because of our examples below, but our results also hold for more general (first order) boundary conditions as well (see \cite[(3.1.3),(3.1.4)]{lunardi}), for example those investigated in \cite{allee-effect}.
\end{rem}

\begin{rem}\label{rem:subtract-from-c}
    In our examples, we will not verify (i) or (iii) in 
    \Cref{sass-A}, because they are always guaranteed by subtracting a suitably large constant from $c$ and then adding that constant to $F$, which, as the reader may check, does not violate any  assumptions on $F$ stated below.
    
    Specifically, since $D$ is bounded, the constant we need to subtract 
    from $c$ to satisfy (i) in \Cref{sass-A} 
    does not depend on $p$, see \cite[Theorem 3.1.2]{lunardi}. In addition (iii) in \Cref{sass-A} follows from \cite[Section 7.1, Theorem 12(ii)]{evans}. Indeed, if $c \leq 0$ and $u_0 \geq 0$
     suppose that there is $t > 0$ such that 
    $S(t)u_0$ is not strictly positive in $D$. Then $S(t)u_0$ attains non-positive minimum and by the strong maximum principle, see \cite[Section 7.1, Theorem 12(ii)]{evans},  $S(t)u_0 \equiv c \leq 0$. By intermediate value theorem and continuity we obtain $S(\tau)u_0 \equiv 0$ for $\tau = \inf \{t > 0 : S(t)u_0 \not > 0\}$. If $\tau = 0$, then $u_0 \equiv 0$ and consequently $u \equiv 0$. If $\tau > 0$, this means that $S(\tau)$ is not injective. However, by (i) and (ii) of \Cref{sass-A}, 
 $S(\tau) = e^{\tau A}$ is compact and self-adjoint with strictly positive eigenvalues, so it is injective, a contradiction.

    Finally, we can without loss of generality assume in what follows that $S(t)$ is a contraction in $B$, which again is immediate if $c \leq 0$ (\cite[Section 7.1, Theorem 9]{evans}).
\end{rem}

\begin{deff}\label{semigrp-deff}
    By \cite[Theorem 3.1.3]{lunardi}, the operators $A_i$ are sectorial and generate analytic semigroups $S_{i,p}(t)$ on $L^p(D)$ for $1 < p < \infty$ (see \cite[Definition 2.0.2]{lunardi}). We use $S_p(t) = (S_{1,p}(t), \dots, S_{m,p}(t))$ to denote the corresponding semigroup on $L^p(D;\R^m)$. In addition, 
    $\|S_{i, p}(t)\|_{\mathcal{L}(L^p, L^p)} \lesssim e^{-t}$ for any $t \geq 0$, $p \in (1, \infty)$, and $i = 1, \dots, m$, 
    see \cite[Proposition 2.1.1]{lunardi}, 
\end{deff}

\begin{deff}
Let $S$ be a semigroup  with generator $A$ satisfying $\|S(t)\|_{\mathcal{L}(L^p, L^p)} \lesssim e^{-t}$ (for example if $A$
is a sectorial operator, see \Cref{semigrp-deff}).
  Then for any $\vartheta > 0$ we define the fractional power of $(-A)$ as
  \begin{equation}\label{eq:doapt}
      (-A)^{-\vartheta} = \frac{1}{\Gamma(\vartheta)} \int_0^\infty t^{\vartheta - 1} S(t)  dt \,, 
  \end{equation}
  where $\Gamma$ is the usual $\Gamma$-function and the integral exist by the assumed bound on $S$. 

  Furthermore, we define 
  \begin{equation}
      (-A)^\vartheta = ((-A)^{-\vartheta})^{-1} \qquad \textrm{and}\quad A^0 = I \,.
  \end{equation}
\end{deff}

\begin{rem}\label{rem:neg-powers-r-inj}
    Observe that for $\vartheta = 1$, $(-A)^{-1}$ coincides with the inverse of $(-A)$, which is equal to the resolvent operator for $\lambda = 0$, see \cite[Lemma 2.1.6]{lunardi}. 
One can verify (see for example \cite[Theorem 3.1.3]{lunardi}) that 
$(-A)^{-\vartheta}(-A)^{-\varsigma} = (-A)^{-(\vartheta + \varsigma)}$
for any $\vartheta, \varsigma > 0$. Consequently 
if $0$ is in the resolvent set of $A$, as assumed in \Cref{sass-A}, then
$(-A)^{-\vartheta}$ is injective. Indeed, $(-A)^{-1}$ is injective by definition and thus so is $(-A)^{-n}$ for any positive integer $n$. Choosing $n > \vartheta$ we have $(-A)^{-(n - \vartheta)}(-A)^{-\vartheta} = (-A)^{-n}$, and thus $(-A)^{-\vartheta}$ is injective.
\end{rem}

\begin{deff}\label{fractional-domains-deff}
    For $\vartheta > 0$, let $\B_{\vartheta,r} = \Dm((-A)^\vartheta) = \textrm{Range}((-A)^{-\vartheta})$ 
    be the fractional domains 
    on $L^r$ with the corresponding norms $\|u\|_{\vartheta,r} \coloneqq \|(-A)^\vartheta u\|_{L^r}$ (see \cite[Section 5.4]{hairer-notes} for more details). If $\vartheta < 0$, we define $\B_{\vartheta,r}$ as the completion of $L^r$ with respect to the norm $\|u\|_{\vartheta,r} \coloneqq \|(-A)^\vartheta u\|_{L^r}$. 
    When $r$ is fixed, we often omit it and simply write $\B_\vartheta$ or $\|u\|_\vartheta$. 
    
    Assume $g$ is a linear operator mapping $\U$ to $\B_{\vartheta,r}$
and there exists an orthonormal basis $(h_n)_{n = 1}^\infty$ of $\U$ such that 
\begin{equation}\label{eq:gron}
    \|g\|_{\gamma(\U,\B_{\vartheta,r})} := \Big(\E \Big\| \sum_{n =1}^\infty \gamma_n g(h_n) \Big\|_{\B_{\vartheta,r}}^2\Big)^{\frac{1}{2}} < \infty \,,
\end{equation}
where $(\gamma_n)_{n = 1}^\infty$ is a sequence of independent standard Gaussian random variables. Then we say $g$ is a $\gamma$-radonifying operator and for a brevity of notation we use $\|g\|_{\vartheta,r} \coloneqq \|g\|_{\gamma(\U,\B_{\vartheta,r})}$
 (see \cite{radonifying-survey, umd-existence, umd-integration} for more about $\gamma$-radonifying operators). 
\end{deff}

\begin{rem}\label{rmk:gropp}
    We remark that the norm in \eqref{eq:gron} is independent of the orthonormal basis $(h_n)_{n = 1}^\infty$ and the space of $\gamma$-radonifying operators is a separable Banach space, since $\U$ is assumed to be separable. In addition, 
    $\|\cdot\|_{\mathcal{L}(\U, \B_{\vartheta,r})} \leq \|\cdot\|_{\gamma(\U, \B_{\vartheta,r})}$ (see \cite{radonifying-survey} for more details). 
\end{rem}

The following lemma describes the relationship between $B_{\vartheta, r}$ and the Sobolev spaces (\Cref{sobolev-spaces}). Below when we write that two Banach spaces are equal we also mean that their norms are equivalent.

\begin{lem}\label{characterization-of-interpolation-spaces}
Let $a = 1$ if the boundary conditions in \Cref{sass-A} are Neumann, and $a = 0$ otherwise. Then for $1 < p < \infty$:
\begin{enumerate}
    \item For $1/(2p) + a/2 < \vartheta < 1$, $\B_{\vartheta,p} = \{u \in W^{2\vartheta,p} \mid Bu \equiv 0 \text{ on } \partial D\}$ (where $B = \partial_\nu$ for Neumann and $B = 0$ for periodic).
    \item For $0 < \vartheta < 1/(2p) + a/2$, $\B_{\vartheta,p} = W^{2\vartheta,p}$.
    \item For $-1 < \vartheta < 0$, $\B_{\vartheta,p} = (\B_{-\vartheta,p^*})^*$, where $1/p + 1/p^* = 1$. In particular, for periodic boundary conditions $\B_{\vartheta,p} = W^{2\vartheta,p}$ for all $-1 \leq \vartheta \leq 1$.

\end{enumerate}
\end{lem}

\begin{rem}
    If $A$ were not self-adjoint, then \Cref{characterization-of-interpolation-spaces} would still hold except for (3), where we would have $\B_{\vartheta,p} = (\tilde \B_{-\vartheta,p^*})^*$ instead, with $\tilde \B_{\vartheta,p}$ being the interpolation spaces corresponding to $A^*$.
\end{rem}

\begin{proof}
    The first and second claims follow from \cite[Theorem 3 and Corollary after Theorem 4]{complex-interp-equals-domain-of-frac-powers}, which identifies $\B_{\vartheta,p}$ as the complex interpolation space $[L^p, D(A)]_\theta$, and \cite[Theorem 4.1]{interpolation-with-bc}, which describes the complex interpolation spaces exactly as stated in \Cref{characterization-of-interpolation-spaces}.

    For the third claim, consider the map $u \mapsto \phi_u$ from $\B_{\vartheta,p}$ to $(\B_{-\vartheta,p^*})^*$ given by $\phi_u(v) \coloneqq \ip{(-A)^\vartheta u}{(-A)^{-\vartheta}v}$ for $u \in \B_{\vartheta,p}$ and $v \in \B_{-\vartheta,p^*}$. This map is an isometry because
    \begin{align*}
        \|u\|_{\vartheta,p} &= \|(-A)^\vartheta u\|_{L^p} = \sup_{\|w\|_{L^{p^*}} = 1} |\ip{(-A)^\vartheta u}{w}| \\
        &= \sup_{\|v\|_{-\vartheta,p^*} = 1} |\ip{(-A)^\vartheta u}{(-A)^{-\vartheta}v}| = \|\phi_u\|_{(\B_{-\vartheta,p^*})^*} \,,
    \end{align*}
     where we made the substitution $w = (-A)^{-\vartheta}v$. This substitution is justified because $(-A)^{-\vartheta}$ can be defined as an isomorphism from $\B_{-\vartheta,p^*}$ to $L^{p^*}$:
     
     Indeed, by \Cref{fractional-domains-deff} $L^{p^*}$ is dense in $\B_{-\vartheta,p^*}$ and $(-A)^{-\vartheta}$ is an isometry from the subspace $L^{p^*}$ of $\B_{-\vartheta,p^*}$ to $L^{p^*}$, so it can be extended to an isometry on all of $\B_{-\vartheta,p^*}$. Suppose its image is not dense in $L^{p^*}$. Then there is $0 \not \equiv \bar{w} \in L^{p}$ such that $\langle (-A)^{-\vartheta}v, \bar{w}\rangle = 0$ for each $v \in \B_{-\vartheta,p^*}$. Since $A$ is self-adjoint, we obtain $
     \langle v, (-A)^{-\vartheta}\bar{w}\rangle = 0$ for all smooth $v$, which yields $(-A)^{-\vartheta}\bar{w} = 0$.  Using that $(-A)^{-\vartheta}$ is injective (\Cref{rem:neg-powers-r-inj}), we obtain $w \equiv 0$, a contradiction. Thus, $(-A)^{-\vartheta}: \B_{-\vartheta,p^*} \to L^{p^*}$ is an isometry with dense image, so it is an isomorphism.

    Finally, $\phi_\cdot$ is surjective because every $\phi \in (\B_{-\vartheta,p^*})^*$ can be factored as the composition of $(-A)^{-\vartheta}: \B_{-\vartheta,p^*} \to L^{p^*}$ and $\psi \coloneqq \phi \circ (-A)^{\vartheta}: L^{p^*} \to \mathbb{C}$, and $\psi$, by the Riesz representation theorem, has the form $\psi(f) = \ip{g}{f}$ for some $g \in L^p$, meaning $\phi$ has the form $\phi(v) = \ip{g}{(-A)^{-\vartheta}v}$, and thus $\phi = \phi_{(-A)^{-\vartheta}g}$.
\end{proof}

Next, we define convolutions with respect to the semigroup from \Cref{semigrp-deff}. For more details about (stochastic) integration in infinite dimensions, we refer the reader to \cite{umd-integration, umd-existence}.

\begin{deff}
    We use $*$ to denote convolution and $\diamond$ to denote stochastic convolution. That is,
    \begin{equation}\label{convolutions}
        (S * y)_t \coloneqq \int_0^tS(t-s)y_sds \quad \text{and} \quad (S \diamond Z)_t \coloneqq \int_0^t S(t-s)Z_sdW_s \,.
    \end{equation}
\end{deff}

Our main tools when using the semigroup approach are 
\Cref{convolution-bounds}, which is instrumental for analyzing the long-term tightness properties of our system 
and \Cref{convolution-sup-bounds}, which is 
crucial for proving well-posedness.
We remark that \Cref{convolution-bounds} and  \Cref{convolution-sup-bounds}  hold generally for any analytic semigroup on a UMD Banach space of type $2$ with interpolation spaces defined as in \cite[Definition 5.30]{hairer-notes}. In the proofs below $f*g$  denotes the usual convolution 
of functions $f,g: (0,T) \to \R$ given by 
$$
(f*g)(t) = \int_0^tf(s)g(t-s)ds
$$
and $L^p(0,T)$ is  
the set of all such $f$ with 
$$
\|f\|_{L^p(0,T)} \coloneqq \Big(\int_0^T |f(t)|^pdt\Big)^{1/p} < \infty \,.
$$

\begin{lem}\label{convolution-bounds}
    
    Fix $2 \leq r < \infty$ and $\B \coloneqq L^r$. For all $\alpha \in \R$, $\infty > p \geq 1$, $\epsilon > 0$, $T > 0$ the following assertions hold:
    \begin{enumerate}[label=(\roman*)]
        \item For $x \in \B_{\alpha - 1/p + \epsilon}$,
        $$\int_0^T \|S(t)x\|^p_\alpha dt \lesssim \|x\|_{\alpha - 1/p + \epsilon}^p \,.$$
        \item For $y \in L^p([0,T]; \B_{\alpha - 1 + \epsilon})$,
        $$\int_0^T \|(S * y)_t\|^p_\alpha dt \lesssim \int_0^T \|y_t\|^p_{\alpha - 1 + \epsilon} dt \,.$$
        \item If $p > 1$ then for $Z \in L^p(\Omega \times [0,T];\gamma(\U, \B_{\alpha - 1/2 + \epsilon}))$,
        $$\E\Big[\int_0^T \|(S \diamond Z)_t\|^p_\alpha dt\Big] \lesssim \E\Big[\int_0^T \|Z_t\|^p_{\gamma(\U,\B_{\alpha - 1/2 + \epsilon})} dt\Big] \,.$$
    \end{enumerate}
    In the above, the constant in the $\lesssim$'s do not depend on $x,y,Z$, and the dependence on $T$ is uniform over bounded subsets of $(0,\infty)$.
\end{lem}

\begin{proof}
    The proof relies heavily on the estimate

    \begin{equation}\label{semigrp-bound}
        \|S(t)x\|_\alpha \lesssim t^{\beta - \alpha}\|x\|_\beta
    \end{equation}
    
    (see \cite[Exercise 5.41]{hairer-notes}).    Since $\int_0^1 t^{p(\beta - \alpha)}dt < \infty$ if $\beta - \alpha > -1/p$, the statement (i) follows from \eqref{semigrp-bound}.

    Define $f_p(t) \coloneqq t^p$ and note $\|f_p\|_{L^1(0,T)} \lesssim_p 1$ for $p > -1$.
    
    For (ii), apply Young's inequality:
    \begin{multline*}
        \int_0^T \|(S * y)_t\|^p_\alpha dt \leq \int_0^T \Big(\int_0^t \|S(t-s)y_s\|_\alpha ds\Big)^pdt \\
        \begin{aligned}
        &\lesssim \int_0^T \Big(\int_0^t f_{-1+\epsilon}(t-s)\|y_s\|_{\alpha - 1 + \epsilon} ds\Big)^pdt \\
        &= \Big\|f_{-1+\epsilon}(t) *\|y_t\|_{\alpha - 1 + \epsilon}\Big\|^p_{L^p(0,T)} \leq \|f_{-1+\epsilon}\|_{L^1(0,T)}^p \int_0^T \|y_t\|^p_{\alpha - 1 + \epsilon} dt \,.
    \end{aligned}
    \end{multline*}

   To prove (iii), first note that \eqref{eq:gron} and \eqref{semigrp-bound} give
  \begin{equation}
  \begin{aligned}
  \label{semigrp-bound-Z}
        \|S(t)Z\|_{\gamma(\U, \B_\alpha)} 
        &= \Big(\E \Big\| S(t) \Big(\sum_{n =1}^\infty \gamma_n Z(h_n) \Big)\Big\|_{\B_{\alpha,r}}^2\Big)^{\frac{1}{2}} 
        \\
        &\lesssim t^{\beta - \alpha}
        \Big(\E \Big\|\Big(\sum_{n =1}^\infty \gamma_n Z(h_n) \Big)\Big\|_{\B_{\beta,r}}^2\Big)^{\frac{1}{2}}  
        \lesssim t^{\beta - \alpha}\|Z\|_{\gamma(\U,\B_\beta)} \,.
    \end{aligned}
    \end{equation}   
 Then use  \cite[Corollary 3.10]{umd-integration} and 
   \eqref{semigrp-bound-Z} to obtain
    \begin{align*}
        \E\Big[\Big\|\int_0^t S(t-s)Z_s dW_t\Big\|^p_\alpha\Big] &\lesssim \E\|S(t-s)Z_s\|^p_{L^2((0,T);\gamma(\U,\B_\alpha))} \\
        &= \E\Big[\Big(\int_0^t \|S(t-s)Z_s\|^2_{\gamma(\U,\B_\alpha)} dt\Big)^{p/2}\Big] \\
        &\lesssim \E\Big[\Big(\int_0^t f_{-1 + 2\epsilon}(t-s)\|Z_s\|^2_{\gamma(\U,\B_{\alpha-1/2+\epsilon})} dt\Big)^{p/2}\Big] \,,
    \end{align*}
and consequently Young's inequality yields
    \begin{align*}
        \E\Big[\int_0^T \|(S \diamond Z)_t\|^p_\alpha dt\Big] &\lesssim \E\Big[\Big\|f_{-1+2\epsilon}(t) * \|Z_t\|^2_{\gamma(\U,\B_{\alpha-1/2+\epsilon})}\Big\|_{L^{p/2}(0,T)}^{p/2}\Big] \\
        &\leq \Big\|f_{-1+2\epsilon}\Big\|_{L^{1}(0,T)}^{p/2}\E\Big[\int_0^T \|Z_t\|^p_{\gamma(\U,\B_{\alpha - 1/2 + \epsilon})} dt\Big] \,.
    \end{align*}
\end{proof}

For our special case of a positivity-preserving semigroup on $L^r$, \Cref{convolution-bounds} gives a bound on the $L^q$ norm for $q > r$ given a bound on the $L^r$ norm.

\begin{cor}\label{lq-to-lr}
    Fix $T > 0,p > 1$, $q \in (2,\infty]$ and suppose that $u_t,y_t^+,y_t^- \in L^p(\Omega;C([0,T];\cdom))$ and there is $0 < \eta^* \leq 1/2$ such that, for all $r \in [2,q)$, $Z_t \in L^p(\Omega \times [0,T]; \gamma(\U,\B_{\eta^*-1/2,r}))$. With $y_t = y^+_t - y_t^-$, suppose
    $$u_t = S(t)u_0 + (S * y)_t + (S \diamond Z)_t \,.$$
    Then for all $0 < \eta \leq \eta^*  $ there is $r \in [2,q)$ such that
    $$
    \E\Big[\int_0^T \|u_t\|_{L^q}^pdt\Big] \lesssim \|u_0\|_{L^r}^p + \E\Big[\int_0^T \|y^+_t\|_{\eta-1,r}^pdt\Big] + \E\Big[\int_0^T \|Z_t\|_{\gamma(\U,\B_{\eta-1/2,r})}^p dt\Big] \,.
    $$
    
    Specifically, we may choose 
    \begin{equation}\label{eq:lborc}
    \frac{1}{r} = \Big(\frac{\eta\wedge \frac{1}{p}}{N} + \frac{1}{q}\Big) \wedge \frac{1}{2} \,.
    \end{equation}
    Furthermore, the $\lesssim$ only depends on $p,q,\eta,r$ (and $T$, but the constant is uniform over all $T$ in any bounded subset of $(0,\infty)$).
\end{cor}

\begin{proof}
    First note that $(S * y)_t = (S * y^+)_t - (S * y^-)_t$. Since $S(t)$ preserves positivity (\Cref{sass-A}), $(S * y^-)_t \geq 0$. The assumed $u_t \geq 0$ implies
    \begin{equation}\label{eq:lq-to-lr}
        \|u_t\|_{L^q} \leq \|u_t + (S * y^-)_t\|_{L^q} \leq \|S(t)u_0\|_{L^q} + \|(S * y^+)_t\|_{L^q} + \|(S \diamond Z)_t\|_{L^q} \,.
    \end{equation}
    Then by Sobolev embedding there is $r \in [2,q)$ and $\epsilon > 0$ such that $W^{2\eta-2\epsilon,r} \subset L^q$ and 
    $L^r \subset W^{-2/p + 2\epsilon, q}$ (which is valid if $1/q > 1/r - 2\eta/N$, $1/r < 1/q + 2/(Np) $ and $\epsilon > 0$ is sufficiently small, that is, if for example \eqref{eq:lborc} holds), and \Cref{characterization-of-interpolation-spaces} implies $\B_{\eta - \epsilon,r} \subset L^q$ and $L^r \subset \B_{-1/p + \epsilon,q}$. Thus,

\begin{equation*}
        \|u_t\|_{L^q} \lesssim \|S(t)u_0\|_{q} + 
        \|(S * y^+)_t\|_{\eta-\epsilon,r} + \|(S \diamond Z)_t\|_{\eta-\epsilon,r} 
    \end{equation*}
and by \Cref{convolution-bounds}
\begin{align*}
 \E\int_0^T\|S(t)u_0\|_{q}^p dt \lesssim \|u_0\|^p_{-1/p+\epsilon, q} \lesssim \|u_0\|^p_{L^r} 
 \end{align*}
 and
 \begin{multline*}
 \E\int_0^T\|(S * y^+)_t\|_{\eta-\epsilon,r}^p dt+ \E\int_0^T\|(S \diamond Z)_t\|_{\eta-\epsilon,r}^p dt \\
 \lesssim
 \E\int_0^T\|y^+_t\|_{\eta-1,r}^p dt+ \E\int_0^T\| Z_t\|_{\gamma(\U,\B_{\eta-1/2,r})}^p dt    
 \end{multline*}
as desired.  
\end{proof}

\begin{lem}\label{convolution-sup-bounds}
    Fix $2 \leq r < \infty$ and $\B \coloneqq L^r$. For all $\alpha \in \R$, $\epsilon \in (0,1/2)$, $p > \epsilon^{-1}$, $T > 0$ the following hold:
    \begin{enumerate}[label=(\roman*)]
        \item For $y \in L^p([0,T]; \B_{\alpha - 1 + \epsilon})$,
        $$[0,T] \ni t \mapsto (S*y)_t \in \B_\alpha$$
        is continuous and
        $$
        \sup_{t \leq T} \|(S * y)_t\|^p_\alpha \lesssim \int_0^T \|y_t\|^p_{\alpha - 1 + \epsilon} dt \,.$$
        \item For $Z \in L^p(\Omega \times [0,T];\gamma(\U, \B_{\alpha - 1/2 + \epsilon}))$,
        $$[0,T] \ni t \mapsto (S\diamond Z)_t \in \B_\alpha$$
        is continuous and
        $$\E\Big[\sup_{t \leq T} \|(S \diamond Z)_t\|^p_\alpha \Big] \lesssim \E\Big[\int_0^T \|Z_t\|^p_{\gamma(\U,\B_{\alpha - 1/2 + \epsilon})} dt\Big] \,.$$
    \end{enumerate}
    In the above, the constants in $\lesssim$'s do not depend on $x,y,Z$, and the dependence on $T$ is uniform over bounded subsets of $(0,\infty)$.
\end{lem}
\begin{proof}
    The first claim is a direct corollary of \cite[Lemma 3.6]{umd-existence}. 

    For the second claim, fix $\beta \in ( 1/2 + 1/p - \epsilon,  1/2) $ and roughly follow \cite[Pages 1057-1058]{feller-stoch-reac-diff}:
    \begin{align*}
    \E\Big[\sup_{t \leq T} \|(S \diamond Z)_t\|^p_\alpha \Big] &\lesssim \int_0^T \E\Big[\|s \mapsto (t-s)^{-\beta}Z_s\|^p_{\gamma(L^2((0,T);\U),\B_{\alpha-1/2+\epsilon})}\Big]dt \\
    &\lesssim
    \int_0^T  \E\Big[\|s \mapsto (t-s)^{-\beta} Z_s\|^p_{L^2((0,T); \gamma(\U,\B_{\alpha-1/2+\epsilon}))}\Big]dt \\
    &= 
    \E \int_0^T  \Big[ \int_0^t (t-s)^{-2\beta}\| Z_s \|_{\gamma(\U,\B_{\alpha-1/2+\epsilon})}^2 ds \Big]^{\frac{p}{2}} dt
    \\
    &\lesssim \Big(\int_0^T (t-s)^{-2\beta}\Big)^{\frac{p}{2}}
    \E\Big[\int_0^T \|Z_t\|^p_{\gamma(\U,\B_{\alpha - 1/2 + \epsilon})} dt\Big] \,,
    \end{align*}
    where in the first inequality we used \cite[Proposition 4.2]{umd-existence}, in the second one the embedding $L^2((0,T); \gamma(\U,\B_{\alpha-1/2+\epsilon})) \subset \gamma(L^2((0,T);\U),\B_{\alpha-1/2+\epsilon})$ (this is valid because $\B_{\alpha-1/2+\epsilon}$ has type 2, see \cite[Page 24, the discussion before Corollary 3.10]{umd-integration} and the references cited within),
    
    and in the third one Young's  convolution inequality.
\end{proof}

\subsection{Well-posedness and It\^{o} Formulas}\label{sec:spde-well-posed}
Next we define precisely what we mean by a solution $x_t = (u_t,z_t)$ to \eqref{SPDE} and list assumptions which imply that the problem is well-posed and positivity preserving. We also prove these properties for the crucial ``projective process" 
 \begin{equation}\label{eq:proj-process-first-def}
 (r_t,v_t,z_t) = 
 (\|u_t\|_{L^1}, u_t/\|u_t\|_{L^1},z_t) \,.
\end{equation}

Since  we work with mild as opposed to what we call ``strong" solutions, we derive nontrivial It\^{o}/energy formulas for the evolution of certain norms. Finally, we provide sufficient conditions under which \Cref{as1} and \Cref{as2} hold.

Next, we introduce different strong and mild solutions to 
\begin{equation}\label{eq:gefx}
   dx_t = [Ax_t + F(x_t)]dt + G(x_t)dW_t \,. 
\end{equation}
For technical reasons we must allow local solutions as well, but we will omit the statement ``up to time $\tau$" when $\tau = \infty$.

\begin{deff}\label{def:strong}
Fix $r \in [2,\infty), \theta \in \R$, a stopping time $\tau$, and assume $x \in L^p(\Omega;C([0,T];B))$. 
   We say a process $x_t$ is a strong solution in $\B_{\theta,r}$ to \eqref{eq:gefx} up to time $\tau$
    if for all $T > 0$:
    \begin{enumerate}[label=(\roman*)]
        \item $\E[\int_0^{T\wedge \tau} \|x_t\|_{1+\theta, r} dt] < \infty$
        \item $F: B \to \B_{\theta, r}$ and $\E[\int_0^{T\wedge \tau} \|F(x_t)\|_{\theta,r} dt] < \infty$
        \item $G: B \to \gamma(\U,\B_{\theta, r})$ and $\E[\int_0^{T\wedge\tau} \|G(x_t)\|_{\theta,r}^2dt] < \infty$
    \end{enumerate}
    (see \Cref{fractional-domains-deff} for the definition of the norms) and furthermore
    \begin{equation}\label{eq:seqcn}
    x_{t \wedge \tau} = x_0 + \int_0^{t \wedge \tau} Ax_s + F(x_s)ds + \int_0^{t \wedge \tau} G(x_s)dW_s \,.    
    \end{equation}
\end{deff}

\begin{deff}\label{def:mild}
Fix numbers $r \in [2,\infty), \theta \in \R$, functions $F: B \to \B_{-1+\theta, r}$, $G: B \to \gamma(\U,\B_{-1/2+\theta, r})$, a stopping time $\tau$, and assume $x \in L^p(\Omega;C([0,T];B))$.  
     A process $x_t$ is a mild solution in $\B_{\theta,r}$ to \eqref{eq:gefx} up to time $\tau$

    if there is some $\epsilon > 0$ such that for all $T > 0$
    \begin{equation}\label{eq:int-condition-mild}
        \E\Big[\int_0^{T \wedge \tau} \|F(x_t)\|_{-1+\theta + \epsilon, r} + \|G(x_t)\|_{-1/2+\theta +\epsilon}^2dt\Big] < \infty \,,
    \end{equation}
    (see \Cref{fractional-domains-deff} for the definition of the norms)  
    and furthermore
    \begin{equation}\label{eq:mild-sol-def}
        x_t = S(t)x_0 +(S*F(x))_t + (S \diamond G(x))_t
    \end{equation}
    almost surely on the event $\{t \leq \tau\}$.
\end{deff}

\begin{rem}\label{rem:convolution-bounds-hold-locally}
    Note that \Cref{convolution-bounds} still holds for local mild solutions. For example, in the setting of \Cref{convolution-bounds}(ii) we have
    \begin{align*}
        \int_0^{T\wedge \tau} \|(S * y)_t\|_\alpha^pdt &= \int_0^{T \wedge \tau} \|(S * (y_\cdot1_{\cdot \leq \tau}))_t\|_\alpha^pdt \leq \int_0^T \|(S * (y_\cdot1_{\cdot \leq \tau}))_t\|_\alpha^pdt \\
        &\lesssim \int_0^{T \wedge \tau} \|y_t\|_{\alpha-1+\epsilon}dt \,,
    \end{align*}
    where \Cref{convolution-bounds}(ii) is applied for the last inequality. The analogous statement holds for \Cref{convolution-bounds}(iii) by the same reasoning.
\end{rem}

The following provides a relationship between strong and mild solutions:

\begin{lem}\label{mild-strong}
Fix $r \in [2,\infty), \theta \in \R$, $F: B \to \B_{\theta, r}$, $G: B \to \gamma(\U,\B_{1/2+\theta, r})$, a stopping time $\tau$, and assume $x \in L^p(\Omega;C([0,T];B))$.
    If $x_t$ is a mild solution in $\B_{\theta,r}$ of \eqref{eq:gefx} up to time $\tau$ and there is some $\epsilon > 0$ such that
    \begin{equation}\label{eq:mild-vs-strong}
        \|x_0\|_{\theta + \epsilon,r} + \E\Big[\int_0^{T \wedge \tau} \|F(x_t)\|_{\theta + \epsilon,r} + \|G(x_t)\|_{\theta + 1/2+\epsilon, r}^2dt\Big] < \infty \,,
    \end{equation}
    then $x_t$ is a strong solution in $\B_{\theta,r}$ up to time $\tau$. Conversely, if $x_t$ is a strong solution in $\B_{\theta,r}$ up to time $\tau$ and \eqref{eq:mild-vs-strong} holds, then $x_t$ is a mild solution in $\B_{\theta+1,r}$ up to time $\tau$.
\end{lem}

\begin{proof}
First, assume that $x_t$ is a mild solution. 
    The conditions (ii) and (iii) in \Cref{def:strong} on $F$ and $G$ are clearly satisfied by \eqref{eq:mild-vs-strong}. The condition (i) is a consequence of \Cref{convolution-bounds} and \eqref{eq:seqcn} follows from (stochastic) Fubini's theorem as in \cite[Proposition 1.3.5]{mild-strong}.

Second, assume that  $x_t$ is a strong solution.  Since \eqref{eq:mild-vs-strong} holds in $\B_{\theta, r}$, then \eqref{eq:int-condition-mild} holds with $\theta$ replaced by $\theta+1$. Then use \cite[Corollary 2.6]{ito-formula} on $(S(T-t)^*l)(x_t)$, where $l \in \Dm(A^\dagger)$ (see \cite[Proposition 6.7]{hairer-notes} for the relevant details) to obtain \eqref{eq:mild-sol-def}.
\end{proof}

As expected, if $F,G$ have some sort of global Lipschitz property then the problem \eqref{eq:gefx} is well-posed:

\begin{lem}\label{general-mild-existence}
    
    Recall $B = C(\overline D)$ and set $\B = L^r$ for some $r > N$
    (set $\theta = 0$ in comparison to \Cref{def:strong}-\ref{def:mild}).  Fix $\eta \in (N/(2r), 1/2)$ and $\theta_F, \theta_G \leq 0$ that satisfy $\theta_F > \eta - 1$, $\theta_G > \eta - 1/2$. If $F: B \to \B_{\theta_F}$ and $G: B \to \gamma(\U,\B_{\theta_G})$ are globally Lipschitz, then for all $p,T > 0$ the following statements hold: 
    \begin{enumerate}[label=(\roman*)]
        \item The problem \eqref{eq:gefx} has a unique mild solution $x_t \in L^p(\Omega; C([0,T];B))$ (\Cref{def:mild}) for any initial condition $x_0 \in B$.
        \item If $F_n: B \to \B_{\theta_F}$ and $G_n: B \to \gamma(\U,\B_{\theta_G})$ are globally Lipschitz for each $n \geq 1$, $F_n \to F, G_n \to G$ pointwise, and $x_0^{(n)} \to x_0$ in $B$, then the solutions $x_t^{(n)}$ converge to $x_t$ in $L^p(\Omega; C([0,T];B))$.
    \end{enumerate}
\end{lem}

\begin{proof}
    The statements are consequences of \cite[Proposition 3.8]{feller-stoch-reac-diff}. Note that \cite[(A5)]{feller-stoch-reac-diff} follows from the Sobolev embedding $W^{2\eta, r} \subset C(\overline D)$ and \Cref{characterization-of-interpolation-spaces}. Furthermore \cite[(A4)]{feller-stoch-reac-diff} follows from \cite[Exercise 5.2]{hairer-notes}; indeed, by \cite[(A5)]{feller-stoch-reac-diff} $S(t)x \to x$ as $t \downarrow 0$ for all $x \in \B_{\eta,r}$, and we have $\|S(t)\|_{\Ll(B,B)} \leq 1$ (see \Cref{rem:subtract-from-c}).
\end{proof}

Often we would rather work with strong solutions in order to use It\^{o}'s formula \cite[Theorem 2.4]{ito-formula}. This is done using the following approximation scheme:

\begin{lem}\label{def:yosida}
    Consider the notation and assumptions from \Cref{general-mild-existence}. For each $n \geq 1$ define
    $$
    R_n \coloneqq n(n - A)^{-1} = \int_0^\infty ne^{-nt}S(t)dt 
    $$
    
    and $F_n \coloneqq R_nF$,  $G_n \coloneqq R_nG$. 
    Then, there is a unique strong solution $x_n(t)$ (called Yosida approximations of \eqref{eq:gefx})
    of the problem
    \begin{equation}\label{eq:yapxn}
    dx_n(t) = [Ax_n(t) + F_n(x_n(t))]dt + G_n(x_n(t))dW_t \,, \quad x_n(0) = R_n x_0 \,.
    \end{equation}
    Furthermore, $(x_n)_{n \geq 1}$ converge as $n \to \infty$ to the mild solution $x$
    of \eqref{eq:gefx} in $L^p(\Omega; C([0,T];B))$ for all $p,T > 0$.
\end{lem}

\begin{proof}
    Since the proof follows from  standard properties of the resolvent listed below, we leave details to the reader. 
    \begin{enumerate}[label = (\roman*)]
    
    \item From \eqref{semigrp-bound} follows $\|R_n\|_{\Ll(\B_\theta,\B_\theta)} \leq 1$ for any $\theta \in \R$. Thus, \Cref{general-mild-existence} still applies with $F$ and $G$ replaced respectively by $F_n$ and $G_n$. Hence, there is a unique mild  solution $x_n$ of \eqref{eq:yapxn}.
        
        \item By combining \cite[Exercise 5.42]{hairer-notes} 
        and \cite[Proposition 2.1.11]{lunardi}
        we obtain that $S(t)$ is an analytic semigroup on $\B_\theta$
        for any $\theta \in \R$. In paticular, $\|S(t)x - x\|_{\B_\theta} \to 0$ as $t \to 0^+$, and therefore 
         $R_mx \to x$ in $\B_\theta$ for any $\theta \in \R$ for all $x \in \B_\theta$. Thus, $F_m \to F$ as $m \to \infty$ point-wise.
        
         \item Since $R_m - I = AR_m$, then \eqref{semigrp-bound} implies $R_m \to I$ in $\Ll(\B_\theta,\B_{\theta-\epsilon})$ for all $\epsilon > 0$. Hence, $G_m \to G$ point-wise (after slightly decreasing $\theta_G$), and after combining with (ii), \Cref{general-mild-existence} gives the convergence of $x_m$ to $x$.  
        
        \item Note that $R_m \in \Ll(\B_{\theta},\B_{\theta+a})$ for all $a < 1$, which implies \eqref{eq:mild-vs-strong}. We conclude by \Cref{mild-strong} that $x_m$ is a strong solution.
    \end{enumerate}
\end{proof}

As previously mentioned, we would like to apply It\^{o}'s formula to the Yosida approximations. The following lemma shows that as $n \to \infty$ \eqref{eq:yapxn} certain terms coming from It\^{o}'s formula converge as expected:

\begin{lem}\label{lem:ito-convergence}
    Consider the setup of \Cref{general-mild-existence} and \Cref{def:yosida}, and suppose $\theta_F = \theta_G = 0$. Let $\Phi: \B \to \R$ be twice continuously differentiable with linearly bounded first and second derivatives: $\|\Phi'(x)\|_{\Ll(\B,\R)} + \|\Phi''(x)\|_{\Ll(\B,\Ll(\B,\R))} \lesssim 1 + \|x\|_B$. Then for all $T > 0$, $p \geq 2$:
    \begin{itemize}
        \item $\lim_{n \to \infty} \E\Big[\int_0^T |\Phi'(x_n(t))F_n(x_n(t)) - \Phi'(x(t))F(x(t))|^pdt\Big] = 0$.
        \item $\lim_{n \to \infty} \E\Big[\int_0^T \|\Phi'(x_n(t))G_n(x_n(t)) - \Phi'(x(t))G(x(t))\|_{\U^*}^pdt\Big] = 0$.
        \item $\lim_{n \to \infty} \E\Big[\int_0^T |\operatorname{tr}_{G_n(x_n(t))}\Phi''(x_n(t)) - \operatorname{tr}_{G(x(t))}\Phi''(x(t))|^pdt\Big] = 0$. (For any operator $T$ on $\U$ we define $\tr_{T}\Phi''(x) = \sum_{j \geq 1} \Phi''(x)[Tu_j, Tu_j]$ where $(u_j)$ is an orthonormal basis for $\U$.)
    \end{itemize}
    Thus,
    \begin{align*}
        \int_0^t \Phi'(x_n F_n(x_n))ds + \int_0^t \Phi'(x_n G_n(x_n)) dW_s + \frac{1}{2}\int_0^t \operatorname{tr}_{G_n(x_n)}\Phi''(x_n)ds \\
        \to \int_0^t \Phi'(x)F(x)ds + \int_0^t \Phi'(x)G(x)dW_s + \frac{1}{2}\int_0^t \operatorname{tr}_{G(x)}\Phi''(x)ds
    \end{align*}
    in $L^p(\Omega; C([0,T], \R))$, where we omitted the dependence of $x_n$ and $x$ on $s$.
\end{lem}

\begin{proof}
We start by proving the first bullet point. Since each $R_n$ is a contraction in $B$ and $\B$, and $F: B \to \B$ is globally Lipschitz,
\begin{multline*}
    \sup_n\E \int_0^T |\Phi'(x_n(t))F_n(x_n(t))|^{2p}dt \leq \sup_n\E \int_0^T \|\Phi'(x_n(t))\|_{\Ll(\B,\R)}^{2p}|F(x_n(t))|_B^{2p}dt \\
    \begin{aligned}
    &\lesssim \sup_n\E\int_0^T \|\Phi'(x_n(t))\|_{\Ll(\B,\R)}^{4p} + |x_n(t)|_B^{4p} + |x_0|_B^{4p}dt \\
    &\lesssim \sup_n \|x_n\|_{C([0,T], B)}^{4p} < \infty \,,
\end{aligned}
\end{multline*}
where the last inequalities are from our assumption on $\Phi'$ and the convergence of $x_n$ to $x$ in $L^{4p}(\Omega; C([0,T],B))$ (\Cref{def:yosida}). Thus, $\Phi'(x_n)F_n(x_n)$ is uniformly bounded in $L^{2p}(\Omega \times [0,T])$. To show convergence $\Phi'(x_n)F_n(x_n) \to \Phi'(x)F(x)$ in $L^p(\Omega \times [0,T])$, it suffices to show convergence in measure. Thus, we fix a subsequence of $n$'s and show that there is a further subsequence which converges almost everywhere (this is an equivalent statement). By the convergence of $x_n$ to $x$ in $L^{4p}(\Omega; C([0,T],B))$ (\Cref{def:yosida}) we may WLOG assume that $x_n \to x$ almost everywhere.

Then the first bullet point follows from
\begin{align*}
    |\Phi'(x_n(t))F_n(x_n(t)) - \Phi'(x(t))F(x(t))| &\leq \|\Phi'(x_n(t)) - \Phi'(x(t))\|_{\Ll(\B,R)}\|F(x_n(t))\|_\B \\
    &+ \|\Phi'(x(t))\|_{\Ll(\B,\R)}\|F_n(x_n(t)) - F_n(x(t))\|_\B \\
    &+ \|\Phi'(x(t))\|_{\Ll(\B,\R)}\|F_n(x(t)) - F(x(t))\|_\B \\
    &\to 0 \,,
\end{align*}
where the first term goes to $0$ by continuous differentiability of $\Phi$, the second by $F(x_n(t)) \to F(x(t))$ in $\B$ and $R_n$ being a contraction (recall $F_n = R_nF$), and the third by $F_n \to F$ point-wise (see the proof of \Cref{def:yosida}).

The other bullet points follow similarly because (see \Cref{rmk:gropp})
\begin{align*}
\|\Phi'(x(t))G(x(t))\|_{\U^*} &\leq \|\Phi'(x(t))\|_{\Ll(\B,\R)}\|G(x(t))\|_{\Ll(\U,\B)} \\
&\leq \|\Phi'(x(t))\|_{\Ll(\B,\R)}\|G(x(t))\|_{\gamma(\U,\B)}
\end{align*}
 and
$$|\operatorname{tr}_{G(x(t))}\Phi''(x(t))| \leq \|\Phi''(x(t))\|_{\Ll(\B,\Ll(\B,\R))}\|G(x(t))\|_{\gamma(\U,\B)}^2$$
(\cite[Lemma 2.3]{ito-formula}). The final claim follows from BDG inequality.
\end{proof}
Next,  we describe our assumptions on $F$ and $G$. 
Recall that $d$ was fixed at the beginning of the section and $\pi_u$
is the projection on the first $d$ coordinates. 

\begin{rem}
    In the assumptions below, the constants in every $\lesssim$ are allowed to depend on $r$.
\end{rem}

\begin{sass}\label{SPDE-F}
    The function $F: B \times \Theta \to B$ is given by
    $$F(x,\theta)(\cdot) = x(\cdot)\cdot f_1(\cdot,x(\cdot),\theta) + f_2(\cdot,x(\cdot),\theta) =: f(\cdot,x(\cdot),\theta)$$
    for some locally Lipschitz functions $f_1,f_2: \overline D \times \R^m \times \Theta \to \R^m$. We assume:
    \begin{itemize}
        \item $f_2(\cdot,x,\theta) \geq 0$ if $x \geq 0$. 
        \item $f$ is polynomially bounded: there is $k \geq 0$ such that $|f(\cdot,x,\cdot)| \lesssim 1 + |x|^k$.
        \item $x_u = 0$ implies $\pi_u(f_2(\cdot, x, \theta)) = 0$, 
        and therefore $\pi_u(f(\cdot, x, \theta)) = 0$.
        
        \item $f_1,f_2$ are differentiable in $x_1,\dots,x_d$ with locally Lipschitz derivatives which are polynomially bounded.
        
        \item $f^+ \coloneqq f \vee 0$ satisfies $|f^+_u(\cdot,x,\cdot)| \lesssim |x_u|$ and $|f^+_z(\cdot,x,\cdot)| \lesssim 1 + |x|$.
    \end{itemize}
\end{sass}

In the following assumption on the noise we allow the sequence $(\|e_n\|_{L^\infty})_{n\geq 1}$ (defined below) to be unbounded, unlike in 
 \cite[Hypothesis 8.1]{stochastic-lotka-volterra-SPDE} and  \cite[Hypothesis 1-3]{Cerrai2003}. 
 We remark that $\|e_n\|_{L^\infty} \leq P(n)$ for some polynomial $P$ (see \cite[Remark 2.2]{Cerrai2003}).

\begin{sass}\label{SPDE-G-sass}
    The function $G: B \times \Theta \to \Ll(\U,L^2)$ is given by
    $$G(x,\theta) = \sigma(\cdot,x(\cdot),\theta) \cdot H$$
    for a function $\sigma: \overline D \times \R^m \times \Theta \to \R^m$ and an operator $H \in \Ll(\U,L^2)$. We assume:
    \begin{itemize}
        \item $\sigma$ is differentiable in $x$ with bounded locally Lipschitz derivative (jointly in all variables).
        
        \item $x_i = 0$ implies $\sigma_i(\cdot,x,\theta) = 0$.
        \item There is an (orthonormal) basis $\{u_n\}_{n \in \N}$ for $\U$ such that $Hu_n = a_ne_n$, where $a_n \in \R$ and $e_n$ are the eigenfunctions of $A$ (see \Cref{sass-A}) normalized with the unit $L^2$ norm.
        \item If $N > 1$ then there is some $2 \leq p < 2N/(N-2)$ such that
        \begin{equation}\label{eq:an-en-assumption}
            \sum_{n \geq 1} a_n^p\|e_n\|_{L^\infty}^2 < \infty \,,
        \end{equation}
        and if $N = 1$, then $\sup_n |a_n| < \infty$ (and in this case we say $p = \infty$).
    \end{itemize}
\end{sass}

\begin{lem}\label{SPDE-G-lem}
    Under \Cref{SPDE-G-sass}, we have the following assertions:
    \begin{enumerate}[label=(\roman*)]
        \item There is $\beta \in (-1/2,0]$ such that $\|x \cdot H\|_{\beta,r} \lesssim \|x\|_{L^r}$ for all $r \in [2,\infty)$ and $x \in L^r$. In particular, the estimate is valid for any $\beta \in \big(-\frac{1}{2},  -N(\frac{1}{4} - \frac{1}{2p})\big)$, and additionally for $\beta = 0$ if \eqref{eq:an-en-assumption} holds with $p = 2$.
        \item For all $r > 2$ there is $\epsilon > 0$ such that the following holds for all $x,y \in \{x \in \cdom \mid \|x\|_{L^1} \leq 1\}$:

        \begin{equation}\label{eq:wierd-g}
        \begin{aligned}
        \|H^*(x)\|_\U^2\|y\|_{\epsilon-1,r} + \|y\cdot H H^*(x)\|_{\epsilon-1,r} +\|H^*(x)\|_\U\|y\|_{\epsilon-1/2,r} \\
        \lesssim \|x\|_{L^2}^2 + \|x\|_{L^r} + \|y\|_{L^2}^2 + \|y\|_{L^r}\,,
        \end{aligned}
        \end{equation}
         where the adjoint of $H$ is taken in $\Ll(\U,L^2)$.

         \item With $\beta$ as in (i),
\begin{equation}\label{eq:g-star-1}
    \|G(x)^*1\|_\U\|x\|_{\beta,2} \lesssim \|x\|_{\beta+1/2,2}\|x\|_{L^1} \,,
\end{equation}
where the adjoint of $G(x)$ taken in $\Ll(\U,L^2)$.
    \end{enumerate}
\end{lem}

\begin{proof}
In what follows assume $p > 2$ (we leave the case $p = 2$ to the reader as it is quite similar but simpler).

We start by showing $\|x \cdot H\|_{\beta,r} \lesssim \|x\|_{L^r}$ for all $r \in [2,\infty)$, $x \in L^r$, and $\beta < -N(\frac{1}{4} - \frac{1}{2p})$. This will prove (i) because by \Cref{SPDE-G-sass} $-N(\frac{1}{4} - \frac{1}{2p}) > -1/2$, so such a $\beta \in (-1/2,0]$ exists. First notice that 
for any $T \in \Ll(\U,L^2)$ and $\beta \in (-1,0)$ we have
\begin{equation}\label{eq:t-neg-power-bound}
      \begin{aligned}
        \|T\|_{\gamma(\U,\B_{\beta,r})} &= \|(-A)^{\beta}T\|_{\gamma(\U,L^r)} = \Big\|\frac{1}{\Gamma(-\beta)}\int_0^\infty t^{-\beta-1}S(t)Tdt\Big\|_{\gamma(\U,L^r)} \\
        &\leq \frac{1}{\Gamma(-\beta)}\int_0^\infty t^{-\beta-1}\|S(t)T\|_{\gamma(\U,L^r)}dt \,,
    \end{aligned}
\end{equation}
  
    which is finite if $\|S(t)T\|_{\gamma(\U,L^r)} \lesssim t^{-\alpha}$ for some $\alpha < -\beta$ (see \eqref{eq:doapt} for the second equality).

    By \cite[Lemma 2.1]{umd-existence} we have 
    \begin{equation}\label{eq:eosaxa}
    \|S(t) \circ (x \cdot H)\|_{\gamma(\U,L^r)} \lesssim \Big\|\Big(\sum_{n \geq 1} a_n^2|S(t)(xe_n)|^2\Big)^{1/2}\Big\|_{L^r} \,.    
    \end{equation}
    Let $q$ be such that $1/q + 1/p = 1/2$. Then by H\"{o}lder's inequality and \eqref{eq:an-en-assumption} we have
    \begin{multline}\label{eq:sdmlf}
    \Big(\sum_{n \geq 1} a_n^2|S(t)(xe_n)|^2\Big)^{\frac{1}{2}} \leq \Big(\sum_{n \geq 1}a_n^p\|e_n\|_{L^\infty}^2\Big)^{\frac{1}{p}}\Big(\sum_{n \geq 1}|S(t)(xe_n)|^q\|e_n\|_{L^\infty}^{-\frac{2q}{p}}\Big)^{\frac{1}{q}} \\
       \begin{aligned}
                &\lesssim \Big(\sup_{n \geq 1} |S(t)(xe_n\|e_n\|_{L^\infty}^{-1})|^{q-2} \Big)^{\frac{1}{q}}\Big(\sum_{n \geq 1}|S(t)(xe_n)|^2\|e_n\|_{L^\infty}^{-2q/p+q-2}\Big)^{\frac{1}{q}} \\
        &\leq (S(t)|x|)^{1 - \frac{2}{q}}\Big(\sum_{n \geq 1}|S(t)(xe_n)|^2\Big)^{\frac{1}{q}} \,,
    \end{aligned} 
    \end{multline}
where in the last inequality we used the facts that $S(t)$ is positivity preserving (\Cref{sass-A}) and $-2q/p +q-2 = q(1-2/p)-2 = 0$.

To proceed, we use \cite[(2.7),(2.8)]{Cerrai2003} (see also the references within) to express $S(t)$ as an integral operator with gaussian upper bounds. In other words, there is $K: (0,\infty) \times D \times D \to \R^m$ and $c > 0$ such that
$$[S(t)x](u) = \int_D K(t,u,v) \cdot x(v)d\mu(v)$$
and
\begin{equation}\label{eq:gaussian-bounds}
    0 \leq K(t,u,v) \lesssim t^{-N/2}\exp\Big(-c\frac{|u-v|^2}{t}\Big) \,.
\end{equation}
Then since $e_n$ is an orthonormal basis for $L^2$ (which is a consequence of $A$ being self-adjoint in \Cref{sass-A}) we have for $u \in D$ that
\begin{align*}
\sum_{n \geq 1}|S(t)(xe_n)(u)|^2 &= \sum_{n \geq 1} \Big(\int_D K(t,u,v)x(v)e_n(v)d\mu(v)\Big)^2 
\\
&= \int_D K(t,u,v)^2x(v)^2d\mu(v) \,.
\end{align*}
Combining with \eqref{eq:eosaxa} and \eqref{eq:sdmlf} yields
$$
\|S(t) \circ (x \cdot H)\|_{\gamma(\U,L^r)} \lesssim \Big(\int_D [S(t)|x|(u)]^{r(1-\frac{2}{q})}\Big[\int_D K(t,x,y)^2x(v)^2d\mu(v)\Big]^{\frac{r}{q}}d\mu(u)\Big)^{\frac{1}{r}}.$$

By \eqref{eq:gaussian-bounds} we have
$$
K(t,u,v)^2 \lesssim t^{-N/2}k(t,u,v) \,,
$$
where $k$ is the kernel of a heat semigroup $\Pp_t: C(\R^N;\R^m) \to C(\R^N;\R^m)$ (meaning $\Pp_t f$ solves $\frac{\partial}{\partial t}\Pp_t f = c'\Delta \Pp_tf$ for some $c' > 0$). Thus,
$$
\|S(t) \circ (x \cdot H)\|_{\gamma(\U,L^r)} \lesssim t^{-\frac{N}{2q}}\Big(\int_D [S(t)|x|(u)]^{r(1-\frac{2}{q})}[\Pp_t x^2(u)]^{\frac{r}{q}}d\mu(u)\Big)^{\frac{1}{r}} \,.
$$
Using H\"{o}lder's inequality ($2/p + 2/q = 1$) and the fact that $S(t)$ and $\Pp_t$ are contractions in all $L^p$ spaces (for $\Pp_t$ this is well-known, for $S(t)$ see \Cref{rem:subtract-from-c}) we have
\begin{align*}
    \|S(t) \circ (x \cdot H)\|_{\gamma(\U,L^r)} &\lesssim t^{-\frac{N}{2q}}\Big(\Big[\int_D [S(t)|x|(u)]^{r}d\mu(u)\Big]^{\frac{2}{p}}\Big[\int_D[\Pp_t x^2(u)]^{\frac{r}{2}}d\mu(u)\Big]^{\frac{2}{q}}\Big)^{\frac{1}{r}} \\
    &= t^{-\frac{N}{2q}}\|S(t)|x|\|_{L^r}^{2/p}\|\Pp_tx^2\|_{L^{r/2}}^{1/q} \\
    &\leq t^{-\frac{N}{2q}}\|x\|_{L^r}^{2/p}\|x^2\|_{L^{r/2}}^{1/q}\\
    &= t^{-\frac{N}{2q}}\|x\|_{L^r} \,.
\end{align*}
Since
\begin{equation}\label{eq:q-bigger-than-n}
    \frac{N}{2q} = N\Big(\frac{1}{4} - \frac{1}{2p}\Big) < 
    -\beta
    < \frac{1}{2} \,,
\end{equation}
(i) follows from the criterion discussed after \eqref{eq:t-neg-power-bound} with $\alpha = \frac{N}{2q}$.

To prove \eqref{eq:wierd-g}, we use H\" older inequality, \eqref{eq:an-en-assumption}, and $2(q-2)/q = 2-4/q = 4/p$ to obtain
\begin{equation}\label{eq:H-star-x}
    \begin{aligned}
    \|H^*(x)\|_\U^2 &= \sum_{n \geq 1} \ip{x}{a_ne_n}^2 \leq \Big(\sum_{n \geq 1}a_n^p\|e_n\|_{L^\infty}^2\Big)^{2/p}\Big(\sum_{n\geq 1} |\ip{x}{e_n}|^q\|e_n\|_{L^\infty}^{-2q/p}\Big)^{2/q} \\
    &\lesssim \sup_n |\ip{x}{e_n\|e_n\|_{L^\infty}^{-1}}|^{2(q-2)/q}\Big(\sum_{n \geq1} \ip{x}{e_n}^2\Big)^{2/q} \\
    &\leq \|x\|_{L^1}^{4/p}\|x\|_{L^2}^{4/q} \leq \|x\|_{L^2}^{4/q}\,,
\end{aligned}
\end{equation}
for any $\|x\|_{L^1} \leq 1$.

By (i) and a Sobolev embedding (see \Cref{characterization-of-interpolation-spaces}) we have
\begin{equation}\label{eq:feoxh}
\|x\cdot H\|_{\gamma(\U,\B_{\epsilon-1,r})} \lesssim \|x \cdot H\|_{\gamma(\U,\B_{\beta,s})} \lesssim \|x\|_{L^s} \,,
\end{equation}
where $\epsilon$ and $s$ are such that
$$
\frac{1}{r} - \frac{2\epsilon - 2}{N} = \frac{1}{s} - \frac{2\beta}{N}\,, \qquad \beta \geq \epsilon - 1 \,.
$$
Choosing $\beta$ close to $-1/2$ and $\epsilon > 0$ small, we have that \eqref{eq:feoxh} holds for any $s$ such that
\begin{equation}\label{eq:choice-of-s}
\frac{1}{r} + \frac{1}{N} > \frac{1}{s} \,.
\end{equation}
By another Sobolev embedding and making $\epsilon > 0$ smaller if necessary, for $s$ satisfying \eqref{eq:choice-of-s} we also have $\|y\|_{\epsilon-1/2,r} \lesssim \|y\|_{L^s}$.
Additionally, from \Cref{rmk:gropp}, \eqref{eq:feoxh}, and \eqref{eq:H-star-x} follows
\begin{align*}
 \|y\cdot H H^*(x)\|_{\epsilon-1,r} &\leq 
 \|y\cdot H\|_{\mathcal{L}(\U,\B_{\epsilon-1,r})} \|H^*(x)\|_{\U}
 \leq \|y\cdot H\|_{\gamma(\U,\B_{\epsilon-1,r})} \|H^*(x)\|_{\U}
 \\
 &\lesssim \|y\|_{L^s}\|x\|_{L^2}^{2/q} \,.
\end{align*}

Thus, the expression on the left hand side of \eqref{eq:wierd-g} is bounded by
\begin{multline*}
        \|x\|_{L^2}^{4/q}\|y\|_{\epsilon-1,r} + \|x\|_{L^2}^{2/q}(\|x\|_{L^s} + \|y\|_{L^s}) 
        \lesssim \|x\|_{L^2}^2 + \|y\|_{\epsilon-1,r}^{p/2} + \|x\|_{L^s}^a + \|y\|_{L^s}^a \,,
    \end{multline*}
where $1/a + 1/q = 1$. We finish the proof of \eqref{eq:wierd-g} by showing 
\begin{equation}
 \|x\|_{L^s}^a \lesssim 1 + \|x\|_{L^2}^2 + \|x\|_{L^r} \qquad \textrm{and} \quad 
\|y\|_{\epsilon-1,r}^{p/2} \lesssim 1+ \|y\|_{L^2}^2 + \|y\|_{L^r} \,.
\end{equation}
 We only show the details for the former and leave the latter to the reader, as it is extremely similar: just replace $a$ with $p/2$ so that instead of $1/a = 1-1/q$ we have $1/a = 1-2/q$, and also replace \eqref{eq:choice-of-s} by $1/r + 2/N > 1/s$. 
 (Notice that both sides of the final expression below are multiplied by $2$, so the desired inequality still holds.)

We begin with $\|x\|_{L^s}^a \leq \|x\|_{L^2}^{a\theta}\|x\|_{L^r}^{a(1-\theta)}$, where $\theta/2 + (1-\theta)/r = 1/s$. Since $cd \lesssim 1 + c^k + d^l$ if $c,d \geq 0$ and $1/k + 1/l \leq 1$, it  suffices to show that $a\theta/2 + a(1-\theta) \leq 1$. We have the equivalent expressions
\begin{align*}
    1-\frac{\theta}{2} &\leq \frac{1}{a} \\
    1- \frac{1}{2}\Big(\frac{1}{s}-\frac{1}{r}\Big)\Big(\frac{1}{2}-\frac{1}{r}\Big)^{-1} &\leq 1 - \frac{1}{q} \\
    2\Big(\frac{1}{2}- \frac{1}{r}\Big) &\leq q\Big(\frac{1}{s}- \frac{1}{r}\Big) \,.
\end{align*}
The last relation is indeed satisfied since $q > N$ (see \eqref{eq:q-bigger-than-n}) and \eqref{eq:choice-of-s} holds with arbitrary precision, which allows us to make the right hand side larger than $1$ by choosing $\epsilon$ small enough.

Finally, we prove \eqref{eq:g-star-1}. By \eqref{eq:H-star-x}
and \Cref{SPDE-G-sass}
    $$
    \|G^*(x) 1\|_\U 
    = \|H^*(\sigma(x))\|_\U 
    \lesssim
    \|\sigma(x)\|_{L^2}^{2/q}\|\sigma(x)\|_{L^1}^{2/p}
    \lesssim \|x\|_{L^2}^{2/q}\|x\|_{L^1}^{2/p} \,.
    $$
    By Sobolev embedding we have $\|x\|_{-N/4-\epsilon,2} \lesssim \|x\|_{L^1}$ for all $\epsilon > 0$, so we conclude via interpolation that
    \begin{equation}\label{eq:crazy-exponents}
        \|G^*(x)1\|_\U\|x\|_{\beta,2} \lesssim \|x\|_{\beta+1/2,2}^{2(1-a)/q+1-b}\|x\|_{L^1}^{2/p + 2a/q +b}
    \end{equation}
    where $(1-b)(\beta+1/2) + b(-N/4-\epsilon) = \beta$ and $(1-a)(\beta+1/2) + a(-N/4-\epsilon) = 0$. Then
    \begin{align*}
        \frac{2(1-a)}{q} + (1-b) &= \frac{2}{q}\Big(1-\frac{\beta+1/2}{N/4+\epsilon+\beta+1/2}\Big) + \Big(1-\frac{1/2}{N/4+\epsilon+\beta+1/2}\Big)\\
        &= \frac{(2/q)(N/4+\epsilon)}{N/4+\epsilon+\beta+1/2} + \Big(1-\frac{1/2}{N/4+\epsilon+\beta+1/2}\Big) \\
        &= 1+\frac{(2/q)(N+4\epsilon) - 2}{N+4\epsilon+4\beta+2} \\
        &< 1 + \frac{(2/N)(N+4\epsilon)-2}{N+4\epsilon+4\beta+2}\,,
    \end{align*}
    where in the last inequality we used $q>N$ (recall \eqref{eq:q-bigger-than-n}). As $\epsilon \downarrow 0$ the limit of the right hand side is $\leq 1$, so it follows that the left hand side is $\leq 1$ for some small enough $\epsilon > 0$.

    Thus, the exponent of $\|x\|_{\beta+1/2,2}$ is at most $1$. Since the exponents on the right hand side of \eqref{eq:crazy-exponents} add to $2$ and $\|x\|_{L^1} \lesssim \|x\|_{\beta+1/2,2}$, \eqref{eq:g-star-1} follows.
\end{proof}

Next we introduce some notation which appears frequently in the study of the projective process \eqref{eq:proj-process-first-def}.

\begin{rem}\label{rem:heavy-notation}
    The definition below is notationally heavy, so we provide alternate insight; if $x_t$ is an $\R^m$-valued path satisfying $dx/dt = F(x)$ and for any $r > 0$ we make the substitution $v = r^{-1}x_u$, $z = x_z$ (recall $x_u$ and $x_z$ are first $d$ respectively last $m -d$ components of $x$).  Then 
    $\tilde x = (v,z)$ satisfies $d\tilde x/dt = \tilde F(r,\tilde x)$ for $\tilde F$ detailed in \Cref{r-functions} below and properly modified initial condition. If $F$ is sufficiently regular,  it is possible to continuously extend $\tilde{F}$ to $r = 0$.

\end{rem}

Recall that $\pi_u$ and $\pi_z$ (from $\R^m$ to $\R^m$) are projections on the first $d$ respectively last $m-d$ components of a vector in $\R^m$. 

\begin{deff}\label{r-functions}
    If $h: \R^m \to \R^m$ is a continuous function  which is continuously differentiable in $x_u$ and satisfies $\pi_u \circ h \circ \pi_z = 0$, then $\tilde h: [0,\infty) \times \R^m \to \R^m$ is defined as 
     $$\tilde h(r,x) \coloneqq \begin{cases}
        r^{-1}\pi_u(h(rx_u,x_z)) + \pi_z(h(rx_u,x_z))  & r > 0 \\
            \pi_u\big(\sum_{i \leq d}\frac{\partial h}{\partial x_i}(0,x_z)x_i\big) + h(0,x_z) & r = 0 \,.
        \end{cases}$$
    In other words, $\tilde h$ is the unique continuous function such that
    \begin{equation}\label{eq:alt-def-tilde}
        (r\pi_u + \pi_z)(\tilde h(r,x)) = h(rx_u,x_z) \,.
    \end{equation}

    For the functions $F$ and $G$ from \Cref{SPDE-F} and 
    \Cref{SPDE-G-sass} we also denote
    $$\tilde F(r,x,\theta)(\cdot) \coloneqq \tilde f(\cdot, r,x(\cdot),\theta) \,, \quad \tilde G(r,x,\theta) \coloneqq \tilde \sigma(\cdot, r,x(\cdot),\theta) \cdot H \,.$$

\end{deff}

\begin{lem}\label{lem:general-tilde-lipschitz}
Assume that conditions of \Cref{r-functions} hold for some $h:\R^m \to \R^m$, which is  locally Lipschitz and $\partial h/\partial x_i$
are locally  Lipschitz for all $i \leq d$ as well. Then
$\tilde h$ is locally Lipschitz on $[0, \infty) \times \R^m$. 
\end{lem}

\begin{proof}
Denote $\nabla_u$ the gradient with respect to first $d$ variables. 
For any fixed $r, r' > 0$ and $x, x' \in \R^m$ we have 
\begin{align*}
    |\tilde{h}(r, x) - \tilde{h}(r', x')| &\leq 
    |\tilde{h}(r, x) - \tilde{h}(r', x)| + 
    |\tilde{h}(r', x_u, x_z) - \tilde{h}(r', x_u', x_z)| \\
    & \qquad +
    |\tilde{h}(r', x_u', x_z) - \tilde{h}(r', x_u', x_z')|
    \\
    &=: (I) + (II) + (III) \,.
\end{align*}
Let us estimate each term separately. First, 
\begin{align*}
    (I) \leq \Big|\frac{1}{r} \pi_u(h(rx_u, x_z))- \frac{1}{r'} \pi_u(h(r'x_u, x_z)) \Big| + |\pi_z(h(rx_u, x_z)) - \pi_z(h(r'x_u, x_z))| 
\end{align*}
and since $\pi_u h(0, x_z) = 0$, we have for any $j = 1, \ldots, d$ that
\begin{multline*}
\Big|\frac{1}{r} h_j(rx_u, x_z) - \frac{1}{r'} h_j(r'x_u, x_z)\Big| \\ 
\begin{aligned}
&=
\Big|\int_0^1 \nabla_u h_j(srx_u, x_z) x_u ds -
\int_0^1 \nabla_u h_j(sr'x_u, x_z) x_u ds\Big|  \\
&\leq L |r - r'| \,,
\end{aligned}
\end{multline*}
where we used Lipschitz continuity of $\nabla_u h_j$ for each $j$ and $L$ is bounded on bounded sets in $(0, \infty) \times \R^m$. 
In addition, the Lipschitz continuity of $h$ yields
\begin{equation*}
|\pi_z(h(rx_u, x_z)) - \pi_z(h(r'x_u, x_z))| \leq L|r - r'| \,.   
\end{equation*}
To estimate (II), we observe that 
\begin{align*}
 (II)
 = \frac{1}{r'} |h(r'x_u, x_z) - h(r'x_u', x_z)| \leq L|x_u - x_u'|\,,
\end{align*}
where we again used Lipschitz continuity of $h$ and $L$ is bounded on bounded sets in $(0, \infty) \times \R^m$. 

Finally, to obtain (III) we proceed as in the estimate of (I) to have
\begin{align*}
    (III) &= \frac{1}{r'} |h(r'x_u', x_z) - h(r'x_u', x_z')|
    \\
    &= \Big|\int_0^1 \nabla_u h_j(sr'x_u', x_z) x_u' ds -
\int_0^1 \nabla_u h_j(sr'x_u', x_z') x_u' ds\Big|
\leq L|x_z - x_z'| \,,
\end{align*}
where we used Lipschitz continuity of $\nabla_u h_j$ for each $j$ and $L$ is bounded on bounded sets in $(0, \infty) \times \R^m$. 

 If $r' = 0$ and $r > 0$, then 
\begin{equation}
    \tilde h(r',x) = \pi_u\Big( \nabla_u h(0,x_z)\cdot x_u\Big) + \pi_z(h(0,x_z)) \,.
\end{equation}
 In addition, 
$\pi_u \circ h \circ \pi_z = 0$ and the  mean value theorem imply  (for all $r \in [0,\infty), x \in \R^m$)
\begin{equation}\label{eq:mvt-for-tilde-h}
    \tilde h(r,x) = \pi_u\Big( \nabla_u h(u^*,x_z) \cdot x_u\Big) + \pi_z(h(rx_u,x_z))
\end{equation}
for some $u^* \in [0,rx_u]$. The Lipschitz continuity then follows 
from Lipschitz continuity of $h$ and $\nabla_u h$. 
The case $r = r' = 0$ directly follows from the Lipschitz continuity of $h$ and $\nabla_u h$. 
\end{proof}

\begin{cor}\label{lem:lipschitz-f-g}
The following statements hold:
\begin{enumerate}[label=(\roman*)]
    \item $\tilde F: [0,\infty) \times B \times \Theta \to B$ is locally Lipschitz.
    \item $\tilde G: [0,\infty) \times B \times \Theta \to \gamma(\U,\B_{\beta,\rho})$ is locally Lipschitz for all $\rho \in [2,\infty)$, where $\beta$ is as in \Cref{SPDE-G-lem}(i).
    \item For all $M,p > 0$ there is some constant $C > 0$ such that
    \begin{equation}\label{eq:local-u-bound}
        R+\|x\|_B \leq M \implies \|\tilde F_u(R,x,\theta)\|_{L^p} \leq C\|x_u\|_{L^p} \,.
    \end{equation}
    \item For any $r \geq 0$, $\theta \in \Theta$, and any $x$
    \begin{equation}\label{eq:sigma-u-bound} 
        |\tilde \sigma_u(\cdot,r,x,\theta)| \lesssim |x_u| \,.
    \end{equation}
    \item There is some $k > 0$ such that, for $\rho \in [2,\infty)$ and $\beta$ as in \Cref{SPDE-G-lem}:
    \begin{equation}\label{eq:local-u-bound-G}
    \begin{aligned}
         \|\tilde G_u(R,x,\theta)\|_{\beta,\rho} &\lesssim \|x_u\|_{L^\rho} &\quad &\text{and} &\quad \|\tilde G(R,x,\theta)\|_{\beta,\rho} &\lesssim (1+R)\|x\|_{L^\rho}\\
         |\tilde F^+(R,x,\theta)| &\lesssim 1 + (1+R)|x| &\quad &\text{and} &\quad |\tilde F(R,x,\theta)| &\lesssim (1+R)^k|x|^k
         \,.
    \end{aligned}
    \end{equation}
\end{enumerate}
\end{cor}

\begin{proof}
    Recall \Cref{SPDE-F} and \Cref{SPDE-G-sass} where we assumed that $f_1,f_2,\sigma$ and their derivatives are locally Lipschitz as functions from $\overline D \times \R^m \times \Theta$ to $\R^m$. It follows by \Cref{lem:general-tilde-lipschitz} that $\tilde f,\tilde \sigma$ are locally Lipschitz as functions from $\overline D \times [0,\infty) \times \R^m \times \Theta$ to $\R^m$ and (i) follows. 
    
    For (ii) note that by \Cref{SPDE-G-lem}
    $$\|x\cdot H - y\cdot H\|_{\beta,\rho} = 
    \|(x-y)\cdot H\|_{\beta,\rho}
    \lesssim \|x - y\|_{L^\rho} \lesssim \|x - y\|_B \,,$$
    so the result follows from $\tilde \sigma$ being locally Lipschitz after replacing $x$ and $y$ by $\tilde \sigma$ at two different variabes.

    To show \eqref{eq:local-u-bound}, note that $\tilde F_u(R,\pi_z(x),\theta) = 0$ for each $R$ and $\theta$, so
    $$
    |\tilde F_u(R,x,\theta)| = |\tilde F_u(R,x,\theta) - \tilde F_u(R,\pi_z(x),\theta)| \lesssim |x-\pi_z(x)| = |x_u| \,,
    $$
    where the constant in the $\lesssim$ is the (local) Lipschitz constant of $\tilde f$. A similar argument implies \eqref{eq:sigma-u-bound}, because $\sigma$ has bounded derivatives (\Cref{SPDE-G-sass}), and therefore it is globally Lipschitz.

    To prove \eqref{eq:local-u-bound-G}, we use the first assertion of \Cref{SPDE-G-lem} with $x = \tilde{\sigma}$
    to conclude that
    $$\|\tilde G_u(R,x,\theta)\|_{\beta,\rho} \lesssim \|\tilde \sigma_u(\cdot,R,x(\cdot),\theta)\|_{L^\rho} \quad \text{and  } \|\tilde G(R,x,\theta)\|_{\beta,\rho} \lesssim \|\tilde \sigma(\cdot,R,x(\cdot),\theta)\|_{L^\rho}.$$
    Since $\sigma$ 
    has bounded derivatives and thus it is globally Lipschitz, the inequalities in \eqref{eq:local-u-bound-G}
    for $\tilde G$ follow from \eqref{eq:sigma-u-bound} and \eqref{eq:mvt-for-tilde-h}.

    By the last assertion of \Cref{SPDE-F} and \eqref{eq:mvt-for-tilde-h},
    $$|\tilde F_u^+(R,x,\theta)| \lesssim |x_u| \quad \text{and} \quad |\tilde F_z^+(R,x,\theta)| \lesssim 1 + |(Rx_u,x_z)| \,,$$
    which proves $|\tilde F^+(R,x,\theta)| \lesssim 1 + (1+R)|x|$. The final inequality of \eqref{eq:local-u-bound-G} follows from \eqref{eq:mvt-for-tilde-h} and the polynomial boundedness of $f$ and $\partial f/\partial x_i$, $i = 1,\dots,d$.
\end{proof}

Below we frequently use the fact that $\|x\|_{L^1} = \ip{x}{1}$ if $x$ is nonnegative. Thus, the following definition will be useful.

\begin{deff}\label{def:tilde-e}
    Since $A$ is self adjoint
    
    there is a smooth function $\tilde e: \overline D \to \R^m$ such that
    $$\ip{Ax}{1} = \ip{x}{A1} = 
     \ip{x}{\tilde e}$$
    for all $x \in \B_{1,2}$.
    
\end{deff}

\begin{rem}
    Above we used the fact that $1 \in \B_{1,2}$, which is clearly true for the boundary conditions we consider in \Cref{sass-A}. If we were to also consider Dirichlet boundary conditions, we would need to be more careful about what happens near the boundary of $D$ in all of our proofs regarding positivity (ex. \Cref{bigthm}\ref{bigthm:nonneg}).
\end{rem}

Recall that $y_u$ and $y_z$ denote the first $d$ respectively the last $m-d$ components of a vector $y$. To simplify the notation, we denote $u = y_u$
and $z = y_z$. The following theorem will provide the well-posedness of the projective process. 

\begin{thm}\label{bigthm}
    Let $r$ be large enough. Under \Cref{sass-A}-\ref{SPDE-G-sass}, for every $R \geq 0$, $\theta \in \Theta$, and initial condition $\tilde{x}_0 \in \cdom$ there is a unique mild solution in $L^r$ (see \Cref{def:mild}), denoted by $\tilde{x}_t = (\tilde{u}_t, \tilde{z}_t)$, to
    \begin{equation}\label{eq:mild-solution}
        \tilde{x}_t =S(t) \tilde{x}_0 + (S * \tilde F(R,\tilde{x},\theta))_t + (S \diamond \tilde G(R, \tilde{x},\theta))_t
   \end{equation}
    Furthermore, the following properties hold:
    \begin{enumerate}[label=(\roman*)]
        \item \label{bigthm:nonneg} $\tilde{x}_t \geq 0$.
        \item \label{bigthm:invt-set} $\tilde{u}_t = 0 \iff \tilde{u}_0 = 0$.
        \item \label{bigthm:cont-in-time} $t \mapsto \tilde{x}_t$ is almost surely a continuous map from $[0,\infty)$ to $B$.
        \item \label{bigthm:cont-in-space} If $(R_n,\tilde{x}_0^{(n)},\theta_n) \to (R,\tilde{x}_0,\theta)$ in $[0,\infty) \times \cdom \times \Theta$, then $\tilde{x}^{(n)}$, the solutions to \eqref{eq:mild-solution} with initial condition $\tilde{x}_0^{(n)}$ and $R = R_n$, $\theta = \theta_n$, converge to $\tilde{x}$ in $L^p(\Omega;C([0,T];B))$ for all $p,T > 0$.
        \item \label{bigthm:drdv} If $\tilde{u}_0 \neq 0$ then $(r_t,v_t) \coloneqq (R\|\tilde{u}_t\|_{L^1}, \tilde{u}_t/\|\tilde{u}_t\|_{L^1})$ satisfies
        \begin{equation}\label{new-dr-dv}
            \begin{aligned}
            dr &= r[\ip{v}{\tilde e_u} + \ip{\tilde F_u(r,v,\tilde{z})}{1_u}]dt + r\ip{\tilde G_u(r,v,\tilde{z})^*1_u}{dW_t} \\
            dv &= A_uvdt + v[-\ip{v}{\tilde e_u} - \ip{\tilde F_u(r,v,\tilde{z})}{1_u} + \|\tilde G_u(r,v,\tilde{z})^*1_u\|_\U^2]dt \\
            &+ \tilde F_u(r,v,\tilde{z})dt - \tilde G_u(r,v,\tilde{z})\tilde G_u(r,v,\tilde{z})^*1_udt \\
            &+ \tilde G_u(r,v,\tilde{z})dW_t- \ip{\tilde G_u(r,v,\tilde{z})^*1_u}{dW_t}v \,,
        \end{aligned}
    \end{equation}
    where the equation for $v$ holds in the mild sense of \Cref{def:mild}. We omitted the dependence on $\theta$ for brevity.
    
    \item \label{bigthm:dbeta} Let $\bar{x}_t \coloneqq (R\tilde{u}_t, \tilde{z}_t) = (r_tv_t,\tilde{z}_t)$, where $(r_t, v_t)$ was defined in (v). Then
    \begin{equation}\label{eq:dbeta-x-tilde}
        d\|\bar x_t\|_{\beta,2}^2 = [-2\|\bar x_t\|_{\beta+1/2,2}^2 +2\ip{\bar x_t}{F(\bar x_t)}_\beta + \|G(\bar x_t)\|_{\beta,2}^2]dt + 2\ip{\bar x_t}{G(\bar x_t)dW_t}_\beta \,,
    \end{equation}
    where $\beta$ is as in \Cref{SPDE-G-lem}. 
    \end{enumerate}
\end{thm}

\begin{proof}
To simplify the notation, we drop tildes from $x$, $u$, and $z$ in the proof. 
We split the proof into the following steps:

\begin{enumerate}
    \item Truncate $\tilde F, \tilde G$ to obtain approximations $x_n$ to \eqref{eq:mild-solution} which satisfy \ref{bigthm:cont-in-time}, \ref{bigthm:cont-in-space}.
    \item Show that \ref{bigthm:nonneg} holds for $x_n$. This is the longest step, but the essential idea is as follows:
    
    Replace $F_n,G_n$ with $F_n \circ \psi, G_n \circ \psi$, where $\psi(x) = x \vee 0$ is the positive part of $x$. If the solution of this new equation $y^*$ is nonnegative, then it solves the same equation as $x_n$, so by uniqueness $x_n = y^*$ is nonnegative. To show that $y^*$ is nonnegative, try applying It\^{o}'s formula to a smooth version $\Phi$ of the Lyapunov function $x \mapsto \int_D -(x\wedge 0)$. There are three deterministic terms: one coming from $A$, another coming from $F_n \circ \psi$, and the third is the It\^{o} correction term. Since we precomposed with $\psi$, the latter two are both $0$. By \Cref{lem:e-Phi-is-noninc} (the crucial argument in this step) the term coming from $A$ is $\leq 0$. This implies that $\E[\Phi(y^*(s))]$ is nonincreasing, and since $\Phi(y^*(s)) = 0$ iff $y^*(s) \geq 0$, this proves our claim.
    
    Unfortunately, rigorously justifying the proof sketch above requires a couple of nontrivial approximations. For example, we must use the Yosida approximation (\Cref{def:yosida}) because the hypotheses of It\^{o}'s formula require strong solutions as opposed to mild ones. There is also a less obvious approximation which is explained in greater detail below (\Cref{rmk:extra-approx}).
    
    \item Use another It\^{o} formula/Yosida approximation argument to show that \ref{bigthm:dbeta} holds for $x_n$ until the time $\tau_n$ when $\|x_n\|_B$ first exceeds $n-1$.
    \item Use another Ito formula/Yosida approximation argument to show that \ref{bigthm:drdv} holds for $x_n = (u_n,z_n)$ until $\tau_n \wedge \sigma_\epsilon$, where $\sigma_\epsilon$ is the first time $\|u_n\|_{L^1}$ drops below $\epsilon$. In the process of showing that $v_t$ is indeed a mild solution, we prove $\sup_{\epsilon > 0}\E[\int_0^T \|u_n(s \wedge \sigma_\epsilon)\|_{L^2}^2\|u_n(s \wedge \sigma_\epsilon)\|_{L^1}^{-2}ds] < \infty$.
    \item Show that \ref{bigthm:invt-set} is satisfied for $x_n$ using the Lyapunov function $-\ln{\|u_n\|_{L^1}}$, which requires the use of the bound above. This implies that $\sigma_\epsilon \to \infty$ almost surely, and so \ref{bigthm:drdv} holds for $x_n$ until $\tau_n$.
    \item Prove $\sup_n \E[\sup_{t \leq T} \|x_n\|_B^p] < \infty$, thus $\sup_n \tau_n = \infty$ a.s.
    \item Show that there is a well-defined process $x_t$ which satisfies $x_n(t \wedge \tau_n) = x_{t \wedge \tau_n}$ for every $n$ and solves \eqref{eq:mild-solution}.
    \item Conclude the theorem by noting that $x$ solves \eqref{eq:mild-solution} and inherits the aforementioned properties of $x_n$.
    
\end{enumerate}

Since $R,\theta$ are held constant for much of the proof, we sometimes omit the dependence on $R,\theta$ for ease of notation.

\textit{Step 1:} Recall \Cref{r-functions}. The truncation function $\Psi_n: \R^m \to \R^m$ given by
\begin{equation}\label{eq:trunc-psi-map}
          \begin{aligned}
        \Psi_n(x) = \begin{cases}
             x  & |x| \leq n \\
            n\frac{x}{|x|} & |x| > n
        \end{cases}
    \end{aligned}
\end{equation}
is Lipschitz with bounded range and it induces a Lipschitz map $\Psi_n: B \to B$ via
$f \mapsto \Psi_n \circ f$. We define the truncated functions 
\begin{equation}\label{eq:dffngn}
F_n(r,x,\theta) \coloneqq \tilde F(r, \Psi_n(x),\theta) \qquad 
\textrm{and} \quad G_n(r,x,\theta) \coloneqq \tilde G(r,\Psi_n(x),
\theta).    
\end{equation}
 Let $\beta$ be as in \Cref{SPDE-G-lem} and $r \in [2,\infty)$ in the definition of $\B \coloneqq L^r$ be large. By \Cref{lem:lipschitz-f-g} i-ii, $F_n(R,\cdot,\theta):B \to B$ and $G_n(R,\cdot,\theta): B \to \gamma(\U,\B_{\beta,r})$ are globally Lipschitz.

    Then with $\theta_F = 0$ and $\theta_G = \beta$ we conclude by \Cref{general-mild-existence} that for any $x_0 \in \cdom$ there is a unique $B$-valued processes $x_n$ (mild solution) satisfying \ref{bigthm:cont-in-time} and
    \begin{equation}\label{eq:truncated}
            x_n(t) = S(t) x_0 + (S * F_n(x_n))_t + (S \diamond G_n(x_n))_t\,, \quad x_n(0) = x_0 \,.
    \end{equation}
    Following the notation of \Cref{general-mild-existence}, we  use $\B \coloneqq L^r$, so $\B_{-1} = \B_{-1,r}$, etc. For fixed $x \in B$ we have by \Cref{lem:lipschitz-f-g} that $F_n(\cdot,x,\cdot): [0,\infty) \times \Theta \to \B$ and $G_n(\cdot,x,\cdot): [0,\infty) \times \Theta \to \gamma(\U,\B_{\beta,r})$ are continuous, so by \Cref{general-mild-existence} we conclude that for each $n$ the solutions of \eqref{eq:truncated} satisfy \ref{bigthm:cont-in-space}.

    \textit{Step 2:} For this step, fix some $n$ and approximate $G_n$ by
    $$
    \mathcal{G}_k(x) \coloneqq G_n(x)\circ \mathfrak{p}_k \,,
    $$
    where $\mathfrak{p}_k$ is the orthogonal projection onto $\operatorname{Span}(u_1,\dots,u_k)$ and $\{u_k\}_{k \in \N}$ is the orthonormal basis of $\U$ from \Cref{SPDE-G-sass}. Since 
    by the elliptic regularity the eigenfunctions $(e_n)$ of $A$ are in $L^\infty$, it follows from \eqref{eq:gron} and the definition of $H$ that $H \circ \mathfrak{p}_k \in \gamma(\U,\B)$. Since $\|f \mapsto x\cdot f\|_{\Ll(\B,\B)} \lesssim \|x\|_{L^\infty}$, we also obtain from the local Lipschitzness of $\tilde \sigma$ (\Cref{lem:general-tilde-lipschitz}) that $\mathcal{G}_k: B \to \gamma(\U,\B)$ is globally Lipschitz (global because we truncated with $\Psi_n$ in \textit{Step 1} and $R$ is fixed). 
    
    Notice that, for all $x \in B$, $\lim_{k \to \infty} \mathcal{G}_k(x) = G_n(x)$ in  $\gamma(\U,\B)$, and therefore in $\gamma(\U,\B_{\beta})$, by \cite[Proposition 2.4]{umd-integration} ($\mathfrak{p}_k$ converges to the identity on $\U$).

    Then \Cref{general-mild-existence} implies that there exists a mild solution to
    \begin{equation}\label{eq:smoothed-G}
        y_k(t) = S(t) x_0 + (S * F_n(y_k))_t + (S \diamond \mathcal{G}_k(y_k))_t\,, \quad y_k(0) = x_0
    \end{equation}
    and $y_k \to x_n$ as $k \to \infty$ in $L^p(\Omega; C([0,T];B))$ for all $p,T > 0$. Thus, to show nonnegativity of $x_n$ it suffices to show the nonnegativity of $y_k$ for each $k$.

    In what follows we fix $k$ and
    let $y^*(t)$ be the solution of
    \begin{equation}\label{eq:truncated-positive}
            y^*(t) = S(t) x_0 + (S * (F_n \circ \psi)(y^*))_t + (S \diamond (\mathcal{G}_k \circ \psi)(y^*))_t\,, \quad y^*(0) = x_0 \,,
    \end{equation}
    where $\psi(f) = f \vee 0$ ($\psi$ is Lipschitz from $B$ to $B$ so \Cref{general-mild-existence} still provides the existence of the solution). To simplify notation for already fixed $n$ and $k$, we let $\hat F \coloneqq F_n \circ \psi$, $\hat G \coloneqq \mathcal{G}_k \circ \psi$.

    Let $\phi: \R \to \R$ be a function which satisfies:
    \begin{enumerate}[label=(\roman*)]
        \item $\phi(x) = 0$ for $x \geq 0$ and $\phi(x) > 0$ for $x < 0$.
        \item $\phi$ is twice continuously differentiable with uniformly bounded first and second derivatives.
        \item $\phi'' \geq 0$ (so $\phi$ is convex).
    \end{enumerate}
    For example, we may set $\phi'(x) = -1$ for $x \leq -1$, and 
    $\phi'(x) = 0$ for $x \geq 0$, smoothly and increasingly interpolate between $-1$ to $0$ for $x \in (-1,0)$, and then set $\phi(x) = \int_0^x\phi'(t)dt$.

    Since $S(t)$ preserves positivity (\Cref{sass-A}), by Krein-Rutman ($S(t)$ is compact, see \Cref{rem:subtract-from-c}) there is a nonnegative eigenfunction $\nu$ corresponding to the largest eigenvalue of $S(t)$, and since $\nu \in \B_{\theta,q}$ for all $\theta,q$ we conclude via Sobolev embedding that $\nu \in C^\infty(\overline D)$ (see \Cref{characterization-of-interpolation-spaces}).  Similarly, $S(t)^*$, the adjoint semigroup, also preserves positivity and so there is $\nu^* \geq 0$, a nonnegative eigenfunction for $S(t)^*$ (with the same eigenvalue). By the strong maximum principle (see \Cref{rem:subtract-from-c}) $\nu, \nu^*$ are strictly positive on $D$, and thus the function
    $$\Phi(z) \coloneqq \int_D (\phi \circ (z\cdot\nu^{-1}))\nu^*\nu$$

    is well-defined (note the integrand is in $L^r$ since $z \in \B = L^r$ and $\phi$ is globally Lipschitz with $\phi(0) = 0$). By (ii) in the definition of $\phi$, dominated convergence theorem, and mean value theorem we have that $\Phi$ is twice continuously differentiable from $\B$ to $\R$ (for $r \geq 3$). Specifically,
    \begin{equation}\label{eq:phi-double-prime}
        \Phi'(z)w = \int_D (\phi' \circ(z \cdot \nu^{-1}))w\nu^*, \qquad 
    \Phi''(z)[w_1, w_2] = \int_D (\phi'' \circ(z \cdot \nu^{-1}))w_1w_2\frac{\nu^*}{\nu}
    \end{equation}
    (the $\nu^{-1}$ is not problematic since we work with Neumann or periodic boundary conditions, so the maximum principle implies that $\nu$ is uniformly bounded below on $D$).
    By (i) in the definition of $\phi$, one observes $z \geq 0$ if and only if $\Phi(z) \leq 0$ (in fact $\Phi(z) = 0$), and therefore if $\Phi(y^*(t)) \leq 0$ almost surely then $y^*(t) \geq 0$. Consequently, $\psi(y^*(t)) = y^*(t)$, which implies $y^*(t)$ solves \eqref{eq:smoothed-G}, and so by uniqueness $y_k(t) = y^*(t) \geq 0$. Thus, the claim of \textit{Step 2} holds if $\Phi(y^*(t)) \leq 0$ almost surely.

      To proceed, we wish to apply It\^o's formula \cite[Theorem 2.4]{ito-formula} to $\Phi$, but the solutions to \eqref{eq:truncated-positive} are not strong. For each $\ell \geq 1$ let $z_\ell(t)$ be the  strong solution of \eqref{eq:truncated-positive} (see \Cref{def:yosida}) 
with $\hat{F}$ and $\hat{G}$ (recall $\hat{F} = F_n \circ \psi$, 
$\hat{G} = \mathcal{G}_k \circ \psi$) replaced respectively by their
      Yosida approximations  $\hat F_\ell \coloneqq R_\ell \hat F$, and $\hat G_\ell \coloneqq R_\ell \hat G$.

    Then It\^o's formula can be applied to $z_\ell$:
    \begin{equation}\label{eq:ito-positivity}
        \begin{aligned}
        \Phi(z_\ell(t)) &= \Phi(z_\ell(0)) + \int_0^t \Phi'(z_\ell(s))[Az_\ell(s) + \hat F_\ell(z_\ell(s))]ds \\
        &+ \int_0^t \Phi'(z_\ell(s))\hat G_\ell(z_\ell(s))dW_s + \frac{1}{2}\int_0^t \tr_{\hat G_\ell(z_\ell(s))}\Phi''(z_\ell(s))ds \,,
    \end{aligned}
    \end{equation}
where for any operator $T$ on $\U$ we have $\tr_{T}\Phi''(z_\ell(s)) = \sum_{j \geq 1} \Phi''(z_\ell(s))[Tu_j, Tu_j]$ and $(u_j)$ is the orthonormal basis as in \Cref{SPDE-G-sass}.
    
    We analyze each term on the right hand side separately. The following lemma treats the term containing $A$. 
    
    \begin{lem}\label{lem:e-Phi-is-noninc}
        For any $z \in \B_1$ (not necessarily positive), $\Phi'(z)(Az) \leq 0$.
    \end{lem}
    
    \begin{proof}
    From \Cref{sass-A} (i) (assumption on the spectrum of $A$) follows that there is 
$\lambda \leq 0$ be such that $S(t)\nu = e^{\lambda t}\nu$ and $S(t)^*\nu^* = e^{\lambda t}\nu^*$. 
        Let $\mathfrak{m}_g$ denote the multiplication by a function $g$ and denote $\Pp_t \coloneqq e^{-\lambda t} \mathfrak{m}_{\nu^{-1}} \circ S(t) \circ \mathfrak{m}_\nu$. It is standard to check that $\Pp_t 1 = 1$ and $\Pp_t$ preserves positivity, and therefore $\Pp_t$ is a Markov semigroup with the generator
         $\Ll u = -\lambda u + \nu^{-1}A(u\nu)$, and $\nu^* \nu$ is a fixed point of $\Pp_t^*$. Since $\phi$ is convex and $\Pp_t$ is Markov, 
         by Jensen's inequality and the invariance of $\nu^* \nu$ we have
        $$
        \int_D \phi(\Pp_t u)\nu^*\nu \leq \int_D \Pp_t\phi(u)\nu^*\nu = \int_D \phi(u)\nu^*\nu \,,
        $$
        and so we conclude that
        $$\int_D \phi'(u)[-\lambda u + \nu^{-1}A(u\nu)]\nu^*\nu  = \int_D \phi'(u)(\Ll u)\nu^*\nu = \frac{d}{dt} \int_D \phi(\Pp_t u)\nu^*\nu \Big|_{t = 0} \leq 0 \,.$$
        Using \eqref{eq:phi-double-prime} and the above with $u = z\nu^{-1}$, one obtains
        $$\Phi'(z)(Az) = \int_D \phi'(z\nu^{-1})(Az)\nu^* \leq \lambda \int_D \phi'(z\nu^{-1})z\nu^* \,.$$
        Since $\lambda \leq 0$ and $\phi'(z\nu^{-1})z \geq 0$ ($\phi'' \geq 0$ and $\phi'(0) = 0$), the claim is proven.
    \end{proof}

    To estimate the other terms in \eqref{eq:ito-positivity} we observe that 
    by the convergence of the Yosida approximations (see \Cref{def:yosida}), $\Phi(z_\ell(t)) \to \Phi(y^*(t))$ and $\Phi(z_\ell(0)) \to \Phi(x_0) = 0$ because $x_0 \in C(\overline D)_+$. Moreover, since $\phi'(x_i) = \phi''(x_i) = 0$ if $x_i \geq 0$ and $\tilde \sigma_i(\cdot,R,\Psi_n(\psi(x)),\theta) = 0$ if $x_i \leq 0$ (see \Cref{SPDE-G-sass}), it follows from $\hat G = (\tilde{\sigma} \cdot \mathcal{G}_k) \circ \psi$ that
$\phi'(x) \hat G(x)u_\ell = \phi''(x) \hat G(x)u_\ell = 0$, and therefore 
    $$
    \int_0^t \Phi'(y^*(s))\hat G(y^*(s))dW_s + \frac{1}{2}\int_0^t \tr_{\hat G(y^*(s))}\Phi''(y^*(s))ds = 0\,.
    $$
    Using the same reasoning as for $\hat G$ we obtain 
$\phi'(x)(x \vee 0)f_1(x \vee 0) = 0$. Then from 
 $f(x) \geq x\cdot f_1(x)$ if $x \geq 0$ (\Cref{SPDE-F}) and $\phi' \leq 0$ follows
    $$
    \int_0^t \Phi'(y^*(s))\hat F(y^*(s))ds 
    = 
    \int_0^t \Phi'(y^*(s)) F_n(y^*(s)\vee 0)ds 
    \leq 0 \,.
    $$
    
    Using \Cref{lem:ito-convergence} to take the limit as $l \to \infty$ on both sides of \eqref{eq:ito-positivity}, we conclude from \Cref{lem:e-Phi-is-noninc} and the above that $\Phi(y^*(t)) \leq \Phi(x_0) = 0$ almost surely, which finishes \textit{Step 2}. Indeed, the assumptions of \Cref{lem:ito-convergence} are satisfied since $\|\Phi'(z)\|_{\Ll(\B,\R)} + \|\Phi''(z)\|_{\Ll(\B,\Ll(\B,\R))} \lesssim 1$. To see this, note that $\phi$ has uniformly bounded first and second derivatives, so by \eqref{eq:phi-double-prime} it follows that
    $$|\Phi'(z)w| \leq \int_D |\phi'(z\nu^{-1})w\nu^*| \lesssim \|w\|_{L^1} \,.$$
    The inequality for $\Phi''$ follows similarly using that $S(t)$ is self-adjoint, and therefore $\nu^*\nu^{-1} = 1$. 
    
    We remark that if we were to allow $S$ not self-adjoint but the boundary conditions were still either Neumann or periodic, then $\nu$ is bounded away from zero on $D$ and this proof still works. However, in  the case of Dirichlet boundary conditions one would have to be more careful and use arguments related to the boundary Harnack inequality.

     To summarize, in \textit{Step 2} we have shown for each $n\geq 1$ that $x_n(t) \geq 0$ almost surely.

    \textit{Step 3:}
    Next we fix $n$ and show that $y(t) = x_n(t)$ solving 
    \eqref{eq:truncated} satisfies \ref{bigthm:dbeta}. It is essentially a consequence of It\^{o}'s formula applied to $\|\cdot\|_{\beta,2}^2$. For technical reasons we must use the Yosida approximations again, but since many steps are similar to the other Yosida approximation arguments in this paper (\textit{Step 2} of this theorem, \Cref{lem:ito-convergence}, etc.) we omit some of the details.
    
    First, we claim that $\hat y(t) = \pi_u \circ y(t)$ (recall $\pi_u$ is the projection on the first $d$ coordinates) satisfies 
    \begin{equation}\label{eq:dbeta-of-y}
        \begin{aligned}
         \|\hat y(t)\|_{\beta,2}^2 &= \|\hat y(0)\|_{\beta,2}^2 + \int_0^t 2\ip{\hat y(s)}{F_n(y(s))}_\beta + \|\pi_u \circ G_n(y(s))\|_{\beta,2}^2 ds \\
         &+ \int_0^t \ip{2\hat y(s)}{G_n(y(s))dW_s}_\beta - \int_0^t 2\|\hat y(s)\|_{\beta+1,2}^2ds \,,
    \end{aligned}
    \end{equation}
    where $\beta$ is as in \Cref{SPDE-G-lem}. By using the same argument, we obtain that \eqref{eq:dbeta-of-y} holds with 
     $\pi_u$ replaced by $\pi_z$ and $\hat y$ by $\pi_z \circ y$.

    For each $\ell \geq 1$ we replace $F_n$ and $G_n$ in \eqref{eq:truncated} by the Yosida approximations (\Cref{def:yosida}) respectively $R_\ell F_n$
    and $R_\ell G_n$ and denote $y_\ell$ the corresponding strong solution 
    (see \Cref{def:yosida} for well posedness). 
    Applying the energy (or It\^o) formula (see \cite[Theorem 4.2.5]{liu-rockner} or \cite[Theorem 3.3.15]{SPDE-book}) to $\hat y_\ell - \hat y_k$ in the normal triple $(\B_{\beta+1/2,2}, \B_{\beta,2}, \B_{\beta-1/2,2})$ (note that $\B_{\beta,2}$ is a Hilbert space with the inner product $\ip{\cdot}{\cdot}_\beta$), and noting 
    that the norm in $\gamma(\U, \B_{\beta,2})$ is the same as the Hilbert-Schmidt norm $\mathcal{L}_2(\U, \B_{\beta,2})$ (see \cite[Proposition 13.5]{radonifying-survey}), we obtain
    \begin{multline}\label{eq:dbeta-of-ym}
    \|\hat y_\ell(t) - \hat y_k(t)\|_{\beta,2}^2 = \|\hat y_\ell(0) - \hat y_k(0)\|_{\beta,2}^2 \\
    \begin{aligned}
         &+ \int_0^t 2\ip{\hat y_\ell(s) - \hat y_k(s)}{R_\ell F_n(y_\ell(s)) - R_kF_n(y_k(s))}_\beta \\
         &+ \|\pi_u\circ R_\ell G_n(y_\ell(s)) - \pi_u \circ R_kG_n(y_k(s))\|_{\beta,2}^2 ds \\
         &+ \int_0^t 2\ip{\hat y_\ell(s) - \hat y_k(s)}{ [R_\ell G_n(y_\ell(s)) - R_kG_n(y_k(s))]dW_s}_\beta \\
         &- \int_0^t 2\|\hat y_\ell(s) - \hat y_k(s)\|_{\beta+1/2,2}^2ds \,,
    \end{aligned}
    \end{multline}
where we used that $\ip{Ay}{y}_\beta = -\ip{(-A)^{1/2}y}{(-A)^{1/2}y}_\beta = -\|y\|_{\beta+1/2}^2$, because $A$ is self-adjoint (\Cref{sass-A}). Since $\beta \leq 0$ and $F_n$ is globally Lipschitz from $B$ to $\B$, it is also globally Lipschitz from $B$ to $\B_{\beta,2}$ since $\B \subset L^2 \subset \B_{\beta,2}$. Also, recall from \textit{Step 1} that $G_n: B \to \gamma(\U, \B_{\beta,2})$ is globally Lipschitz. Then it 
follows by the same argument as \Cref{lem:ito-convergence} with $\Phi(x) = \|x\|_{\beta,2}^2$
 that all of the terms on the right hand side of \eqref{eq:dbeta-of-ym} converge to $0$ in $L^p(\Omega; C([0,T],\R))$ as $(\ell,k) \to \infty$ except perhaps $\int_0^t 2\|\hat y_\ell(s) - \hat y_k(s)\|_{\beta+1/2,2}^2ds$. Since $\hat y_l \to \hat y$ in $L^p(\Omega; C([0,T];\B_{\beta,2}))$ for all $p,T > 0$ (\Cref{def:yosida}), the left hand side of \eqref{eq:dbeta-of-ym} converges to $0$.  Hence, $(\hat y_\ell)_{\ell = 1}^\infty$ is Cauchy in $L^2(\Omega \times [0,T]; \B_{\beta+1/2,2})$, so $\int_0^t 2\|\hat y_\ell(s)\|_{\beta+1/2,2}^2ds \to \int_0^t 2\|\hat y(s)\|_{\beta+1/2,2}^2ds$ (the limit must be $\hat y$ because $\hat y_{\ell} \to \hat y$ in $L^p(\Omega; C([0,T];\B_{\beta,2}))$).
 Then \eqref{eq:dbeta-of-y} follows 
 after setting $\hat{y}_k = 0$ and passing $\ell \to \infty$
 in \eqref{eq:dbeta-of-ym}.
  . 

For already fixed $n$ define a (random) time
  \begin{equation}\label{eq:def-tau-n}
        \tau_n \coloneqq \inf\{t > 0 \mid \|x_n(t)\|_B > n - 1\} \,,
    \end{equation}
and $\bar x_n(t) \coloneqq (Ry_u(t), y_z(t))$.
Note that \eqref{eq:dbeta-of-y} holds for $t\leq \tau_n$ with 
$F_n$ and $G_n$ replaced respectively by $\tilde{F}$ and $\tilde{G}_u$ because $\Psi_n(x) = x$ for $\|x\|_B \leq n - 1$ (see \eqref{eq:trunc-psi-map}).

We have by \eqref{eq:alt-def-tilde} that
    $$
    (R\pi_u + \pi_z)\circ \tilde{F}(y(t)) = 
    F(\bar{x}_n(t)) \,,$$
    and the analogous identity holds with $F$ replaced by $G$. Also, $\|\bar{x}_n(t)\|_{\beta,2}^2 = R^2\|\pi_u \circ y(t)\|_{\beta,2}^2 + \|\pi_z \circ y(t)\|_{\beta,2}^2$, because $\ip{\pi_u \circ y(t)}{\pi_z \circ y(t)}_\beta = 0$. Similarly 
    \begin{align*}
      \|(R\pi_u + \pi_z)\circ G_n(y(t))\|_{\beta,2}^2 = R^2\|\pi_u \circ G_n(y(t))\|_{\beta,2}^2 + \|\pi_z \circ G_n(y(t))\|_{\beta,2}^2  \,,
    \end{align*}
    \begin{align*}
        &R^2\ip{\pi_u \circ y(s)}{F_n(y(s))}_\beta + \ip{\pi_z \circ y(s)}{F_n(y(s))}_\beta \\&=  \ip{R\pi_u \circ y(s)}{R\pi_u \circ F_n(y(s))}_\beta + \ip{\pi_z \circ y(s)}{\pi_z \circ F_n(y(s))}_\beta \\
        &= \ip{(R\pi_u + \pi_z)\circ y(s)}{(R\pi_u + \pi_z)\circ F_n(y(s))}_\beta\,,
    \end{align*}
    and the analogous statement holds with $F$ replaced by $G$.
    Thus, a multiplication of \eqref{eq:dbeta-of-y} by $R^2$ and an addition  to \eqref{eq:dbeta-of-y} with $u$ replaced by $z$ implies
    for $t \leq \tau_n$ that
    
    \begin{equation*}
        \begin{aligned}
         \|\bar{x}_n(t)\|_{\beta,2}^2 &= \|\bar{x}_n(0)\|_{\beta,2}^2 + \int_0^t 2\ip{\bar{x}_n(s)}{F(\bar{x}_n(s))}_\beta + \|G(\bar{x}_n(s))\|_{\beta,2}^2 ds \\
         &+ \int_0^t \ip{2\bar{x}_n(s)}{G(\bar{x}_n(s))dW_s}_\beta - \int_0^t 2\|\bar{x}_n(s)\|_{\beta+1,2}^2ds \,.
    \end{aligned}
    \end{equation*}

    To summarize, in \textit{Step 3} we showed that \ref{bigthm:dbeta} holds with $\bar{x}_n(t \wedge \tau_n)$ in place of $\bar{x}_t$.

    \textit{Step 4:}
    We fix $n$ and show that $y(t) = x_n(t)$ solving \eqref{eq:truncated}  satisfies \ref{bigthm:drdv} until $\tau_n$ as in \eqref{eq:def-tau-n}. Throughout, we assume $u_0 \coloneqq (x_0)_u \neq 0$ as postulated in \ref{bigthm:drdv}. Let $\hat F \coloneqq (F_n)_u$, $\hat G \coloneqq (G_n)_u$, 
    where $F_n$ and $G_n$ are as in \eqref{eq:dffngn}. Note that $\hat{F}$ and $\hat{G}$ are different than the ones in Step 2.
    We start by noticing that
    \begin{equation}\label{eq:d-l1-norm}
        \|y_u(t)\|_{L^1} = \|u_0\|_{L^1} + \int_0^t \ip{y(s)}{\tilde e_u} + \ip{\hat F(y(s))}{1_u}ds + \int_0^t \ip{\hat G(y(s))^* 1_u}{dW_s},
    \end{equation}
    where $\tilde e$ is as in \Cref{def:tilde-e} and the adjoint of $\tilde G_u(y(s))$ is in $\Ll(\U,L^2(D;\R^d))$. Indeed, by \textit{Step 2} $y(t) \geq 0$, thus $\|y_u(t)\|_{L^1} = \ip{y_u(t)}{1_u}$, so \eqref{eq:d-l1-norm} almost follows by \Cref{lem:ito-convergence} applied to $\Phi(x) \coloneqq \ip{x_u}{1_u}$. The careful reader may notice that $\hat G$ is not globally Lipschitz from $B$ to $\gamma(\U, L^2)$, but since the It\^{o} correction term is $0$ and the stochastic term $\Phi'(x)\hat G(x)$ is equal to $\hat G(x)^*1_u$, it is enough to have $\hat G$ globally Lipschitz from $B$ to $\Ll(\U,L^2)$. This follows because $H \in \Ll(\U, L^2)$ (\Cref{SPDE-G-sass}) and $\tilde \sigma \circ \Psi_n$ is globally Lipschitz (for fixed $R$).

    Consequently,  the $dr$ portion of \eqref{new-dr-dv} follows for $t \leq \tau_n$ after multiplication by $R$ on both sides of \eqref{eq:d-l1-norm} and
    using for $R > 0$, $\|y\|_B \leq n$, and $\|y_u\|_{L^1} \neq 0$ that
\begin{equation}
    R \hat{F}(y) = R\tilde F_u(R,y)= F(Ry_u,y_z) = R\|y_u\|_{L^1}\tilde{F}_u\Big(R\|y_u\|_{L^1}, 
     \frac{y_u}{\|y_u\|_{L^1}}, y_z \Big) \,, 
\end{equation}
which in $(r,v,z)$ coordinates is equal to $r\tilde F_u(r,v,z)$,
and similarly for $\hat{G}$ term. Note that if $R= 0$, the statement is trivial since $r_t = 0$ for all $t$. 
    
    To obtain the $dv$ portion of \eqref{new-dr-dv}, we have to guarantee that $\|y_u(t)\|_{L^1}$ does not vanish (indeed, if $\|y_u(t)\|_{L^1} = 0$ then $v_t$ is not even well-defined). Thus, for every $\epsilon > 0$
    we define
    \begin{equation}\label{eq:sigma-stopping-time}
        \sigma_\epsilon \coloneqq \inf\{t > 0 \mid \|y_u(t)\|_{L^1} < \epsilon\} 
    \end{equation}
    and apply \eqref{eq:d-l1-norm} and It\^{o}'s formula to the semi-martingale $\|y_u(t \wedge \sigma_\epsilon)\|_{L^1}$ to obtain
    \begin{equation*}
    \begin{aligned}
        d\|y_u\|_{L^1}^{-1} &= \Big(-\|y_u\|_{L^1}^{-2}[\ip{y_u}{\tilde e_u} + \ip{\hat F(y)}{1_u}] + \|y_u\|_{L^1}^{-3}\|\hat G(y)^*1_u\|_\U^2\Big)dt \\
        &- \ip{\|y_u\|_{L^1}^{-2}\hat G(y)^*1_u}{dW_t} \,.
    \end{aligned}
    \end{equation*}
    $y_t$ is a strong solution in $\B_{-1}$ by \Cref{mild-strong} (\eqref{eq:mild-vs-strong} follows from the global Lipschitz properties of $F_n$ and $G_n$ noted in \textit{Step 1}), so by \cite[Corollary 2.6]{ito-formula} we have the following equality in $\B_{-1}$ ($v \coloneqq y_u/\|y_u\|_{L^1}$ and $\rho \coloneqq \|y_u\|_{L^1}^{-1}$):
    \begin{equation}\label{eq:dv-for-y}
           \begin{aligned}
        dv &= \Big(-\ip{v}{\tilde e_u} - \rho\ip{\hat F(y)}{1_u} + \rho^2\|\hat G(y)^*1_u\|_\U^2\Big)vdt - \ip{\rho\hat G(y)^*1_u}{dW_t}v \\
        &+[A_uv + \rho \hat F(y)]dt + \rho \hat G(y)dW_t - \rho^2\hat G(y)\hat G(y)^*1_udt \,,
    \end{aligned}
    \end{equation}
    for all $t \leq \sigma_\epsilon$. We want to use \Cref{mild-strong} to show that \eqref{eq:dv-for-y} holds in $\B$ in the mild sense of \Cref{def:mild}. Thus, we need to verify \eqref{eq:int-condition-mild}. We start by showing that for some small enough $\eta >0$ (independent of $\epsilon$)
    \begin{equation}\label{eq:desired-bc-bound}
        \E\Big[\int_0^{T\wedge \sigma_\epsilon} \|b_t\|_{-1+\eta} + \|c_t\|_{-1/2+\eta}^2dt\Big] < \infty \,,
    \end{equation}
    where
    
    \begin{align*}
        b_t &\coloneqq \Big(-\ip{v}{\tilde e_u} - \rho\ip{\hat F(y)}{1_u} + \rho^2\|\hat G(y)^*1_u\|_\U^2\Big)v + \rho \hat F(y) - \rho^2\hat G(y)\hat G(y)^*1_u \\
        c_t &\coloneqq - \ip{\rho\hat G(y)^*1_u}{\cdot}v + \rho \hat G(y)
    \end{align*}

    Using \Cref{SPDE-G-lem} (especially \eqref{eq:wierd-g}) to handle the terms involving $\hat G$, global Lipschitz properties of $\tilde \sigma \circ \Psi_n$ and $\hat F$ such as $|\hat F(y)| \lesssim |y_u|$ (see \Cref{lem:lipschitz-f-g} for further explanation), 
    $\rho^2\hat G(y)\hat G(y)^*1_u = \rho^2 \tilde{\sigma}HH^* (\pi_u\circ\tilde{\sigma})$, and our definition $\rho y_u = v$ with $\|v\|_{L^1} = 1$, we derive the following bounds:
    \begin{equation}\label{eq:bc-bounds}
        \begin{aligned}
            \|b\|_{-1 + \eta} &\lesssim \Big(\|v\|_{L^1} + \rho\|y_u\|_{L^1} \Big)\|v\|_{-1+\eta} + \rho\|y_u\|_{L^r} + (\|v\|_{L^2}^2 + \|v\|_{L^r}) \\
            &\lesssim \|v\|_{L^2}^2 \vee \|v\|_{L^r} \\
            \|c\|_{-1/2+\eta}^2 &\lesssim (\|v\|_{L^2}^2 + \|v\|_{L^r})^2 + \rho^2\|y_u\|_{L^r}^2 \\
            &\lesssim (\|v\|_{L^2}^2 \vee \|v\|_{L^r})^2 \,,
        \end{aligned}
    \end{equation}
    where the implicit constants are independent of $\epsilon$ (see the definition of $\sigma_\epsilon$ in \eqref{eq:sigma-stopping-time}). 
    Then \eqref{eq:desired-bc-bound} follows from $\rho_{t\wedge\sigma_\epsilon} \leq \epsilon^{-1}$ and $\E[\sup_{t \leq T} \|y_t\|_B^p] < \infty$ for all $p,T > 0$ (see \Cref{general-mild-existence} and note we assumed $r > N$ and $y = x_n$ is the mild solution of \eqref{eq:truncated}). So indeed we can apply \Cref{mild-strong} to deduce that \eqref{eq:dv-for-y} holds in the mild sense of \Cref{def:mild} up to time $\sigma_\epsilon$ for every fixed $\epsilon > 0$. To establish uniform bounds in $\epsilon$, we first use \Cref{convolution-bounds} (see also \Cref{rem:convolution-bounds-hold-locally}) and \eqref{eq:bc-bounds} to conclude that for any small $\eta > 0$
    \begin{equation}\label{eq:bound-on-v-sigma-epsilon}
        \E\Big[\int_0^{T \wedge \sigma_\epsilon} \|v_t\|_{\eta/2}^2\Big] \lesssim \|v_0\|_{L^r}^2 + \E\Big[\int_0^{T \wedge \sigma_\epsilon} \|v_t\|_{L^r}^2 \vee \|v_t\|_{L^2}^4\Big]
    \end{equation}
    for all $T,\epsilon > 0$, where the implicit constants are independent of $\epsilon$. Next, by Sobolev embedding $L^1 \subset W^{-N-1,r}$ for all $1 < r < \infty$, and thus with $a = (N+1)/(\eta+N+1) < 1$ we have by interpolation and $\|v\|_{L^1} = 1$ that
    \begin{equation}\label{eq:lrbwwr}
    \|v\|_{L^r} \leq \|v\|_{W^{\eta,r}}^a\|v\|_{W^{-N-1,r}}^{1-a} \lesssim \|v\|_{W^{\eta,r}}^a \,.    
    \end{equation}
    Then \Cref{characterization-of-interpolation-spaces} and Young's inequality imply that $\|v\|^2_{L^r}$ vanishes over $\|v\|_{\eta/2}^2$ (\Cref{fxn-vanish}). In addition, by \eqref{eq:lrbwwr} and H\"{o}lder's inequality with $b(r) = r/(2(r-1)) \to 1/2$ as $r \to \infty$
    $$
    \|v\|_{L^2}^4 \leq \|v\|_{L^r}^{4b(r)}\|v\|_{L^1}^{4(1-b(r))} = \|v\|_{L^r}^{4b(r)} \lesssim \|v\|_{\eta/2}^{4ab(r)}
    \,,$$
    and so $\|v\|_{L^2}^4$ (thus $\|v\|_{L^r}^2 \vee \|v\|_{L^2}^4$ as well) also vanishes over $\|v\|_{\eta/2}^2$ when $r$ is large enough. We conclude from \eqref{eq:bound-on-v-sigma-epsilon} that
    \begin{equation}\label{eq:bound-v-l2}
        \E\Big[\int_0^{T \wedge \sigma_\epsilon} \|v_t\|_{L^r}^2 \vee \|v_t\|_{L^2}^4\Big] \lesssim \E\Big[\int_0^{T \wedge \sigma_\epsilon} \|v_t\|_{\eta/2}^2\Big] \lesssim \|v_0\|_{L^r}^2  +1 \,,
    \end{equation}
    where the implicit constants are independent of $\epsilon$. 

    In the Step 5 below, we establish the following result. 
    
    \begin{lem}\label{lem:consep}
    We have $\sigma_\epsilon \to \infty$ as $\epsilon \to 0^+$ almost surely. 
    \end{lem}

    Then by \Cref{lem:consep},  \eqref{eq:desired-bc-bound} and \eqref{eq:bound-v-l2} hold with $\sigma_\epsilon$ removed and we conclude that $v_t$ is a mild solution in $\B$ (in the sense of \Cref{def:mild}) to \eqref{eq:dv-for-y}. Also, if $\|y\|_B \leq n$, then  $\rho \hat F(y) = \tilde F_u(R,y_u,y_z) = \tilde F_u(R\rho^{-1},\rho y_u,y_z)$ and similarly $\rho\hat G(y) = \tilde G_u(R\rho^{-1}, \rho y_u,y_z)$. Since $r = R\rho^{-1}$ and $v = \rho y_u$, the $dv$ portion of \eqref{new-dr-dv} follows for $x_n(t \wedge \tau_n)$ by \eqref{eq:dv-for-y}.

    To summarize, in \textit{Step 4} we have shown that \ref{bigthm:drdv} holds with $x_n(t \wedge \tau_n)$ in place of $x_t$.

    \textit{Step 5:}
    Next we fix $n$ and show that $y(t) = x_n(t)$ satisfies \ref{bigthm:invt-set} and \Cref{lem:consep} holds. First assume $u_0 = 0$ and observe that $\hat F(0,y_z) = (F_n)_u(0,y_z) = 0$ and $\hat G(0,y_z) = (G_n)_u(0,y_z) = 0$. Thus, if $z(t)$ is the solution to
    \begin{equation*}
            z(t) = S_z(t) z_0 + (S_z * F_n(0,z))_t + (S_z \diamond \tilde G(0,z))_t\,, \quad z(0) = z_0
    \end{equation*}
    (existence of $z$ follows in the same way as existence of $x_n$ in \textit{Step 1}), then $(0,z(t))$ satisfies \eqref{eq:truncated}. By the uniqueness of mild solutions (see \Cref{general-mild-existence}), $y(t) = (0,z(t))$, and therefore $y_u(t) = 0$. Thus, it remains to show that $\tilde u_0 \neq 0 \implies y_u(t) \neq 0$.
    
    Throughout the rest of this step we assume $\tilde u_0 \neq 0$. 
    Define the function $h: (0,\infty) \to [0,\infty)$ by $h(x) \coloneqq g(-\ln x)$, where $g: \R \to \R$ is a nonnegative smooth function such that  $g(x) = x$ on $[1,\infty)$, $g(x) = 0$ on $(-\infty,0]$, and $g$ has bounded first and second derivatives.
    
    Recall $\|y_u(t)\|_{L^1}$ is a $\R$-valued semi-martingale, $\rho_t \coloneqq \|y_u(t)\|^{-1}_{L^1}$, 
    $v_t = \rho_t y_u(t)$, and $\sigma_\epsilon$ is the stopping time defined in \eqref{eq:sigma-stopping-time}.
    
    Fix any small $\epsilon > 0$ and use \eqref{eq:d-l1-norm} and 
    It\^{o}'s formula to calculate 
    
    \begin{equation}\label{eq:h-ito}
        h(\|y_u(t \wedge \sigma_\epsilon)\|_{L^1}) = h(\|u_0\|_{L^1}) + \int_0^{t\wedge \sigma_\epsilon} \mu_{s }ds + \int_0^{t\wedge \sigma_\epsilon} \ip{\nu_{s}}{dW_s} \,,
    \end{equation}
    where
    \begin{align*}
        \mu_t &\coloneqq h'(\rho_t^{-1})[\ip{y_u(t)}{\tilde e_u} + \ip{\hat F(y_t)}{1_u}] + \frac{1}{2}h''(\rho_t^{-1})\|\hat G(y_t)^*1_u\|_\U^2 \\
        \nu_t &\coloneqq h'(\rho_t^{-1})\hat G(y_t)^*1_u \,.
    \end{align*}
If $\rho_t \in [e, \infty)$ for some $t$, then 
$|h'(\rho_t^{-1})| = \rho_t$ and $|h''(\rho_t^{-1})| = \rho_t^2$. 
If $\rho_t \in (0, e)$, then  
$|h'(\rho_t^{-1})| \leq \|g'\|_{L^\infty}|\rho_t|$ and $|h''(\rho_t^{-1})| \leq (\|g'\|_{L^\infty}) + \|g''\|_{L^\infty})|\rho_t|^2$. Consequently,
    \begin{align*}
         |\mu_t| &\lesssim |\ip{v_t}{\tilde e_u}| + \rho_t|\ip{\hat F(y_t)}{1_u}| + \rho_t^2\|\hat G(y_t)^*1_u\|_\U^2\\
        \|\nu_t\|_\U^2 &=  \rho_t^2\|\hat G(y_t)^*1_u\|_\U^2 \,.
    \end{align*}
Then using \eqref{eq:H-star-x}, global Lipschitz properties of $\tilde \sigma \circ \Psi_n$ and $\hat F$ such as $|\hat F(y)| \lesssim |y_u|$ (see \Cref{lem:lipschitz-f-g} for further explanation), we have 
    
    $$
    |\mu_t| + \|\nu_t\|_\U^2 \lesssim \|v_t\|_{L^2}^4 
     \,.
    $$
     It follows from \eqref{eq:bound-v-l2} that $\int_0^{t\wedge \sigma_\epsilon} \ip{\nu_{s }}{dW_s}$ is a martingale and that there is $M > 0$, independent of $\epsilon$, such that $\E[\int_0^{t\wedge \sigma_\epsilon} |\mu_{s }|ds] \leq M$. Using $h(\|y_u(\sigma_\epsilon)\|_{L^1}) = -\ln \epsilon$ (for small enough $\epsilon$), Chebyshev inequality, and  
    optional stopping one observes
    \begin{equation*}
        (-\ln \epsilon)\Prb(\sigma_\epsilon < t)\leq \E[h(\|y_u(t \wedge \sigma_\epsilon)\|_{L^1})] \leq h(\|u_0\|_{L^1}) + M \,,
    \end{equation*}
    and thus $\Prb(\sigma_\epsilon < t) \to 0$ as $\epsilon \to 0^+$. We conclude that $\sigma_\epsilon \to \infty$ almost surely, which means that $\|y_u(t)\|_{L^1} > 0$ for all $t$ and \Cref{lem:consep} holds, as desired.

    To summarize, in \textit{Step 5} we have shown that \ref{bigthm:invt-set} holds with $x_n(t)$ in place of $x_t$.

\textit{Step 6:}
    Next if $x_n$ solves \eqref{eq:truncated} we prove a uniform bound on $\E[\sup_{t \leq 1} \|x_n(t)\|_B^p]$ for all (large enough) $p$.
    Using the same argument as in \Cref{lq-to-lr} with $q = \infty$, we conclude that by choosing small $\eta > 0$ and $r$ large enough, we have for all $t \geq 0$
    \begin{equation}\label{eq:axxne}
    \|x_n(t)\|_B \lesssim \|x_0\|_B + \|(S * F_n^+(x_n))_t\|_{\eta/2} + \|(S \diamond G_n(x_n))_t\|_{\eta/2} \,,    \end{equation}
    where $F_n^+ \coloneqq F_n \vee 0$ and the constant in the $\lesssim$ depends only on $r,\eta$ (and $t$, but the dependence is uniform over $t \leq T$). By decreasing $\eta$ and choosing $\beta > -1/2$ sufficiently close to $-1/2$,  there are $\eta$ and $\beta$ satisfying respectively \Cref{lq-to-lr}\eqref{eq:lborc} (with $q = \infty$) and \Cref{SPDE-G-lem} and additionally $\eta = \beta + 1/2 > 0$.

    By \Cref{convolution-sup-bounds} and \eqref{eq:axxne} it follows that (for all large enough $p$) 
    $$\E\Big[\sup_{s \leq t} \|x_n(t)\|_B^p\Big] \lesssim \|x_0\|_B^p + \E\Big[\int_0^t \|F_n^+(x_n(s))\|^pds\Big] + \E\Big[\int_0^t \|G_n(x_n(s))\|_{\eta - 1/2}^pds\Big] \,.$$

    By \eqref{eq:local-u-bound-G}, $\|F_n^+(x)\| \lesssim \|x\| + 1 +R\|x\|$ and $\|G_n(x)\|_{\eta - 1/2} \lesssim \|x\| + R\|x\|$ (where the constants do not depend on $n$), so we conclude that
    \begin{align*}
        \E[\sup_{s \leq t} \|x_n(s)\|_B^p] &\lesssim \|x_0\|_B^p + \E\Big[\int_0^t 1+\|x_n(s)\|^pds\Big] 
        \\
        &\lesssim \|x_0\|_B^p + \E\Big[\int_0^t 1+\sup_{r \leq s}\|x_n(r)\|_B^pds\Big] \,.
    \end{align*}
    For every $T>0$, by Gronwall's inequality there is some $C > 0$ (independent of $n$ and uniform over $\theta \in \Theta$ and $R,T$ in bounded subsets of $[0,\infty)$) such that
    \begin{equation}\label{eq:ubonxn}
        \E[\sup_{t \leq T} \|x_n(t)\|_B^p] \leq C(1+\|x_0\|_B^p) \,.
    \end{equation}
    By Chebyshev's inequality we conclude $\lim_{n \to \infty} \Prb(\tau_n < T) = 0$ (see \eqref{eq:def-tau-n}), and thus $\sup_n \tau_n = \infty$ almost surely.
    
    \textit{Step 7:}
    Next we fix integers $m < n$ and define $\sigma \coloneqq \inf\{t > 0 \mid \|x_n(t)\|_B > m\}$. Note that $F_n(x_n(t \wedge \sigma)) = F_m(x_n(t \wedge \sigma))$ and $G_n(x_n(t \wedge \sigma)) = G_m(x_n(t \wedge \sigma))$. 
    Also, if $\tau$ is a bounded stopping time and $Z_s$ is such that $Z_s 1_{s \leq \tau} = 0$, then $\int_0^t S(t-s)Z_s 1_{s\leq \tau \wedge t}dW_s = 0$ and by \cite[Proposition 5.3]{umd-integration}, we have the following localization result: 
    $(S \diamond Z)_\tau \coloneqq \int_0^\tau S(\tau-s)Z_s dW_s = 0$. 
    Setting $\tau = t \wedge \sigma$ and $Z_s = G_n(x_n(s)) - G_m(x_n(s))$ and using the analogous (obvious) localization result for deterministic convolutions, we obtain
    $$
    x_n(t \wedge \sigma) = S(t\wedge \sigma)x_0 + (S * F_m(x_n))_{t \wedge \sigma} + (S \diamond G_m(x_n))_{t \wedge \sigma} \,.
    $$
    Then by our choice of $\eta$ as $r$ in \textit{Step 6} (cf. \eqref{eq:axxne}) we have
    \begin{align*}
        \|x_n(t \wedge \sigma) - x_m(t \wedge \sigma)\|_B &\lesssim \|(S * F_m(x_n))_{t \wedge \sigma} - (S * F_m(x_m))_{t \wedge \sigma}\|_{\eta/2} \\
        &+ \|(S \diamond G_m(x_n))_{t \wedge \sigma} - (S \diamond G_m(x_m))_{t \wedge \sigma}\|_{\eta/2} \,.
    \end{align*}
    Since $F_m:B \to \B$ and $G_m:B \to \B_{-1/2+\eta}$ are globally Lipschitz (see \textit{Step 1}), by \Cref{convolution-sup-bounds} we have (for all large enough $p$)
    $$\E\Big[\sup_{s \leq t} \|x_n(s \wedge \sigma) - x_m(s \wedge \sigma)\|_B^p\Big] \lesssim \E\Big[\int_0^t \|x_n(s \wedge \sigma) - x_m(s \wedge \sigma)\|_B^pds\Big] \,.$$
    We conclude by Gronwall's inequality that
    $$\E\Big[\sup_{t \leq T} \|x_n(t \wedge \sigma) - x_m(t \wedge \sigma)\|_B^p\Big] = 0 \,.$$
    On the event $\{\|x_n(t)\|_B \leq m-1 \text{ for all } t\}$, clearly $\tau_m \leq \sigma$ because they are both $\infty$. Otherwise there is some (random) $t < \sigma$ such that $m-1 < \|x_n(t)\|_B = \|x_m(t)\|_B < m$, and again we conclude $\tau_m \leq \sigma$. Thus, $x_n(t \wedge \tau_m) = x_m(t \wedge \tau_m)$ almost surely and $\tau_m \leq \tau_n$.

    It follows from $\sup_n \tau_n = \infty$ (\textit{Step 6}) that there exists a unique process $x_t$ which satisfies $x_{t \wedge \tau_n} = x_n(t \wedge \tau_n)$ for all $n$ and $t > 0$. Since $\Psi_n(x_n(t \wedge \tau_n)) = x_n(t \wedge \tau_n)$, then $x_t$ is a mild solution of \eqref{eq:mild-solution} (note that by \textit{Step 6} $\E[\sup_{t \leq T} \|x_t\|_B^p] < \infty$ for all $p$, and thus \Cref{def:mild} holds with $\epsilon = \eta$ by \eqref{eq:local-u-bound-G} and the polynomial boundedness of $\tilde F$ assumed in \Cref{SPDE-F}). If $z_t$ is another mild solution, then the same argument as presented above shows that $z_{t \wedge \tau_n} = x_n(t \wedge \tau_n)$, so \eqref{eq:mild-solution} has a unique mild solution.
    
    \textit{Step 8:} In the previous steps we have shown that $x_n(t)$ satisfies \ref{bigthm:nonneg}-\ref{bigthm:dbeta}, at least until $\tau_n$. We claim that $x_t$ also satisfies \ref{bigthm:nonneg}-\ref{bigthm:dbeta} for all $t \geq 0$. The only non-trivial proof is that of \ref{bigthm:cont-in-space}, so we present below and leave the rest to the reader.
    
    First,  Chebyshev inequality and \eqref{eq:ubonxn} imply
    \begin{equation}\label{eq:inxnx}
    \begin{aligned}
    \Prb\Big(x_n(t) \neq x(t) \text{ in } C([0,T];B)\Big) &\leq \Prb(\tau_n < T) \leq \frac{\E\Big[\|x_n(t)\|_{C([0,T];B)}^p\Big]}{(n-1)^{p}} \\
    &\leq 
    \frac{C(1+\|x_0\|_B^p)}{(n-1)^{p}} 
    \,,    
    \end{aligned}
    \end{equation}
    where $C$ (and thus the right-hand side) is bounded 
    uniformly in $\theta \in \Theta, R \leq M, x_0 \in \{\|\cdot\|_B \leq M\}$ for any fixed $M,T,p > 0$. Thus, if $x_n^{(k)}$ denote the solutions to \eqref{eq:truncated} with initial condition $(R^{(k)}, x_0^{(k)}, \theta^{(k)})$ and $(R^{(k)}, x_0^{(k)}, \theta^{(k)}) \to (R^{(\infty)},x_0^{(\infty)},\theta^{(\infty)})$ as $k \to \infty$ in $\R\times B \times \Theta$, then by \eqref{eq:inxnx} for all $\delta > 0$ there exists $n$ such that, for all 
    $k \in \N \cup \{\infty\}$, 
    $$
    \Prb\Big(x_n^{(k)} \neq x^{(k)} \text{ in } C([0,T];B)\Big) < \delta \,.
    $$
    For this fixed $n$, $x_n$ satisfies \ref{bigthm:cont-in-space} (see \textit{Step 1}), so we have $x_n^{(k)} \to x_n^{(\infty)}$ as $k \to \infty$ in probability. Then there is $K = K(n)$ such that, for all $k > K$, $\Prb\big(\|x_n^{(k)} - x_n^{(\infty)}\|_{C([0,T];B)} > \delta\big) < \delta$, and thus $\Prb\big(\|x^{(k)} - x^{(\infty)}\|_{C([0,T];B)} > \delta\big) < 3\delta$. It follows that $x^{(k)} \to x^{(\infty)}$ in probability, and thus also in $L^p(\Omega;C([0,T];B))$ for all $p > 0$ since by \eqref{eq:ubonxn} and monotone convergence we have  $\sup_k\E[\|x^{(k)}(t)\|_{C([0,T];B)}^{2p}] < \infty$. (Indeed, $x_n^{(k)}(\cdot \wedge \tau_n) \to x^{(k)}(\cdot)$ almost surely and $\|x_n^{(k)}(\cdot \wedge \tau_n)\|_{C([0,T];B)}$ is increasing in $n$ by \textit{Step 7}.)

\end{proof}

\begin{rem}\label{rmk:extra-approx}
    Although our proof of \Cref{bigthm} \ref{bigthm:nonneg} was inspired by \cite[Lemma 3.1]{nhu-positivity},
    
     our proof has extra approximation steps, because  
     \cite{nhu-positivity} uses the following unjustified statement: 
     \begin{equation}\label{eq:wsfop}
     \parbox{11cm}{
     \textrm{
     Suppose $f_2 = 0$ in \Cref{SPDE-F}. Since the resolvent is positivity preserving, $\phi'(x)R_mF(x \vee 0) = 0$
     for any function $\phi$ which satisfies $\phi(x) = 0$ for $x \geq 0$, and similarly other terms in the It\^o formula vanish.} 
     }
     \end{equation}
     It is not obvious why such claim should be true because here $x$ is a function and ``positivity preserving" is not a pointwise property. Suppose $x$ takes positive and negative values, so $x \vee 0$ is not constantly zero and there is some $a$ such that $x(a) < 0$. Consider the simple case of the 1-dimensional heat equation on the torus and $F(x) = x$. Even though $F(x \vee 0)(a) = 0$, $R_mF(x \vee 0)(a)$ is some weighted integral of the (nonzero, nonnegative) function $F(x \vee 0)$, and thus $R_mF(x \vee 0)(a) > 0$. Since $\phi'(x(a)) \neq 0$ in general, we cannot conclude that $\phi'(x(a))R_mF(x\vee 0)(a) = 0$.

    However, \eqref{eq:wsfop} is true without $R_m$, that is, since $(x \vee 0)(a) = 0$ implies that $F(x \vee 0)(a) = 0$. Thus we have $\phi'(x)F(x \vee 0) = 0$. So after applying the It\^ o formula we pass $m \to \infty$ to remove $R_m$. In order to justify this limit for the It\^ o correction (second-order) term we need $G(x) \in \gamma(\U,\B)$, which motivates the approximation of  $G = xH$ by $xR_kH$.
\end{rem}

\begin{cor}\label{well-posedness-of-SPDE}
Under \Cref{sass-A}-\ref{SPDE-G-sass}
    the problem \eqref{SPDE} is well-posed for $(x_0,\theta) \in \cdom \times \Theta$ in the sense that it has a unique mild solution (in $L^r$ for some large $r$) (\Cref{def:mild}) which is continuous in time and depends continuously on its initial condition.
\end{cor}
\begin{proof}
    The proof is a direct consequence of \Cref{bigthm} because \eqref{eq:mild-solution} and \eqref{SPDE} coincide when $R = 1$ (see \Cref{r-functions}).
\end{proof}

From the well-posedness of \eqref{eq:mild-solution} also follows the well-posedness of the projective process \eqref{eq:proj-process-first-def}. Below, we given an alternative definition of \eqref{eq:proj-process-first-def} (\Cref{def:rv-process}) and show that our definitions coincide (\Cref{rv-is-markov-process}).

\begin{deff}\label{def:rv-process}
    Given
    $$(r_0,x_0,\theta) \in \mcM \coloneqq [0,\infty) \times \{x \in \cdom \mid \|x_u\|_{L^1} = 1\} \times \Theta \,,$$
    let $\tilde{x}_t = (\tilde{u}_t, \tilde{z}_t)=(\tilde{u}_t, z_t)$ be the unique mild solution to
    $$
    \tilde{x}_t = S(t)x_0 + (S * \tilde F(r_0,\tilde{x},\theta))_t + (S \diamond \tilde G(r_0, \tilde{x},\theta))_t \,,$$
    with initial condition $\tilde{x}_0 = x_0$
    as guaranteed by \Cref{bigthm}. As in \Cref{bigthm} \ref{bigthm:drdv}, define $(r_t,v_t) \coloneqq (r_0\|\tilde{u}_t\|_{L^1}, \tilde{u}_t/\|\tilde{u}_t\|_{L^1})$.

    We also define $\mcM_0 \coloneqq \mcM \cap \{r = 0\}$ and $\inv \coloneqq \mcM \setminus \mcM_0$.
\end{deff}

\begin{cor}\label{rv-is-markov-process}
     Using notation from \Cref{def:rv-process},   $(r_t,v_t, z_t,\theta)$ forms a Markov process on $\mcM$ which satisfies \Cref{as1} and \Cref{as2}. Furthermore, $(r_tv_t,z_t)$ is the unique mild solution to \eqref{SPDE} with initial condition $(r_0v_0,z_0) = (r_0 \tilde{u}_0, \tilde{z}_0)$.
\end{cor}

\begin{proof}
Since $\|\tilde{u}_0\|_{L^1} = \|\pi_{u}x_0\|_{L^1} = 1$, we have 
$\tilde{u}_0 \neq 0$, and by \Cref{bigthm}\ref{bigthm:invt-set} 
$\tilde{u}_t \neq 0$ for all $t \geq 0$. Thus, $r_t = r_0\|\tilde{u}_t\|_{L^1} = 0$  if and only if $r_0 = 0$ and 
\Cref{as1} follows. 

    To show \Cref{as2}, that is, $\Pp_tf \in C_b(\mcM)$ and $\lim_{t \downarrow 0}\Pp_tf \to f$ pointwise for each $f \in C_b(\mcM)$, it is equivalent to show: 
    \begin{enumerate}[label=(\roman*)]
        \item $X_t^{x_n} \to X_t^x$ in distribution
    as $x_n \to x$.
    \item $\lim_{t \downarrow 0} X_t^x = x$ in distribution.
    \end{enumerate}
     Since $(r_0,u,z) \mapsto (r_0\|u\|_{L^1}, u/\|u\|_{L^1},z)$ is continuous as a map from 
     $[0,\infty) \times \{(u, z) \in \cdom : u \neq 0\}$ to $[0,\infty) \times \cdom$, (i) follows from \Cref{bigthm}\ref{bigthm:cont-in-space} and (ii) from \Cref{bigthm}\ref{bigthm:cont-in-time}.

    Finally, 
    applying $r_0\pi_u + \pi_z$ to both sides of \eqref{eq:mild-solution} and using that $\pi_u$ and $\pi_z$ commute with $S$,  shows that $\hat x_t \coloneqq (r_0 \tilde{u}_t,z_t) = (r_tv_t,z_t)$ satisfies
    $$
    \hat x_t = S(t)(r_0v_0,z_0) + (S * (r_0\pi_u + \pi_z)\circ\tilde F(r_0,\tilde x,\theta))_t + (S \diamond (r_0\pi_u + \pi_z)\circ\tilde G(r_0, \tilde x,\theta))_t 
    $$
    and $\hat x_0 = (r_0 v_0, z_0)$. 
    Hence,  the final claim follows from $(r_0\pi_u + \pi_z)\circ \tilde F(r_0,\tilde{x},\theta) = F(r_0 \tilde{u},z,\theta) = F(\hat x, \theta)$ and $(r_0\pi_u + \pi_z)\circ\tilde G(r_0, \tilde{x},\theta) = G(r_0 \tilde{u}, z,\theta) = G(\hat x, \theta)$ (see \eqref{eq:alt-def-tilde}) and uniqueness of mild solutions (\Cref{well-posedness-of-SPDE}).
\end{proof}

\subsection{Lyapunov Functions}\label{sec:SPDE-lyap}
In this section, we assume $(r_t,v_t,z_t,\theta)$ is the fixed Markov process on $\mcM$ defined in \Cref{def:rv-process} and $\beta$ is as in \Cref{SPDE-G-lem}. In most of the expressions below we omit the dependence on $\theta$ because it is irrelevant (for example, we write $(r,v,z) \in \mcM$). We also write $x_t = (r_tv_t,z_t)$ (see \Cref{rv-is-markov-process}).

We introduce the Lyapunov functions which are central to our analysis. In particular, we introduce functions which satisfy \Cref{as-W} and \Cref{as-U}\ref{5.1}-\ref{5.3}, as well as the function satisfying \Cref{as-V}. We encourage the reader to recall $\Ll$ and $\Gamma$ 
introduced respectively in \Cref{D+} and \Cref{D2}.
 
\begin{lem}\label{generator-for-SPDE}
    Let $\mathcal{A} = \mcM$ or $\mathcal{A} = \inv \coloneqq \mcM_0^c$ and $f: \mathcal{A} \to \R$, $b: \mathcal{A} \to \R \cup \{-\infty\}$, and $\sigma: \mathcal{A} \to \Ll_2(\U, \R) \cong \U$ be measurable functions such that $df(r,v,z) = b(r,v,z)dt + \sigma(r,v,z) dW_t$ holds, meaning that for all $(r_0,v_0,z_0) \in \mathcal{A}$ 
    we have 
    \begin{equation}\label{eq:fieq}
        f(r_t,v_t,z_t) = f(r_0,v_0,z_0) + \int_0^t b(r_s,v_s,z_s)ds + \int_0^t \sigma(r_s,v_s,z_s)dW_s
    \end{equation}
    and all the integrals in \eqref{eq:fieq} are well-defined. Then the following statements hold:

    \begin{enumerate}[label = (\roman*)]
        \item If $f \geq 0$, then $f \in \Dme_+(A)$ with $\Ll f = b$.
        \item Assume $\hat f, \hat b, \hat \sigma$ satisfy the same assumptions as $f,  b,  \sigma$
        above, $\hat f \geq 0$, and there are some constants $K, C > 0$ such that $\sup_A (\|\sigma\|_\U^2 + K \hat b) =: C < \infty$. Then $f \in \Dme_2(A)$ with $\Ll f = b$ and $\Gamma f = \|\sigma\|_\U^2$.
    \end{enumerate}
\end{lem}

\begin{proof}
   Note that (i) directly follows from the definitions and properties of (stochastic) integration. To prove (ii), first note that by 
   (i) we have $\hat f \in \Dme_+(A)$ and that $\Ll \hat f = \hat b \leq (C - \|\sigma\|_\U^2)/K$. Then as in  \cite[Lemma 4.1]{extinction} we obtain that $\E[\int_0^t \|\sigma(r_s,v_s,z_s)\|_\U^2ds] < \infty$, and consequently $\int_0^t \sigma(r_s,v_s,z_s)dW_s$ is a square integrable martingale with quadratic variation $\int_0^t \|\sigma(r_s,v_s,z_s)\|_\U^2ds$, proving the claim.
\end{proof}

We use \Cref{generator-for-SPDE} to prove that the Lyapunov functions defined below indeed belong to the extended domains $\Dme_+$ and $\Dme_2$. Even if the function $f$ from \Cref{generator-for-SPDE} is not known yet to be in $\Dme_2$, we make the convention of writing $\Gamma f = \|\sigma\|_\U^2$. This leads to following notation:

\begin{deff}\label{spde-tilde-domain}
    If $\mathcal{A},f,b,\sigma$ are as in \Cref{generator-for-SPDE} and $f \geq 0$, then we write $f \in \tDme_+(\mathcal{A})$ (this is a definition of $\tDme_+(\mathcal{A})$) and formally set $\Ll f = b$ and $\Gamma f = \|\sigma\|_\U^2$.
\end{deff}

Let be $h: (0,\infty) \to [0,\infty)$ given by $h(x) \coloneqq g(-\ln x)$, where $g: \R \to \R$ is a nonnegative smooth function such that
\begin{equation}\label{eq:dfgc}
    g(x) = 
    \begin{cases}
        x & x\geq 1 \\
        0 & x \leq 0
    \end{cases}
\end{equation}
 and $g$ has bounded first and second derivatives. Our average Lyapunov function for \Cref{as-V} is
\begin{equation}\label{def-of-V-for-SPDE}
    V(r,v,z) \coloneqq h(r) \,.
\end{equation}

\begin{lem}\label{l-v-lemma}
For $V$ as in \eqref{def-of-V-for-SPDE} we have $V \in \tDme_+(\inv)$ (see \Cref{spde-tilde-domain}) and
    \begin{equation*}
    \begin{aligned}
        \Ll V(r,v,z) &= h'(r)r\Big[\ip{v}{\tilde e_u} + \ip{\tilde F_u(r,v,z)}{1_u}\Big] + \frac{1}{2}h''(r)r^2\|\tilde G_u(r,v,z)^*1_u\|_U^2 \\
        \Gamma V(r,v,z) &= h'(r)^2r^2\|\tilde G_u(r,v,z)^*1_u\|_U^2 \,,
    \end{aligned}
\end{equation*}
where the adjoint of $G_u(r,v,z)$ is taken in $\Ll(\U,L^2)$ and $\tilde e$ is as in \Cref{def:tilde-e}.
\end{lem}
\begin{proof}
    Recall that by \Cref{bigthm}\ref{bigthm:drdv} the process $(r_t,v_t,z_t)$ satisfies \eqref{new-dr-dv}. Then the claim follows immediately from It\^{o}'s formula and \Cref{generator-for-SPDE}.
\end{proof}

To construct $W$ for \Cref{as-W}, we combine the following functions (here $x = (rv, z)$):
\begin{equation}\label{w123}
\begin{aligned}
    W_1(r,v,z) &\coloneqq 1+\frac{1}{2}\|x\|_{\beta,2}^2 = 1+\frac{1}{2}\|(rv,z)\|_{\beta,2}^2 \\
    W_2(r,v,z) &\coloneqq 1+\|x\|_{L^1} = 1+\ip{x}{1} = 1+r\ip{v}{1_u} + \ip{z}{1_z} \,.
\end{aligned}
\end{equation}

\begin{lem}\label{rigorous-w123}
    For $W_1, W_2$ from \eqref{w123} we have $W_1,W_2 \in \tDme_+(\mcM)$ and with $x = (rv, z)$:
    \begin{equation}\label{l-gamma-w1}
    \begin{aligned}
        \Ll W_1(r,v,z) &= -\|x\|_{\beta+1/2,2}^2 +\ip{x}{F(x)}_\beta + \frac{1}{2}\|G(x)\|_{\beta,2}^2 \\
        \Gamma W_1(r,v,z) &= \|G(x)^*x\|_\U^2 \,,
    \end{aligned}
\end{equation}
where the adjoint of $G(x)$ is taken in $\Ll(\U,\B_{\beta,2})$. Moreover, 
\begin{equation}\label{l-gamma-w2}
    \begin{aligned}
        \Ll W_2(r,v,z) &= \ip{x}{\tilde e} + \ip{F(x)}{1} \\
        \Gamma W_2(r,v,z) &= \|G(x)^*1\|_\U^2 \,,
    \end{aligned}
\end{equation}
where the adjoint of $G(x)$ is taken in $\Ll(\U,L^2)$ and $\tilde{e}$ is as in \Cref{def:tilde-e}.
\end{lem}

\begin{proof}
    Recall that by \Cref{bigthm}\ref{bigthm:dbeta} the process 
    $x_t = (r_tv_t, z_t)$ (denoted $\bar{x}$ in \Cref{bigthm}\ref{bigthm:dbeta}) satisfies \eqref{eq:dbeta-x-tilde} with $\bar{x}$ replaced by $x$ and \eqref{l-gamma-w1} follows.
    
Using that $x$ satisfies \eqref{SPDE} (\Cref{rv-is-markov-process}), applying It\^{o}'s formula to $\Phi(x) = \ip{x}{1}$ shows
    $$d\|x\|_{L^1} = [\ip{x}{\tilde e} + \ip{F(x)}{1}]dt + \ip{G(x)^*1}{dW_t} \,,$$
    and \eqref{l-gamma-w2} holds. (Technically since we work with mild solutions one cannot apply It\^{o}'s formula directly. Since the details of the truncation + Yosida approximation argument that handles this subtelty are contained in the proof of \Cref{bigthm}, we do not repeat the argument here. See the proof of \eqref{eq:d-l1-norm}.)

    The claim $W_1,W_2 \in \tDme_+(\mcM)$ follows from \Cref{generator-for-SPDE}.
\end{proof}

\begin{cor}\label{cor:w1-w2-p-in-dme}
    For all $p,C > 0$ we have $W_1^{p/2} + CW_2^p \in \tDme_+(\mcM)$ and the following hold with $x = (rv, z)$:
\begin{multline*}
\begin{aligned}
    \Ll (W_1^{p/2} + CW_2^p) &= \frac{p}{2}W_1^{p/2 - 1}\Ll W_1 + \frac{p}{4}(p/2-1)W_1^{p/2-2}\Gamma W_1 \\ 
    &\qquad + pCW_2^{p-1}\Ll W_2 + \frac{p}{2}(p-1)CW_2^{p-2}\Gamma W_2 
\end{aligned}
    \\
\begin{aligned}
    &= \frac{p}{2^{p/2}}(2 + \|x\|_{\beta,2}^2)^{p/2-1}\Big(-\|x\|_{\beta+1/2,2}^2 +\ip{x}{F(x)}_\beta + \frac{1}{2}\|G(x)\|_{\beta,2}^2\Big) \\
    &\qquad + \frac{p}{2^{p/2}}(p/2-1)(2+\|x\|_{\beta,2}^2)^{p/2-2}\|G(x)^*x\|_\U^2 \\
    &\qquad + pC(1+\|x\|_{L^1})^{p-1}\Big(\ip{x}{\tilde e} + \ip{F(x)}{1}\Big) \\    
    &\qquad + \frac{p}{2}(p-1)C(1+\|x\|_{L^1})^{p-2}\|G(x)^*1\|_\U^2 
\end{aligned}
\end{multline*}
and
\begin{multline*}
    \Gamma(W_1^{p/2} + CW_2^p) \lesssim (p/2)^2W_1^{p-2}\Gamma W_1 + p^2C^2W_2^{2p-2}\Gamma W_2 \\
    = \frac{p^2}{2^p}(2+\|x\|_{\beta,2}^2)^{p-2}\|G(x)^*x\|_\U^2 + p^2C^2(1+\|x\|_{L^1})^{2p-2}\|G(x)^*1\|_\U^2  \,.
\end{multline*}    
\end{cor}

\begin{proof}
    Follows immediately from \Cref{rigorous-w123}, It\^{o}'s formula, and \Cref{generator-for-SPDE}.
\end{proof}

\begin{sass}\label{SPDE-lyap}
There are $K,c > 0$ such that
\begin{equation}\label{eq:l-l1-bound}
   \ip{x}{\tilde e} + \ip{F(x)}{1} = \ip{Ax + F(x)}{1} \leq K - c\|x\|_{L^1} \,,
\end{equation}
for all $x \in \Dm(A) \cap \cdom$ ($\tilde e$ is as in \Cref{def:tilde-e}).

Let $F^+(x) \coloneqq F(x) \vee 0$ and $F^-(x) \coloneqq -(F(x) \wedge 0)$. Then (at least) one of the following hold:
\begin{enumerate}[label = (\roman*)]
    \item \label{case-1} With $p$ as in \Cref{SPDE-G-sass}, there is some $\beta \in (-1/2,-N(\frac{1}{4} - \frac{1}{2p}))$ (or we can take $\beta = 0$ if $p = 2$) such that the function
    $$\cdom \ni x \mapsto \|x\|_{\beta,2}^2 \in \R$$
    vanishes over (see \Cref{fxn-vanish}) 
    $$\cdom \ni x \mapsto \ip{x}{F^-(x)}_\beta \in \R \,.$$
    \item \label{case-2} $|F_u(x)| \lesssim |x_u|$ and $|F(x)| \lesssim 1+|x|$.
\end{enumerate}
\end{sass}

If (i) is assumed in \Cref{SPDE-lyap}, then we show that there is a nonlinear negative drift and the process enjoys several bounds in expectation. However, such bounds are not available if (ii) is assumed, which occurs when all terms are linear or lower order.

\begin{rem}
    For example, consider the SPDE on the torus
    \begin{equation}\label{eq:eospde}
    du = [u_{xx} - u]dt + 10udB_t \,,    
    \end{equation}
    where $B_t$ is a standard (one-dimensional) Brownian motion. One solution to \eqref{eq:eospde} is $u_t \equiv a_t$, where $a_t$ is a real-valued process, constant in the spacial variable $x$, satisfying
    $$da = -adt + 10adB_t \,,$$
    which has the explicit solution
    $$a_t = a_0\exp(-51t + 10B_t) \,.$$
    Then 
    $$\E[a_t^2] = a_0^2e^{98t} \,,$$
    so that
    $$\E\Big[\frac{1}{T}\int_0^T \|a_t\|_{L^1}^2dt\Big] \to \infty \quad \text{as} \quad T \to \infty \,.$$
    In the context of \Cref{as-W}, this means that we cannot set $W = W' = \|u\|_{L^1}$, because $\Gamma W = 100\|u\|_{L^1}^2$ cannot satisfy \Cref{as-U}\ref{5.3} for any $g \in \Y$ (all $g \in \Y$ satisfy $\E\big[\frac{1}{T}\int_0^T g(u_t)\|dt\big] \lesssim 1$). Thus, if we wanted to use expectation of norms as a tool for obtaining bounds, we would need to work with $W = \|u\|_{L^1}^p$ for some $p < 1$. Indeed, we could use \Cref{convolution-bounds} to show that there is some $p \in (1/2,1)$ such that 
    $$\limsup_{T \to \infty}\E\Big[\frac{1}{T}\int_0^T \|u_t\|_{W^{\eta,r}}^{2p}dt\Big] < \infty \,,
    $$
    where $W^{\eta,r} \subset L^\infty$ compactly (this type of argument will be given in \Cref{sec:SPDE-semigrp-method}). However, this does not immediately imply $\|x\|_{W^{\eta,r}}^p \in \X$, and we cannot apply the same argument with $2p$ replaced by $p$ because \Cref{convolution-bounds}(iii) requires $p > 1$.

    This sort of example is why we need two separate cases in \Cref{SPDE-lyap}. Note that if we are only interested in existence of invariant measures (point (iii) in \Cref{main}), then by \Cref{rem:weaker-thm} the issue above is irrelevant and we would not need these separate cases.
\end{rem}

The following lemma is useful for verifying \Cref{as-W} and \Cref{as-U}\ref{5.1}-\ref{5.3} at the same time:

\begin{lem}
\label{Wk-lemma}
    Suppose $W \in \tDme_+(\mcM)$ (see \Cref{spde-tilde-domain}) and $W': \mcM \to [0,\infty]$ are nonnegative functions satisfying:
    \begin{itemize}
        \item $\Ll W \leq K - W'$ for some constant $K > 0$.
        \item $\Gamma W / W$ vanishes over $W'$ (see \Cref{fxn-vanish}).
        \item $W << W'$ (see \Cref{<<}).
    \end{itemize}
    Then for all $k \geq 1$, $W^k \in \Dme_2(\mcM)$ and there is a constant $K' > 0$ such that \begin{equation}\label{Wk-lemma-equation}
        \Ll W^k \leq K' - \frac{k}{4}W^{k-1}W' \quad \text{and} \quad \Gamma W^k \leq K' + k^2W^{2k-1}W' \,.
    \end{equation}
\end{lem}

\begin{proof}
    Recalling \Cref{generator-for-SPDE}, we start by applying It\^{o}'s formula to obtain:
    \begin{equation}\label{eq-l-wk}
        \Ll W^k = kW^{k-1}\Ll W + \frac{1}{2}k(k-1)W^{k-2}\Gamma W
    \end{equation}
    \begin{equation}\label{eq-gamma-wk}
        \Gamma W^k = k^2W^{2k-2}\Gamma W \,.
    \end{equation}
    In \eqref{eq-l-wk} we use our assumptions $\Ll W \leq K - W'$ and $\Gamma W/W$ vanishes over $W'$ to conclude that there is some constant $K' > 0$ such that
    \begin{align*}
        \Ll W^k &= W^{k-1}(k\Ll W + \frac{1}{2}k(k-1)\Gamma W/W) \\
        &\leq W^{k-1}(kK - kW' + \frac{1}{2}k(k-1)\Gamma W/W) \\
        &\leq W^{k-1}(K' - \frac{k}{2}W')
    \end{align*}
    Next we note that if $kW'/2 > 2K'$, then $K' - kW'/2 < -kW'/4$. Also, by our assumption $W << W'$ there is some $K''$ such that $kW'/2 \leq 2K'$ implies $W \leq K''$. Thus,
    $$\Ll W^k \leq K'(K'')^{k-1}-\frac{k}{4}W^{k-1}W' \,.$$

    Similarly, from \eqref{eq-gamma-wk} we have \begin{equation*}
        \Gamma W^k = W^{2k-1}\Big(k^2 \frac{\Gamma W}{W}\Big) \leq W^{2k-1}(K''' + \frac{1}{2}k^2W') \leq K'''(K'''')^{2k-1} + k^2W^{2k-1}W' \,,
    \end{equation*}
    where $K'''$ is such that $\Gamma W / W \leq K'''/k^2 + W'/2$ and $K''''$ is such that $k^2W'/2 \leq K'''$ yields $W \leq K''''$.

    This finishes the proof of \eqref{Wk-lemma-equation}. Finally, the statement $W^k \in \Dme_2(\mcM)$ follows from \Cref{generator-for-SPDE} applied with $f = W^k$, $\hat f = W^{2k}$.
\end{proof}

\begin{lem}\label{lem:lyapunov-fxn-construction-spde}
    
    \Cref{as-W} and \Cref{as-U}\ref{5.1}-\ref{5.3} hold for $(r_t,v_t,z_t,\theta)$ as in \Cref{rv-is-markov-process}.
    Furthermore, for $x = (rv, z)$ we have 
    \begin{enumerate}[label=(\roman*)]
        \item If \Cref{SPDE-lyap}\ref{case-1} holds, then for all $p > 1$, $W, W', U, U'$ can be chosen such that  $1+\|x\|_{L^1}^p + \|x\|_{L^2}^2 \lesssim W'$ and $1+\|x\|_{L^1}^p + \|x\|_{L^2}^2 \lesssim U'$.
        \item If \Cref{SPDE-lyap}\ref{case-2} holds, then there is some $p > 1$ such that $1+\|x\|_{L^2}^{p/2} \lesssim W'$ and $1+\|x\|_{L^2}^p \lesssim U'$.
    \end{enumerate}
\end{lem}
\begin{proof}
    By Cauchy-Schwartz inequality, $\beta \leq 0$, and \Cref{SPDE-F} (in particular the last bullet point) we have
\begin{equation}\label{eq:fbond}
\ip{x}{F^+(x)}_\beta \leq \|x\|_{\beta,2}\|F^+(x)\|_{\beta,2} \lesssim \|x\|_{L^2}\|F^+(x)\|_{L^2} \lesssim \|x\|_{L^2}^2+1 \,.
\end{equation}
By \Cref{SPDE-G-lem} and \Cref{SPDE-G-sass} (recall that for the following inequality the adjoint of $G(x)$ is taken in $\Ll(\U,\B_{\beta,2})$)
\begin{equation}\label{eq:gbond}
\|G(x)^*x\|_\U^2 \leq \|G(x)\|_{\beta,2}^2\|x\|_{\beta,2}^2 \lesssim \|x\|_{L^2}^2\|x\|_{\beta,2}^2 \,.    
\end{equation}
Combining \Cref{cor:w1-w2-p-in-dme}, Young's inequality, \eqref{eq:fbond}, \eqref{eq:gbond}, $\|G(x)\|_{\beta,2} \lesssim \|x\|_{L^2}^2$ (see \Cref{SPDE-G-lem}), and \eqref{eq:l-l1-bound} imply the existence of $K,c, c_0 > 0$ (which may change from line to line) such that for any $p \leq 2$ 
\begin{equation*}
    \begin{aligned}
    \Ll (W_1^{\frac{p}{2}} + CW_2^p) &\leq \frac{p}{2^{\frac{p}{2}}}(1+\|x\|_{\beta,2}^2)^{\frac{p}{2}-1}\Big(-\|x\|_{\beta+1/2,2}^2 -\ip{x}{F^-(x)}_\beta + c_0\|x\|_{L^2}^2\Big) \\
    &\qquad + pCK - pCc\|x\|_{L^1}^p + \frac{p}{2}(p-1)C(1+\|x\|_{L^1})^{p-2}\|G(x)^*1\|_\U^2
    \\
    \Gamma(W_1^{p/2} + CW_2^p) &\lesssim \frac{p^2}{2^p}(1+\|x\|_{\beta,2}^2)^{p-1}\|x\|_{L^2}^2 + p^2C^2(1+\|x\|_{L^1})^{2p-2}\|G(x)^*1\|_\U^2 \,.
    \end{aligned}
\end{equation*}
By \Cref{SPDE-G-sass} and $\beta > -1/2$ it follows easily that
\begin{equation}\label{eq:g(x)-star-1}
    \|G(x)^*1\|_\U = \|H^* \sigma(x)\|_\U \lesssim \|x\|_{L^2} \lesssim \|x\|_{\beta+1/2,2}^2
\end{equation}
(recall $H \in \Ll(\U,L^2)$, and therefore $H^* \in \Ll(L^2,\U)$). Then we conclude by \eqref{eq:g-star-1} that
\begin{equation*}
    (1+\|x\|_{\beta,2}^2)\|G(x)^*1\|_\U^2 \lesssim (1+\|x\|_{L^1})^2\|x\|_{\beta+1/2,2}^2 \,.
\end{equation*}
For $p \leq 2$, we can raise the above inequality to the power of $1-p/2$ on both sides and combine this with $\|G(x)^*1\|_\U^p \lesssim \|x\|_{L^2}^p$ (see \eqref{eq:g(x)-star-1}) to conclude that
\begin{equation}\label{eq:turning-l1-to-beta}
    (1+\|x\|_{L^1})^{p-2}\|G(x)^*1\|_\U^2 \lesssim (1+\|x\|_{\beta,2}^2)^{p/2-1}\|x\|_{\beta+1/2,2}^{2-p}\|x\|_{L^2}^p \,.
\end{equation}

Thus, after using Young's inequality, for every $C > 0$ there is some $p^*(C) \in (1,2]$ such that for all $p \leq p^*(C)$
\begin{equation}\label{eq:initial-w1-w2-bd}
    \begin{aligned}
    \Ll (W_1^{\frac{p}{2}} + CW_2^p) &\leq \frac{p}{2^{\frac{p}{2}}}(1+\|x\|_{\beta,2}^2)^{\frac{p}{2}-1}\Big(-\frac{1}{2}\|x\|_{\beta+\frac{1}{2},2}^2 - \ip{x}{F^-(x)}_\beta +c_0\|x\|_{L^2}^2\Big) \\
    &\qquad + pCK - pCc\|x\|_{L^1}^p \,,
    \end{aligned}
\end{equation}
and, by replacing $p$ with $2p$ in \eqref{eq:turning-l1-to-beta}, for $p \leq 1$
\begin{equation}\label{eq:gamma-w1-w2-bd}
    \Gamma(W_1^{p/2} + CW_2^p) \lesssim (1+\|x\|_{\beta,2}^2)^{p-1}\|x\|_{\beta+1/2,2}^{2-2p}\|x\|_{L^2}^{2p} \,.
\end{equation}

\textit{Case 1:}
Suppose \Cref{SPDE-lyap}\ref{case-1} holds. Since $-1/2 < \beta \leq 0$, for all $\epsilon > 0$ there is $C_\epsilon$ such that
\begin{equation}\label{eq:ibbabh}
\|x\|_{L^2}^2 \leq \epsilon \|x\|_{\beta+1/2,2}^2 + C_\epsilon \|x\|_{\beta,2}^2 \,,    
\end{equation}
so we conclude from \Cref{SPDE-lyap}\ref{case-1} that $\|x\|_{L^2}^2$ vanishes over $\frac{1}{2}\|x\|_{\beta+1/2,2}^2 + \ip{x}{F^-(x)}_\beta$.

Then from \eqref{eq:initial-w1-w2-bd} we obtain for $p \leq p^*(C)$ that
\begin{equation*}
    \Ll (W_1^{p/2} + CW_2^p) \leq K -pc(1+\|x\|_{\beta,2}^2)^{p/2-1}\Big(\|x\|_{\beta+1/2,2}^2 + \ip{x}{F^-(x)}_\beta\Big) - pc\|x\|_{L^1}^p
\end{equation*}

By \Cref{cor:w1-w2-p-in-dme},  
$$W \coloneqq W_1^{p^*(C)/4} + CW_2^{p^*(C)/2} \in \tDme_+(\mcM)$$
and $\Ll W \leq K - W'$ for
\begin{equation}\label{eq:def-of-w'-case-1}
    W' \coloneqq c(1+\|x\|_{\beta,2}^2)^{p^*(C)/4-1}\Big(\|x\|_{\beta+1/2,2}^2+ \ip{x}{F^-(x)}_\beta\Big)  + c\|x\|_{L^1}^{p^*(C)/2}  + 1 \,,
\end{equation}

which clearly satisfies $W << W'$ (see \Cref{<<}. In fact $W \lesssim W'$ because $\|x\|_{\beta,2} \lesssim \|x\|_{\beta+1/2,2}$ and, for $p \geq 0$, $(1+a)^p \lesssim (1+a)^{p-1}a + 1$ as functions of $a \in [0,\infty)$.)

By \eqref{eq:gamma-w1-w2-bd} with $p = p^*(C)/2$, 
$W \geq (W_1)^{p^*(C)/4}$, interpolation of $L^2 = \B_{0,2}$ between
$\B_{\beta,2}$ and $\B_{\beta+1/2,2}$ and Young's inequality (as in \eqref{eq:ibbabh}), we have
\begin{align*}
    \frac{\Gamma W}{W} &\leq K(1+\|x\|_{\beta,2}^2)^{p^*(C)/4-1}\|x\|_{\beta+1/2,2}^{2-p^*(C)}\|x\|_{L^2}^{p^*(C)} \\
    &\leq \epsilon (1+\|x\|_{\beta,2}^2)^{p^*(C)/4-1}\|x\|_{\beta+1/2,2}^2 + C_\epsilon(1+\|x\|_{\beta,2}^2)^{p^*(C)/4-1}\|x\|_{\beta,2}^2 \,. 
\end{align*}
Thus,  by \Cref{SPDE-lyap}\ref{case-1} and  we conclude that $\Gamma W/W$ vanishes over $W'$. We have verified the assumptions of \Cref{Wk-lemma}, so by \Cref{u'-in-y} and \Cref{generator-for-SPDE} we have that
\Cref{as-W} and \Cref{as-U}\ref{5.1}-\ref{5.3}  hold with  $W = W^k$ and $U = W^{2k}$, for any $k \geq 1$. By \eqref{eq:def-of-w'-case-1} and $\|x\|_{L^2}^2 \lesssim \|x\|_{\beta+1/2,2}^2$, we indeed have that $1+\|x\|_{L^1}^p + \|x\|_{L^2}^2 \lesssim W'$ and $1+\|x\|_{L^1}^p + \|x\|_{L^2}^2 \lesssim U'$.

\textit{Case 2:}
Suppose \Cref{SPDE-lyap}\ref{case-2} holds, so, using $\beta \leq 0$, 
$$
|\ip{x}{F^-(x)}_\beta| \leq \|x\|_{\beta,2}\|F^-(x)\|_{\beta,2} \lesssim \|x\|_{L^2}\|F^-(x)\|_{L^2}\lesssim 1+\|x\|_{L^2}^2 \,.
$$
Then from \eqref{eq:initial-w1-w2-bd} we obtain (for $p \leq 2$)
that (recall we may change $K,c,c_0$ from line to line)
\begin{equation*}
    \begin{aligned}
    \Ll (W_1^{p/2} + CW_2^p) &\leq \frac{p}{2^{p/2}}(1+\|x\|_{\beta,2}^2)^{p/2-1}\Big(-\frac{1}{2}\|x\|_{\beta+1/2,2}^2 + c_0\|x\|_{L^2}^2\Big) \\
    &+ pCK - pCc\|x\|_{L^1}^p \,.
    \end{aligned}
\end{equation*}
Using \Cref{characterization-of-interpolation-spaces} and Gagliardo-Nirenberg inequality,  we conclude that there is $a \in [0,1]$ such that
$$\|x\|_{L^2} \leq \|x\|_{\beta+1/2,2}^a\|x\|_{L^1}^{1-a}$$
and by $-1/2 < \beta \leq 0$ we have for some $b \in [0,1]$ that
$$\|x\|_{L^2} \leq \|x\|_{\beta+1/2,2}^b\|x\|_{\beta,2}^{1-b} \,.$$
Thus, 
$$\|x\|_{L^2}^2 \leq (\|x\|_{\beta+1/2,2}^a\|x\|_{L^1}^{1-a})^c(\|x\|_{\beta+1/2,2}^b\|x\|_{\beta,2}^{1-b})^{2-c}$$
where 
$c \in [0,2]$ is chosen such that $c(1-a)(2-p) = (2-c)(1-b)p$ (it is possibe for example by the intermediate value theorem). Then by Young's inequality we obtain that for all $\epsilon > 0$ there is $C_\epsilon > 0$ such that
$$\|x\|_{L^2}^2 \leq \epsilon\|x\|_{\beta+1/2,2}^2 + C_\epsilon\|x\|_{L^1}^p\|x\|_{\beta,2}^{2-p} \,.$$
By choosing $\epsilon$ small enough and then $C$ large enough, we obtain that for all $p \leq p^*(C)$
\begin{equation*}
    \Ll (W_1^{p/2} + CW_2^p) \leq K -pc(1+\|x\|_{\beta,2}^2)^{p/2-1}\|x\|_{\beta+1/2,2}^2 - pc\|x\|_{L^1}^p \,.
\end{equation*}
Thus, by \Cref{generator-for-SPDE} and \Cref{cor:w1-w2-p-in-dme} the functions $W \coloneqq W_1^{p^*(C)/4} + CW_2^{p^*(C)/2}$ and $W' \coloneqq c(1+\|x\|_{\beta,2}^2)^{p^*(C)/4-1}\|x\|_{\beta+1/2,2}^2 + c\|x\|_{L^1}^{p^*(C)/2} + 1$ satisfy \Cref{as-W}. Similarly, using \eqref{eq:gamma-w1-w2-bd} and $\|x\|_{L^2} \lesssim \|x\|_{\beta+1/2,2}$ we see that $U \coloneqq W_1^{p^*(C)/2} + CW_2^{p^*(C)}$ and $U' \coloneqq c(1+\|x\|_{\beta,2}^2)^{p^*(C)/2-1}\|x\|_{\beta+1/2,2}^2 + c\|x\|_{L^1}^{p^*(C)} + 1$ satisfy \Cref{as-U}\ref{5.1}-\ref{5.3} (see \Cref{u'-in-y}). The claims $1+\|x\|_{L^2}^{p/2} \lesssim W'$ and $1+\|x\|_{L^2}^p \lesssim U'$ follow since $\|x\|_{\beta,2} \lesssim \|x\|_{L^2} \lesssim \|x\|_{\beta+1/2,2}$.
\end{proof}

\subsection{Semigroup Method Bounds}\label{sec:SPDE-semigrp-method}
In \Cref{sec:SPDE-lyap} we used variational methods to construct certain Lyapunov functions (\Cref{lem:lyapunov-fxn-construction-spde}), but they do not provide adequate control over $(r_t,v_t,z_t)$ (see \Cref{def:rv-process}). Indeed, we have some control over the magnitude of $x_t = (rv_t,z_t)$ in spaces like $L^1$ or $\B_{\beta+1/2,2}$, but not in $L^\infty$ (unless the dimension of our space is small enough and the noise term smooth enough that $\B_{\beta+1/2,2} \subset L^\infty$). Even more importantly, we have no control over the ``angle" $v_t$. In this section we will use the mild formulation of SPDE to ``upgrade" the bounds given to us by our Lyapunov functions. In particular, we verify \Cref{as-V}\ref{as-vanish}, \Cref{as-U}\ref{5.4}, and \Cref{as-compact}.

\begin{lem}\label{lem:bootstrap-inductive-step}
    Define $\eta = 1/2 + \beta > 0$, where $\beta$ is as in  \Cref{SPDE-G-lem} and recall that $N$ is the dimension of domain $D$. For all $T > 0$, $p > 1$, $q \in (2,\infty]$ we have
    $$\E\Big[\int_0^T \|x_t\|_{L^q}^pdt\Big] \lesssim \|x_0\|_{L^r}^p + \E\Big[\int_0^T 1+\|x_t\|_{L^r}^pdt\Big] \,,$$
    where $r$ satisfies
    $$\frac{1}{r} = \Big(\frac{\eta \wedge p^{-1}}{N} + \frac{1}{q}\Big) \wedge \frac{1}{2} \,.$$
\end{lem}

\begin{proof}
    In \Cref{lq-to-lr} we take $u_t = x_t$, $y_t = F(x_t)$, $Z_t = G(x_t)$ (recall that $x_t$ solves \eqref{SPDE} by \Cref{rv-is-markov-process}). By the last bullet point of \Cref{SPDE-F} and $\eta \leq 1$ we obtain
    $$\|y^+_t\|_{\eta-1,r} = \|F^+(x_t)\|_{\eta-1,r} \lesssim \|F^+(x_t)\|_{L^r} \lesssim 1+\|x_t\|_{L^r} \,.$$
    By \Cref{SPDE-G-lem} and \Cref{SPDE-G-sass},
    $$\|Z\|_{\eta-1/2,r} = \|G(x_t)\|_{\eta-1/2,r} = \|G(x_t)\|_{\beta,r} \lesssim \|x_t\|_{L^r} \,.$$
    Then the claim follows from \Cref{lq-to-lr}.
\end{proof}

\begin{lem}\label{lem:case-1-l-infty-to-p-in-X}
    Under \Cref{SPDE-lyap}\ref{case-1}, for all $p > 1$ we have $\|x\|_{L^\infty}^p \in \X$ (see \Cref{x-y-def}). In fact, we may take $f_n(r,v,z) = \|x\|_{L^\infty}^p$ for some $n \geq 1$ in \Cref{fn-in-X}.
\end{lem}

\begin{proof}

    Let $S$ be the set of pairs $(p,q) \in [1,\infty) \times [2,\infty]$ such that $f_n(r,v,z) = \|x\|_{L^q}^p$ satisfies the assumptions of \Cref{fn-in-X} for some $n \geq 1$, which implies  $\|x\|_{L^q}^p \in \X$. 
    Suppose $(p,r) \in S$ with $r < \infty$, $p > 1$ and that $\infty \geq q > 2$ is such that $1/r \leq (\eta \wedge p^{-1})/(4N) + 1/q$, where $\eta>0$ is as in \Cref{lem:bootstrap-inductive-step}. Then by \Cref{lem:bootstrap-inductive-step},
     \begin{equation}\label{eq:bound-lq-w-lr}
         \E\Big[\int_0^1 \|x_t\|_{L^q}^pdt\Big] \lesssim \|x_0\|_{L^r}^p + \E\Big[\int_0^1 1+\|x_t\|_{L^r}^pdt\Big] 
     \end{equation}
and the statement also holds with $p$ replaced by $2p$.
We proceed by showing:
    \begin{enumerate}
        \item If $(p,2) \in S$ and $p > 1$ then $(p,\infty) \in S$.
        \item If $(p,\infty) \in S$ then $(3p/2,2) \in S$.
        \item $(p^*,2) \in S$ for some $p^* > 1$.
        
    \end{enumerate}
    Then, by induction, $(p,\infty) \in S$ for all $p > 1$, which finishes the proof. 

    To prove (1), suppose $(p,r) \in S$ with $2\leq r < \infty$, $p > 1$, that is, $f_n = \|x\|_{L^r}^p$ satisfies the assumptions of \Cref{fn-in-X} for some $n \geq 1$. Let $q$ be as in \eqref{eq:bound-lq-w-lr}, and set $f_{n+1} = \|x\|_{L^q}^p$. Then, by \eqref{eq:bound-lq-w-lr}, we obtain that $f_{n+1}$ satisfies the assumptions of \Cref{fn-in-X} with $h_{n+1} = 0$, and consequently $(p, q) \in S$. 
Since it is possible to decrease $1/q$ by $(\eta \wedge p^{-1})/(4N)$  until $q = \infty$, we obtain that (for $p > 1$) $(p,2) \in S$ implies that $(p,\infty) \in S$, and (1) follows.

    To prove (2), suppose $(p,\infty) \in S$. Then by H\" older's and Young's inequalities
     \begin{equation}\label{eq:hiltin}
     \|x\|_{L^2}^{3p/2} \leq \|x\|_{L^1}^{3p/4}\|x\|_{L^\infty}^{3p/4} \lesssim \|x\|_{L^1}^{3p} + \|x\|_{L^\infty}^{p}\,.
     \end{equation}
     By \Cref{lem:lyapunov-fxn-construction-spde}(i) we may assume $\|x\|_{L^1}^{6p} \lesssim W'$. Let $(f_i)_{i=1}^n$ be as in \Cref{fn-in-X} with $f_n = \|x\|_{L^\infty}^p$. Replace $f_i$ with $f_i + (W')^{1/2}$ for all $i = 1, \ldots n$. Then the assumptions of \Cref{fn-in-X} still hold and $\|x\|_{L^1}^{3p} + \|x\|_{L^\infty}^p \lesssim f_n$, thus $(3p/2,2) \in S$ by \eqref{eq:hiltin}.

To prove (3), we fix $r = 2$, 
$p \in (1, 2)$ sufficiently close to 1 (specified below), and 
$q > 2$ such that $1/r \leq (\eta \wedge 2^{-1})/(4N) + 1/q$, and in particular \eqref{eq:bound-lq-w-lr} holds. 
After decreasing $p > 1$ and $q > 2$ if necessary, H\" older's and Young's inequalities imply that there is $a > 1$ such that
\begin{equation}\label{eq:inxlpr}
\|x\|_{L^r}^p \lesssim \|x\|_{L^q} + \|x\|_{L^1}^a \quad \text{and} \quad \|x\|_{L^r}^{2p} \lesssim \|x\|_{L^q}^2 + \|x\|_{L^1}^{2a} \,.
\end{equation}
By \Cref{lem:lyapunov-fxn-construction-spde}(i) we may assume that $\|x\|_{L^1}^a \lesssim (W')^{1/2}$. Then by \eqref{eq:bound-lq-w-lr} and \eqref{eq:inxlpr}, \eqref{eq-int-f-h-bounds} is satisfied with $f_1(r,v,z) = C_2\|x\|_{L^q}^p$ and $h_1(r,v,z) = C_1\|x\|_{L^q}$. The assertion (3) holds true by \eqref{fn-in-X} because  
$\|x\|_{L^q}$ vanishes over $\|x\|_{L^q}^p$ and 
$\E[\int_0^T \|x_t\|_{L^q}^{2p}dt] < \infty$, which follows from 
 the proof of \Cref{bigthm} (specifically \textit{Step 6}), where we showed $x_t \in L^p(\Omega, C([0,T];B))$ for all $t,p > 0$. 
\end{proof}

The proof of the following lemma is similar to the argument given in \textit{Step 4} of \Cref{bigthm}.

\begin{lem}\label{v-sobolev}
Let $\eta$ be as in \Cref{lem:bootstrap-inductive-step} and 
$v$ be as in \Cref{def:rv-process}.
    For all large enough $q < \infty$, $\|v\|_{\eta/2,q}^2 \in \X$ and $\|v\|_{\eta/2,q}^4 \in \Y$.
\end{lem}
\begin{proof}
    By \eqref{new-dr-dv},
    $$v_t = S_u(t)v_0 + (S_u * y)_t + (S_u \diamond Z)_t \,,$$
    where
    \begin{align*}
    y &= v[-\ip{v}{\tilde e_u} - \ip{\tilde F_u(r,v,z)}{1_u} + \|\tilde G_u(r,v,z)^*1_u\|_\U^2] \\
    &\qquad + \tilde F_u(r,v,z) - \tilde G_u(r,v,z)\tilde G_u(r,v,z)^*1_u \\
    Z &= \tilde G_u(r,v,z) - \ip{\tilde G_u(r,v,z)^*1_u}{\cdot}v
    \end{align*}
and we omitted the dependence on $\theta$. 
    By \Cref{convolution-bounds}
    \begin{equation}\label{eq:meotn}
    \E\Big[\int_0^T \|v_t\|_{\eta/2,q}^2dt\Big] \lesssim \|v_0\|^2_{\eta-1/2,q} + \E\Big[\int_0^T \|y_t\|^2_{\eta-1,q}dt\Big] + \E\Big[\int_0^T \|Z_t\|_{\eta-1/2,q}^2dt\Big] \,.
    \end{equation}

Then \Cref{SPDE-F} and the mean value theorem show that $|\tilde F_u(r,v,z)| \lesssim |v|(1+|x|^k)$ for some $k$ (related to the polynomial growth of $f$ and its derivatives).
Note that if \Cref{SPDE-lyap}\ref{case-2} holds, then we take $k = 0$.

Then
\eqref{eq:wierd-g}, \eqref{eq:local-u-bound-G}, and Young's inequality imply,  similarly to \eqref{eq:bc-bounds}, that
for every $\epsilon > 0$ there is $c_\epsilon > 0$ such that
    \begin{align*}
        \|y\|_{\eta-1,q} &\lesssim (1 + \|x\|_{L^\infty}^k)\|v\|_{L^q} + (\|v\|_{L^2}^2 + \|v\|_{L^q}) \\
        &\lesssim (1 + \|x\|_{L^\infty}^k)^{c_\epsilon} + (\|v\|_{L^2}^2 \vee \|v\|_{L^q})^{1+\epsilon} \\
        \|Z\|_{\eta-1/2,q}^2
        &\lesssim (\|v\|_{L^2}^2 \vee \|v\|_{L^q})^2 \,.
    \end{align*}
    From \Cref{lem:case-1-l-infty-to-p-in-X} (if  \Cref{SPDE-lyap}\ref{case-1} is assumed) or $k = 0$ (if \Cref{SPDE-lyap}\ref{case-2} holds) it follows that 
it is possible to set $f_n(r,v,z) = (1 + \|x\|_{L^\infty}^k)^{c_\epsilon}$ in \Cref{fn-in-X}
for some $n \geq 1$. 
     As explained after \eqref{eq:bound-on-v-sigma-epsilon}, for any large $q$ we have that $h_{n+1}(r,v,z) \coloneqq (\|v\|_{L^2}^{2+2\epsilon} \vee \|v\|_{L^q}^{1+1\epsilon})^2$ vanishes over $f_{n+1}(r,v,z) \coloneqq \|v\|_{\eta/2,q}^2$ for some small enough $\epsilon > 0$. Then from 
     \eqref{eq:meotn} follows 
    \begin{align*}
        \E\Big[\int_0^T f_{n+1}(r_t,v_t,z_t)dt\Big] &\lesssim f_n(r_0,v_0,z_0) + h_{n+1}(r_0,v_0,z_0) \\
        &+ \E\Big[\int_0^T f_n(r_t,v_t,z_t) + h_{n+1}(r_t,v_t,z_t)dt\Big] \,,
    \end{align*}
    and the same inequality holds when every function is squared. Since the right hand side is locally finite (that is, it is finite when $T$ is replaced by $T \wedge \tau_m$ for an increasing sequence of stopping times $\tau_m \uparrow \infty$), $h_{n+1}$ vanishes over $f_{n+1}$, and $\E[\int_0^Tf_n(r_t,v_t,z_t)dt] < \infty$ (see the end of the proof of \Cref{lem:case-1-l-infty-to-p-in-X}), we conclude that the left hand side is finite. Then we can apply \Cref{fn-in-X}, which shows that $f_{n+1} \in \X$ and $f_{n+1}^2 \in \Y$, as desired.

    For more details on the localization argument, notice that $\E[\int_0^T h_{n+1}(r_t,v_t,z_t)dt]$ is the only term on the right hand side which is not a-priori finite. Since $h_{n+1}$ vanishes over $f_{n+1}$, there is some $M > 0$ such that $h_{n+1} \leq f_{n+1}/2 + M$. We want to use this bound, subtract $1/2$ of the left hand side from both sides, and then multiply by $2$ to get a finite right hand side. We do not know a-priori that the left hand side is finite, so we cannot do this subtraction. However, this manipulation is valid if we replace $T$ with $T \wedge \tau_M$, where $\tau_M$ is the first time that $h_{n+1}(r_t,v_t,z_t) + f_n(r_t,v_t,z_t)$ exceeds $M$. Since $h_{n+1}+f_n$ is a continuous function defined on all of $\mcM$, the state space of $(r_t,v_t,z_t)$ lives, and $(r_t,v_t,z_t)$ is continuous in time, we conclude that $\tau_M \to \infty$ as $M \to \infty$. Thus, we can proceed with our manipulation when $T$ is replaced by $T \wedge \tau_M$ and then take the limit as $M \to \infty$ on both sides via Monotone convergence theorem.
\end{proof}

In \Cref{lem:lyapunov-fxn-construction-spde} we obtained control over $r$ since  $r \leq \|x\|_{L^1}$.
In \Cref{v-sobolev} we obtained control over $v \in C(\overline D;\R^d)_+ \cap \{\|\cdot\|_{L^1} = 1\}$ (indeed, $\B_{\eta/2,q} \subset C(\overline D)$ compactly for large enough $q$ by \Cref{characterization-of-interpolation-spaces} and a Sobolev embedding). It remains to obtain control over $z \in C(\overline D;\R^{m-d})_+$. In the case of \Cref{SPDE-lyap}\ref{case-1} we have the following result.

\begin{lem}\label{lem:z-compactness-case-1}
    Under \Cref{SPDE-lyap}\ref{case-1}, for all $r \in [2,\infty)$ we have $\|x\|_{\eta/2,r}^2 \in \X$, where $\eta$ is as in \Cref{lem:bootstrap-inductive-step}.
\end{lem}

\begin{proof}
    Using \Cref{convolution-bounds} we have
    $$\E\Big[\int_0^T \|x_t\|_{\eta/2,r}^2dt\Big] \lesssim \|x_0\|_{L^r}^2 + \E\Big[\int_0^T \|F(x_t)\|_{L^r}^2 + \|G(x_t)\|_{\beta,r}^2dt\Big] \,.$$
    Using \Cref{SPDE-G-lem} and the polynomial boundedness of $F$ (\Cref{SPDE-F}) we obtain
    $$\E\Big[\int_0^T \|x_t\|_{\eta/2,r}^2dt\Big] \lesssim \|x_0\|_{L^r}^2 + \E\Big[\int_0^T 1 + \|x_t\|_{L^r}^{2k+2}dt\Big] $$
    and a similar estimate holds after squaring all norms.

    The claim follows from \Cref{fn-in-X} with $f_{n+1}(x) = \|x_t\|_{\eta/2,r}^2$ and $f_n(x) = 1 + \|x_t\|_{L^r}^{2k+2}$
    (the assumptions of \Cref{fn-in-X} are justified by
    \Cref{lem:case-1-l-infty-to-p-in-X}).
\end{proof}

The case of \Cref{SPDE-lyap}\ref{case-2} requires a similar argument to \Cref{v-sobolev}.

\begin{lem}\label{lem:z-compactness-case-2}
    Under \Cref{SPDE-lyap}\ref{case-2}, for all large enough $q < \infty$
    $$\Big(\frac{\|x\|_{\eta/2,q}}{\|x\|_{L^1}+1}\Big)^2 \in \X \,,$$
    where $\eta$ is as in \Cref{lem:bootstrap-inductive-step}. 
\end{lem}

\begin{proof}
    Let $w = x/(\|x\|_{L^1}+1)$ and $l = \|x\|_{L^1}$. We define
    \begin{equation}\label{eq:hat-g}
        \hat F(l,w) \coloneqq (1+l)^{-1}F((1+l)w) \quad \text{and} \quad \hat G(l,w) = (1+l)^{-1}G((1+l)w) \,.
    \end{equation}
    Analogously as in \Cref{bigthm}\ref{bigthm:drdv} (set $R = 1$ and replace $r_t$ by $r_t+1$), we obtain that $w$ is a mild solution (see \Cref{def:mild}) of
    \begin{align*}
        dw &= Awdt + w[-\ip{w}{\tilde e} - \ip{\hat F(l,w)}{1} + \|\hat G(l,w)^*1\|_\U^2]dt \\
            &\qquad + \hat F(l,w)dt - \hat G(l,w)\hat G(l,w)^*1dt \\
            &\qquad + \hat G(l,w)dW_t- \ip{\hat G(l,w)^*1}{dW_t}w
    \end{align*}
    Then the same argument as in the proof of \Cref{v-sobolev} (using $|\hat F(l,w)| \lesssim 1+|w|$ instead of $|\tilde F_u(r,v,z)| \lesssim |v|$ in the case of \Cref{SPDE-lyap}\ref{case-2}) shows that $\|w\|_{\eta/2,q}^2 \in \X$.
\end{proof}

\begin{lem}\label{lem:compactness-for-spde}
Suppose \Cref{sass-A}--\ref{SPDE-lyap} hold.
    Then \Cref{as-V}\ref{as-vanish}, \Cref{as-U}\ref{5.4}, and \Cref{as-compact} hold for $(r_t,v_t,z_t,\theta)$ (see \Cref{rv-is-markov-process}), where $V$ is as in \eqref{def-of-V-for-SPDE}. Additionally, $V \in \Dme_2(\inv)$ and $\Ll V$ (see \Cref{l-v-lemma}) extends continuously to some $H: \mcM \to \R$.
   Note that $H$ also denotes the operator which is part of the noise (see \Cref{SPDE-G-sass}), but the distinction should be clear from the context.  
    
\end{lem}

\begin{proof}
    For \Cref{as-V}\ref{as-vanish}, first note that $\Ll V$ as calculated in \Cref{l-v-lemma} extends to a continuous function on $\mcM$, because $h(x) = g(-\ln x)$ with $g$ as in \eqref{eq:dfgc}, and therefore  $h'(r)r$ and $h''(r)r^2$ are continuous when set respectively equal to $-1$ and $1$ when $r = 0$. By $|h''(r)| \lesssim r^{-2}$, $H \in \Ll(\U,L^2)$, and \eqref{eq:sigma-u-bound} we have
    \begin{equation}\label{eq:gtuest}
    \frac{1}{2}h''(r)r^2\|\tilde G_u(r,v,z)^*1_u\|_\U^2 \lesssim \|v\|_{L^2}^2    
    \end{equation}
    (see for example the proof of \eqref{eq:local-u-bound-G}). Similarly by $|h'(r)| \lesssim r^{-1}$ we have
    $$h'(r)r\Big[\ip{v}{\tilde e_u} + \ip{\tilde F_u(r,v,z)}{1_u}\Big] \lesssim 1 + \|x\|_{L^\infty}^k, 
    \qquad \textrm{where} \quad x = (rv, z)\,.$$
    If \Cref{SPDE-lyap}\ref{case-2} holds (see the proof of \Cref{v-sobolev}), then $k = 0$, and therefore $\Ll V$ vanishes over $\|v\|_{\eta/2,2}^2$ by \Cref{characterization-of-interpolation-spaces},  and Gagliardo-Nirenberg inequality (since $\|v\|_{L^1} = 1$ we have $\|v\|_{L^2} \lesssim \|v\|_{\eta/2,2}^{1-\epsilon}$ for some $\epsilon > 0$). If \Cref{SPDE-lyap}\ref{case-1} holds, $\Ll V$ vanishes over $\|v\|_{\eta/2,2}^2 + \|x\|_{L^\infty}^{k+1}$. In either case, \Cref{as-V}\ref{as-vanish} with $\mu = \delta_0$ and $f(r, v, z) = \|v\|^2_{\eta/2, 2} $ or $f(r, v, z) = \|v\|^2_{\eta/2, 2} + \|v\|_{\eta/2,2}^2 + \|x\|_{L^\infty}^{k+1}$ is satisfied by \Cref{v-sobolev} and \Cref{lem:case-1-l-infty-to-p-in-X}.

    Next we discuss \Cref{as-U}\ref{5.4}. For $\Gamma V$ calculated in 
    \Cref{l-v-lemma}, we have as in \eqref{eq:gtuest} that 
    $$\Gamma V(r,v,z) \lesssim \|\tilde G_u(r,v,z)^*1_u\|_\U^2 \lesssim \|v\|_{L^2}^2 \,.$$
    Since $\|v\|_{L^2}^2$ again vanishes over $g(r, v, z) := \|v\|_{\eta/2,q}^4$ for all $q \geq 2$, \Cref{as-U}\ref{5.4} follows from \Cref{v-sobolev}. Similarly as in \Cref{generator-for-SPDE}, this also shows that $V \in \Dme_2(\inv)$.

To establish \Cref{as-compact}, we verify the sufficient condition 
from \Cref{compact-sublevel}, which requires the existence of 
$f \in \X$ with compact sublevel sets. 
Define $f : \mcM \to \R$ depending on whether \Cref{SPDE-lyap}\ref{case-1} or \Cref{SPDE-lyap}\ref{case-2} holds:
\begin{align*}
    f(r, v, z) = \begin{cases}
        \|v\|_{\eta/2,q}^2 + \|x\|_{_{\eta/2,q}}^2 & \text{in case }(i)\,, \\
        \|v\|_{\eta/2,q}^2 + \|x\|_{L^1}^{p/2} +\Big(\frac{\|x\|_{_{\eta/2,q}}}{\|x\|_{L^1} + 1}\Big)^2 & \text{in case (ii)} \,,
    \end{cases}
\end{align*}
where $x = (rv, z)$, $q$ is sufficiently large as specified below, and $p > 1$ is as in \Cref{lem:lyapunov-fxn-construction-spde}(ii). To show that $f \in \X$, first note that by \Cref{muW'-K} we have $W' \in \X$. Then $f \in \X$ follows from \Cref{lem:lyapunov-fxn-construction-spde}, \Cref{v-sobolev}, \Cref{lem:z-compactness-case-1}, and \Cref{lem:z-compactness-case-2}.

To obtain that $f$ has compact sublevel sets, first note that by \Cref{characterization-of-interpolation-spaces} and a Sobolev embedding we may choose $q$ large enough that $\B_{\eta/2,q} \hookrightarrow C(\overline D)$ compactly. Thus, $\|v\|_{\eta/2,q}^2$ has compact sublevel sets as a function of $\{v \in C(\overline D;\R^d)_+ \mid \|v\|_{L^1} = 1\}$. Also, $r \leq \|x\|_{L^1} \lesssim \|x\|_{\eta/2,q}$, so $\|x\|_{\eta/2,q} \leq M$ implies that $(r,z)$ lies in a compact subset of $[0,\infty) \times C(\overline D;\R^{m-d})_+$. Thus, in the first case, $f$ has compact sublevel sets. For the second case, note that if $\|x\|_{L^1} \leq M$ and $\|x\|_{\eta/2,q}/(\|x\|_{L^1} + 1) \leq M$ then $\|x\|_{\eta/2,q} \leq M(M+1)$. Thus, in the second case, $f$ has compact sublevel sets.

\end{proof}

\subsection{Main Results}\label{SPDE-main-results}
To summarize, in \Cref{sec:spde-well-posed} we defined a Feller Markov process $(r_t,v_t,z_t,\theta)$ on the state space
$$\mcM = [0,\infty) \times \{v \in C(\overline D;\R^d)_+ \mid \|v\|_{L^1} = 1\} \times C(\overline D;\R^{m-d})_+ \times \Theta$$
which has $\mcM_0 \coloneqq \mcM \cap \{r = 0\}$ and its complement $\inv \coloneqq \mcM_0^c$ as invariant sets. The connection between this Markov process and our original equation \eqref{SPDE} is that $x_t = (u_t,z_t) = (r_tv_t,z_t)$, so that $r$ is the distance from $u$ to $0$ and $v$ is the angle. In \Cref{sec:SPDE-lyap} and \Cref{sec:SPDE-semigrp-method} we used variational and semigroup methods to obtain control on
$(r_t,v_t,z_t,\theta)$
in various norms over this projective process, and we are ready to apply \Cref{main} and \Cref{robust} to investigate the stability of $\mcM_0$.

For those who are not familiar with the preceding sections, one may understand the main definitions and statements of the theorems in this section having only read the introduction of \Cref{section-SPDE} and \Cref{sass-A}-\ref{SPDE-G-sass}. For example, we provide two definitions for $\Lambda$. One is used in applications of our theory to examples in \Cref{examples}. A more concise definition of $\Lambda$, used in the proofs in this subsection, is also given but  requires deeper understanding of preceding sections.
\begin{deff}\label{def:equiv-lambda-def}
 Define the system
 
\begin{equation}\label{eq:general-linearized-SPDE}
    \begin{aligned}
        du(\cdot) &= [A_uu(\cdot) + \hat f(\cdot,z(\cdot),\theta)u(\cdot)]dt + \hat \sigma(\cdot,z(\cdot),\theta)u(\cdot)H_udW_t \\
        dz(\cdot) &= [A_zz(\cdot) + f_z(\cdot,u(\cdot),z(\cdot),\theta))]dt + \sigma_z(\cdot,u(\cdot),z(\cdot),\theta)H_zdW_t \,,
    \end{aligned}
\end{equation}
obtained from \eqref{SPDE} by linearizing in $u$ in the first $d$ components. By this we mean that $\hat f,\hat \sigma: D \times \R^{m-d} \times \Theta \to \R^d$ are calculated from $f_u,\sigma_u$ by ignoring any term which is nonlinear in $u$. Then we define
\begin{equation}\label{eq:lambda-alternative}
\begin{aligned}
    \Lambda &\coloneqq \sup_\nu \E_\nu\Big[-\int_D \sum_{i=1}^d A_iv_i(\cdot) + \hat f_i(\cdot,z(\cdot),\theta)v_i(\cdot) d\mu(\cdot)\\
    &+ \frac{1}{2}\sum_{n \geq 1} a_n^2\Big(\int_D \sum_{i=1}^d\hat \sigma_i(\cdot,z(\cdot),
    \theta)v_i(\cdot)(e_n)_i(\cdot)d\mu(\cdot)\Big)^2\Big] \,,
\end{aligned}
\end{equation}
where $v = u/\|u\|_{L^1}$, $(v,z,\theta) \sim \nu$, and $\nu$ ranges over all invariant measures (stationary distributions) on $\{(v,z) \in \cdom \mid \|v\|_{L^1} = 1\} \times \Theta$ for $(u_t/\|u_t\|_{L^1},z_t,\theta)$, where $(u_t,z_t,\theta)$ solves \eqref{eq:general-linearized-SPDE}.
\end{deff}
\begin{rem}\label{rem:explain-lambda-equiv}
More precisely, in \Cref{def:equiv-lambda-def}, we have $\hat f(\cdot,z,\theta) \coloneqq \tilde f_u(\cdot,0,z,\theta)$, where $0$ in the argument of $\tilde{f}$ corresponds to the scaling variable $R$ 
(see \Cref{r-functions}). Also, $v_i$ is not guaranteed to be in the domain of $A_i$ (for example, if we are in one dimension and $dW_t/dt$ is space-time white noise, then, at best, $v_i \in \Dm((-A)^{1/4-\epsilon})$), so technically we understand $A_iv_i$ as $\tilde e_iv_i$ (see \Cref{def:tilde-e}), or equivalently in the sense of distributions. 
\end{rem}

Alternatively, recall \Cref{meas} and \Cref{inv-meas} and then (omitting the dependence on $\theta \in \Theta$) let 
\begin{equation}\label{SPDE-Lambda}
    \Lambda \coloneqq \sup_{\mu \in P_{inv}(\mcM_0)} \mu (-\ip{v}{\tilde e_u} - \ip{\tilde F_u(0,v,z)}{1_u} + \frac{1}{2}\|\tilde G_u(0,v,z)^*1_u\|_U^2 ) \,,
\end{equation}
where $\tilde F,\tilde G$ are defined in \Cref{r-functions}, $\tilde e$ is as in \Cref{def:tilde-e}, and the adjoint of $G_u(r,v,z)$ is taken in $\Ll(\U,L^2(D;\R^d))$. We claim that \eqref{SPDE-Lambda} is equivalent to \Cref{def:equiv-lambda-def}. Indeed, \eqref{eq:general-linearized-SPDE} is obtained by setting $r_0 = 0$ in \Cref{def:rv-process}, so we have $u/\|u\|_{L^1} = v$, where $u$ solves \eqref{eq:general-linearized-SPDE}. Thus, $P_{inv}(\mcM_0)$ is essentially the same as the set of invariant measures for $(u_t/\|u_t\|_{L^1},z_t)$, except that $\nu \in P_{inv}(\mcM_0)$ is technically a measure on $\{0\} \times \{(v,z) \in \cdom \mid \|v\|_{L^1} = 1\} \times \Theta$ instead of $\{(v,z) \in \cdom \mid \|v\|_{L^1} = 1\} \times \Theta$. Thus, taking $\sup_\nu \E_\nu[]$ in \Cref{def:equiv-lambda-def} is the same as taking $\sup_{\mu \in P_{inv}(\mcM_0)} \mu()$ in \eqref{SPDE-Lambda}. Finally, notice that the expression inside the $\E_\nu[]$ in \eqref{eq:lambda-alternative} is just the expression inside $\mu()$ in \eqref{SPDE-Lambda} expanded out according to \Cref{def:tilde-e}, \Cref{SPDE-F}, and \Cref{SPDE-G-sass} (see \Cref{rem:explain-lambda-equiv}).

\begin{rem}
    Under our assumptions, it can be shown that $-\Lambda$ is the smallest possible value of $\liminf_{t \to \infty} \frac{1}{t}\ln \|\tilde u_t\|_{L^1}$ over all initial conditions with $R = 0, u_0 \neq 0$, where $x_t = (\tilde u_t,z_t)$ solves \eqref{eq:mild-solution}. Since \eqref{eq:mild-solution} with $R = 0$ is the linearized version of \eqref{SPDE}, the following result can be interpreted as saying that $\{u = 0\}$ is a repeller for \eqref{SPDE} if it is a repeller for the linearized system.
\end{rem}
\begin{thm}\label{main-SPDE-thm}
Suppose \Cref{sass-A}--\ref{SPDE-lyap} hold. Let $(x_t, \theta) = (u_t,z_t,\theta)$ be the solution to \eqref{SPDE} with initial condition $x_0 = (u_0,z_0)$. If $\Lambda < 0$ then:
\begin{enumerate}[label=(\roman*)]
     \item For all $u_0 \in C(\overline D;\R^d)_+ \setminus \{0\}$, almost surely, all limit points $\mu$ of the random measures $\mu_t = \frac{1}{t}\int_0^t \delta_{u_s}ds$ on $C(\overline D;\R^d)_+$ as $t \to \infty$ satisfy $\mu(\{0\}) = 0$.
        \item $u_t$ is stochastically persistent, meaning for all $\epsilon > 0$ there exists a compact set $\K_\epsilon \subset C(\overline D;\R^d)_+ \setminus \{0\}$ such that for all $u_0 \in C(\overline D;\R^d)_+ \setminus \{0\}$
        $$\Prb(\liminf_{t \to \infty} \mu_t(\K_\epsilon) \geq 1 - \epsilon) = 1 \,.$$
        In particular, for all $\epsilon > 0$, there are $a,A > 0$ such that
        $$\liminf_{T \to \infty} \frac{1}{T} m(\{t \leq T \mid a \leq \|u_t\|_{L^1} \leq A\}) \geq 1 -\epsilon \quad a.s.,
        $$
        where $m$ denotes the Lebesgue measure on $[0,T]$.
        
        \item There is a measure $\mu_+$ on $\cdom \times \Theta$ such that $\mu_+(\{x_u \neq 0\}) = 1$ and $\mu_+$ is invariant for $(x_t,\theta)$ in the sense of \Cref{inv-meas}.
\end{enumerate}
\end{thm}

\begin{rem}
    \Cref{main-SPDE-thm} is only applied below for time independent $\theta$, that is, $\Theta = \{\theta\}$,  however, it is stated more generally. As noted in \Cref{we-can-also-handle-switching}, there might be contexts in which $\theta$ is time-dependent and our results (perhaps with extra assumptions for non-compact $\Theta$) still apply. Of course after modifications of It\^{o} formula when $\theta_t$ is a jump or some other process, there are possibly extra terms in \eqref{SPDE-Lambda}), but the core of the theorem is the same. We do not pursue the most general setting for brevity and since we did not encounter regime-switching systems in our examples. We encourage readers to generalize oour results if a need arises. 
\end{rem}

\begin{rem}\label{rem:strong-maximum-principle}
    In \Cref{main-SPDE-thm}(ii),  we can improve $a \leq \|u_t\|_{L^1}\leq A$ to $a \leq u_t|_K \leq A$ for some fixed compact $K \subset D$ if $u_0 \neq 0$ implies $u_t > 0$ for all $t > 0$ (strong maximum principle). This is a consequence of \Cref{K-contained-in-strictly-positive} with $\U = \{u > 0\}$. Proving such a strong maximum principle in our setting is an open problem. It has been proved in special cases, for example: in \cite[Theorem 1.2]{SPDE-harnack} the noise is Hilbert-Schmidt (\eqref{eq:an-en-assumption} holds with $p = 2$) and the nonlinearity $F$ has linear growth; in \cite[Theorem 2.1]{stochastic-SIR-analytical}  the noise is of a simple form (finite-dimensional); and in \cite[Theorem 2.3]{space-time} when $N=1$. 
    Since we did not find a short way to incorporate 
    these references to our results, we leave it for future endeavors. 
\end{rem}

\begin{proof}
    In order to apply \Cref{main}, we need to verify \Cref{as1}-\Cref{as-compact} for $X_t \coloneqq (r_t,v_t,z_t,\theta)$ from \Cref{def:rv-process}. \Cref{as1} and \Cref{as2} were already verified in \Cref{rv-is-markov-process}.
    
    By \Cref{lem:lyapunov-fxn-construction-spde}, \Cref{as-W} and \Cref{as-U}\ref{5.1}-\ref{5.3} hold. Recall the function $V: \inv \to [0,\infty)$ defined in \eqref{def-of-V-for-SPDE} which agrees with $-\ln{r}$ when $r \in (0,e^{-1})$. 
    
    By \Cref{lem:compactness-for-spde}, $V \in \Dme_2(\inv)$ and $\Ll V$ extends continuously to some $H: \mcM \to \R$ which for $r > 0$ satisfies
    $$
    H(r,v,z) = h'(r)r\Big[\ip{v}{\tilde e_u} + \ip{\tilde F_u(r,v,z)}{1_u}\Big] + \frac{1}{2}h''(r)r^2\|\tilde G_u(r,v,z)^*1_u\|_U^2 \,.
    $$
     As noted in the proof of \Cref{lem:compactness-for-spde}, $h'(r)r \to -1$ and $h''(r)r^2 \to 1$ as $r \downarrow 0$. Thus, the quantity defined in \eqref{SPDE-Lambda} is equal to $\sup_{\mu \in P_{inv}(\mcM_0)} \mu H$. We are assuming $\Lambda < 0$, so \Cref{as-V}\ref{as-alpha} is satisfied. The remainder of \Cref{as-V}-\ref{as-compact} follow from \Cref{lem:compactness-for-spde}.

    With \Cref{as1}--\ref{as-compact} verified, our conclusions are almost direct consequences of \Cref{main} as shown next.
    
    \begin{itemize}
    
    \item Denote $\Upsilon(r,v,z,\theta) = (rv,z,\theta)$, which is a homeomorphism from $\inv$ to $\{x \in \cdom \mid x_u \neq 0\} \times \Theta$. Since the solution $x_t$ to \eqref{SPDE} is equal to $(r_tv_t,z_t)$ (\Cref{rv-is-markov-process}), pushing forward by $\Upsilon$ gives a bijection between $P_{inv}(\inv)$ and the set of invariant measures $\mu$ for $x_t$ with $\mu(\{x_u \neq 0\}) = 1$. Thus, (iii) follows from \Cref{main} (iii).
    \item To prove (ii) notice that $\Upsilon$ gives a correspondence between compact subsets of $\inv$ and compact subsets of $\{x \in \cdom \mid x_u \neq 0\} \times \Theta$, which relates 
$\K_\epsilon$ in \Cref{main}(ii) to  $\K_\epsilon$ in (ii)
    (apply $\Upsilon$ and project onto the $u$ coordinate to obtain a compact subset of $C(\overline D;\R^d)_+ \setminus \{0\}$). Finally, (ii) follows since $\|u\|_{L^1}$ is a continuous strictly positive function on $C(\overline D;\R^d)_+ \setminus \{0\}$, and thus attains a strictly positive minimum on $\K_\epsilon$. 
    \item Let $\nu_t$ be the empirical occupation measures for $(r_t,v_t,z_t,\theta)$ (see \Cref{occ-meas}) and $\mu_t = \frac{1}{t}\int_0^t \delta_{u_s}ds$. Then $\mu_t$ is $\nu_t$ pushed forward by $\Upsilon_u$ ($\Upsilon$ followed by projection onto the $u$ coordinates). Suppose $t_n \to \infty$ and $\mu_{t_n} \to \mu$ weakly as measures on $C(\overline D;\R^d)_+$. Although this does not imply that $\nu_{t_n}$ converges, by \Cref{tightness} there is a convergent subsequence and by \Cref{main}(i) the limit point $\nu$ lies in $P_{inv}(\inv)$. Since $\inv = \Upsilon_u^{-1}(C(\overline D;\R^d)_+ \setminus \{0\})$, we conclude that $\mu(C(\overline D;\R^d)_+ \setminus \{0\}) = \nu(\Upsilon_u^{-1}(C(\overline D;\R^d)_+ \setminus \{0\})) = \nu(\inv) = 1$, which proves (i).
    \end{itemize}
\end{proof}

In the case of Lotka-Volterra or other models for species competition, we may want to know when \textit{all} of the species coexist, as opposed to just saying that at least one will survive. In this case, we use the following result.

\begin{thm}\label{main-SPDE-thm-2}
    Suppose \Cref{sass-A}-\ref{SPDE-lyap} hold when $x_t$ is written as $(u_t,z_t)$ where $u_t = (x_t)_i$ for all $i = 1,\dots,d$. Consider the processes 
    $(r_i(t),v_i(t),x_{-i}(t),\theta) = (\|(x_t)_i\|_{L^1}, (x_t)_i/\|(x_t)_i\|_{L^1}, x_{-i}(t),\theta)$ (where $x_{-i}$ is the vector with $i$th coordinate omitted) on
    
    $$
    \mcM^{(i)} = [0,\infty) \times \{x \in \cdom \mid \|x_1\|_{L^1} = 1\} \times \Theta
    $$
    with invariant set $\mcM_0^{(i)} = \mcM_0^{(i)} \cap \{r_i = 0\}$. 
    Define  $\phi_i: \mcM^{(i)} \to \cdom \times \Theta$ with $\phi_i(r_i,v_i,x_{-i},\theta) = (x,\theta)$, where $x$ is obtained from  $x_{-i}$ by inserting $r_iv_i$ to $i$th coordinate. Furthermore, we set
    $$\mcM^* \coloneqq \cdom \times \Theta \,, \quad \mcM_0^* \coloneqq \{x \in \cdom \mid x_i = 0 \text{ for some } i=1,\dots,d\} \times \Theta \,.$$
    For each choice of $i=1,\dots,d$ and $\mu \in P_{inv}(\mcM_0^*)$ with $\mu(\{x_i = 0\}) = 1$ define
    \begin{equation}\label{eq:inv-rates-general}
        r_i(\mu) \coloneqq \inf_{\nu \in P_{inv}(\mcM_0^{(i)}), \phi_i^*\nu = \mu} \mu (\ip{v}{\tilde e_i} + \ip{\tilde F_i(r,v,z)}{1_i} - \frac{1}{2}\|\tilde G_i(r,v,z)^*1_i\|_U^2)
    \end{equation}
    (here $\phi_i^*\nu$ is the pushforward of the measure $\nu$, i.e. $\phi_i^*\nu(A) = \nu(\phi_i^{-1}(A))$). For arbitrary $\mu \in P_{inv}(\mcM_0^*)$, we write $\mu = p\mu_1 + (1-p)\mu_2$ for some $p \in [0,1]$ and $\mu_1,\mu_2 \in P_{inv}(\mcM_0^*)$ satisfying $\mu_1(\{x_i = 0\}) = 1$ and $\mu_2(\{x_i = 0\}) = 0$, and then define $r_i(\mu) \coloneqq pr_i(\mu_1)$ (this is possible because $\{x_i = 0\}$ is an invariant set).

    If $\max_{i = 1,\dots,m} r_i(\mu) > 0$ for all $\mu \in P_{inv}(\mcM_0^*)$ then $x_t$ is stochastically persistent in the sense of \Cref{main} (with $\inv = \mcM^* \setminus \mcM_0^*$ and $X_t = (x_t,\theta)$).
\end{thm}

\begin{proof}
    The proof follows the same steps as the proof of \Cref{main-SPDE-thm}, with a few key differences. Instead of using
    $$[0,\infty) \times \{v \in C(\overline D;\R^d)_+ \mid \|v\|_{L^1} = 1\} \times C(\overline D;\R^{m-d})_+ \times \Theta$$
    as our state space, we need to use
     $$\mcM \coloneqq [0,\infty)^d \times \{v \in C(\overline D;\R)_+^d \mid \|v_i\|_{L^1} = 1\} \times C(\overline D;\R^{m-d})_+ \times \Theta$$
     and $\mcM_0 \coloneqq \mcM \cap \{r = 0\}$. (We use the notation $(r,v,z,\theta) \in \mcM$ where $r \in [0,\infty)^d$, etc.)
     Our Markov process $(r_t,v_t,z_t,\theta)$ is given by $r_t = (r_1(t),\cdots,r_d(t))$ and $v_t = (v_1(t),\dots,v_d(t))$. Thus, our $\Upsilon: \inv \to \inv^*$ is given by $\Upsilon(r,v,z,\theta) = (r\cdot v,z,\theta)$.
     Finally, for $V$ we use
     $$V(r,v,z,\theta) \coloneqq \sum_{i=1}^d \rho_ih(r_i) \,,$$
     where $h$ is as in \eqref{def-of-V-for-SPDE} and choose $\rho_i > 0$  so that $\sup_{\mu \in P_{inv}(\mcM_0^*)} (-\sum_{i=1}^d \rho_ir_i(\mu)) < 0$ (this is possible by Hahn-Banach separation theorem, see the proof of \Cref{functional-thm}).

     Then the proof follows from \Cref{rv-is-markov-process}, \Cref{lem:lyapunov-fxn-construction-spde}, \Cref{lem:compactness-for-spde}, and \Cref{main} in exactly the same way as \Cref{main-SPDE-thm}.
\end{proof}

The following theorem shows that the stability of $\mcM_0$ depends continuously on the parameters of the system (meaning $\theta \in \Theta$).

\begin{thm}
\label{robust-SPDE-thm}
    
    Suppose that the assumptions of \Cref{main-SPDE-thm} or \Cref{main-SPDE-thm-2} are satisfied when $\Theta$ is replaced by $\{\theta_0\}$ for some $\theta_0 \in \Theta$. Then there is some open $\U \subset \Theta$ such that $\theta_0 \in \U$ and the conclusions of \Cref{main-SPDE-thm} or \Cref{main-SPDE-thm-2} respectively hold when $\Theta$ is replaced by $\U$. In other words, as long as the parameter $\theta$ is close enough to $\theta_0$, the process is stochastically persistent.
\end{thm}

\begin{proof}
    The proof is an immediate consequence of \Cref{robust}.
    
\end{proof}

\section{Examples}\label{examples}
In this section we present four examples of stochastic reaction-diffusion type PDEs to which we apply the theorems from \Cref{section-SPDE}. We give criteria which ensure that a species survives (\Cref{example-logistic}), a fixed point is repelling, indicating turbulence (\Cref{example-torus}), a disease is endemic (\Cref{example-SIR}), and an ecosystem of species in competition can coexist (\Cref{sec:lotka-volterra}). We work with a variety of noise types from space-time white noise (\Cref{example-torus}) to standard Brownian motion (\Cref{example-SIR}) and even some interesting intermediate examples  (\Cref{sec:lotka-volterra}). Rather than formulating results in all possible settings such as different noise, boundary conditions, number of equations etc., 
we provide examples that illustrate applications of our results. 

In what follows, well-posedness of the problem and nonnegativity of the solution is a consequence of \Cref{well-posedness-of-SPDE}, so we do not discuss it further and instead focus on the long-term dynamics. Throughout the following, recall that $\cdom$ is the set of all nonnegative continuous functions on $\overline D$, the closure of $D$,
with range $[0, \infty)^m$ for a specified integer $m \geq 1$.

\subsection{Logistic Growth Model on a Domain}\label{example-logistic}
We consider a population evolving in some domain $D \subset \R^N$ which is open, connected, and bounded with smooth boundary. At each point $x \in D$ there is a carrying capacity $K(x)$, an intrinsic growth rate $r(x)$, and a harvesting rate $E(x)$. 
Specifically, 
fix $d > 0$, $C^1$ functions $K,E,r: \overline D \to [0,\infty)$  with $K > 0$, 
smooth functions 
$g_i : \overline D \to \R$, and a $C^1$ function $\sigma: [0,\infty) \to \R$ with $\sigma(0) = 0$ and with locally Lipschitz, globally bounded first derivatives. Also, assume $w_t$ is a Brownian motion on $\R$ and $\epsilon \in \Theta := [0,1]$ is a constant.

Then the logistic growth model for the population with stochasticity is
\begin{equation}\label{reac-diff-domain}
    du = \Big[d\Delta \Big(\frac{u}{K(x)}\Big) + r(x)u\Big(1 - \frac{u}{K(x)}\Big) - E(x)u\Big]dt + \epsilon \sigma(u)K(x)dw_t \,,
\end{equation}
where we impose 
Neumann boundary conditions $\frac{d(u/K)}{dn} \equiv 0$ on $\partial D$ with $n$
being the outer unit normal vector  .

 Our solution $u_t$ to \eqref{reac-diff-domain} with a specified initial condition $u_0 \in \cdom$ is guaranteed to exist as a continuous (with respect to time) $\cdom$-valued process and is unique by \Cref{well-posedness-of-SPDE} (indeed, we will verify \Cref{sass-A}-\ref{SPDE-G-sass} below).

Using the results in \Cref{SPDE-main-results}, we prove theorems related to the persistence of the population. Heuristically, we  show that the species survives in the long term if analogous statements hold for the linearized system:
\begin{equation}\label{reac-diff-domain-linearized}
     du = \Big[d\Delta \Big(\frac{u}{K(x)}\Big) + (r(x) - E(x))u\Big]dt + \epsilon\sigma'(0) K(x)udw^{(i)}_t \,,
\end{equation}
and that this reduces to a (deterministic) eigenvalue problem if we are allowed to choose $\epsilon$ arbitrarily small.
We remark that our results also apply to other growth models (see \cite{carrying-capacity} for examples), but we choose logistic for concreteness and ease of presentation.

\begin{thm}\label{reac-diff-domain-thm}
    Let $\lambda$ be the largest eigenvalue of the problem 
    $$
    d\Delta\Big(\frac{u}{K(x)}\Big) + (r(x)-E(x))u = \lambda u
    \qquad \textrm{on } D
    $$ 
    with the same (Neumann) boundary conditions imposed on \eqref{reac-diff-domain}. If $\lambda > 0$, then there is some $\overline \epsilon > 0$ such that for all $\epsilon < \overline \epsilon$, \eqref{reac-diff-domain} is stochastically persistent in the sense that: 
    \begin{itemize}
        \item For every $\delta > 0$ there are $0 < b < B < \infty$ such that for any nonnegative initial condition $u_0 \neq 0$, $$
        \liminf_{T \to \infty} \frac{1}{T}m\Big(\Big\{t \leq T \mid b \leq \int_D u_t \leq B\Big\}\Big) \geq 1 - \delta
        $$ 
        almost surely, where $m$ is the Lebesgue measure.
        \item There is an invariant measure $\mu_+$ for \eqref{reac-diff-domain} with $\mu_+(\{0\}) = 0$.
    \end{itemize}
\end{thm}

\begin{rem}
  Note that $\int_D u_t$ is the total population of the species at time $t$.
\end{rem}

\begin{rem}\label{rem:rayleigh}
    \Cref{reac-diff-domain-thm} should be compared to \cite[Corollary 1]{carrying-capacity}, where the authors proved a deterministic version of the theorem ($\epsilon = 0$) under the assumption that there exists some $c > 0$ such that $\gamma \coloneqq \inf_{x \in D} (1-c)r(x) - E(x) > 0$. Thus, $E < r$, and consequently 
    $$
    \lambda = \sup_{u \neq 0} \frac{\int_D -d\|\nabla(uK^{-1})\|^2 + (r-E)u^2K^{-1}dm}{\int_D u^2K^{-1}dm} \geq \frac{\int_D (r-E)Kdm}{\int_D Kdm} > 0
    $$ 
    (we used $u = K$ in the first inequality). Hence, the condition for persistence stated in \Cref{reac-diff-domain-stoch-thm} is sharper than the existing one, even in the deterministic case.
\end{rem}

\begin{rem}\label{rem:rayleigh-quotient}
    In \Cref{rem:rayleigh}, $\lambda$ is expressed as a Rayleigh quotient. The proof of \Cref{reac-diff-domain-thm} indicates that, for linear deterministic PDE, $\Lambda$ as defined in \eqref{SPDE-Lambda} is also a Rayleigh quotient, except it uses the $L^1$ norm instead of the $L^2$ norm (the quotient is the same either way).
\end{rem}

\begin{proof}
First we verify that \Cref{sass-A}-\ref{SPDE-lyap} hold for \eqref{reac-diff-domain}. \Cref{sass-A} (i) and (iii) follow from \Cref{rem:subtract-from-c}.  To satisfy \Cref{sass-A} (ii) we work in $L^2(D,\Sigma,\mu)$, where $\Sigma$ is the Borel sigma-algebra of $D$ and the measure $\mu$ has density $K^{-1}$ with respect to Lebesgue measure $m$. We define
\begin{equation}\label{eq:def-a-for-example-1}
    Au \coloneqq d\Delta\Big(\frac{u}{K(x)}\Big) - cu
\end{equation}
for some suitably large $c > 0$ and use integration by parts (twice):
\begin{align*}
    \ip{Au}{v} &= \int_D d\Delta\Big(\frac{u}{K}\Big)v - cuvd\mu = \int_D d\Delta\Big(\frac{u}{K}\Big)\frac{v}{K} -cuvK^{-1}dm \\
    &= \int_D d\Delta\Big(\frac{v}{K}\Big)\frac{u}{K} - cuvK^{-1}dm = \int_D d\Delta\Big(\frac{v}{K}\Big)u -cuvd\mu \\
    &= \ip{u}{Av} 
\end{align*}
and \Cref{sass-A} (ii) follows. 

Note that we have no $z$ component since $m = 1$ (see \Cref{section-SPDE}), and so \Cref{SPDE-F} is verified for $f_2 = 0$, $f_1(\cdot,u) = cu + r(\cdot)u(1-uK(\cdot)^{-1}) - E(\cdot)u$. For \Cref{SPDE-G-sass}, we have that $\sigma(\cdot,u,\theta) = \theta\sigma(u)$ satisfies all assumptions. Since $K(x)$ is an eigenvector for $A$, indeed \eqref{eq:an-en-assumption} is satisfied with $p = 2$ ($a_n = 0$ for $n > 1$).

Next, we set $\beta = 0$ in \Cref{SPDE-G-lem} and verify \Cref{SPDE-lyap} (recall $c$ is a fixed constant involved in the definition of $A$ \eqref{eq:def-a-for-example-1}):
\begin{align*}
    \ip{x}{\tilde e} + \ip{F(x)}{1} &\leq \|x\|_{L^1}\|\tilde e\|_{L^\infty} + \int [c+r-E]K^{-1}x - rK^{-2}x^2 dm \\
    &\leq \|x\|_{L^1}(\|\tilde e\|_{L^\infty} + c + \|r\|_{L^\infty}) - \|x\|_{L^2}^2\inf rK^{-1} \\
    &\leq C - c'\|x\|_{L^1} \,,
\end{align*}
because $\|x\|_{L^1} \lesssim \|x\|_{L^2}$ vanishes over $\|x\|_{L^2}^2$ (cf. \Cref{fxn-vanish}). Similarly we have $\|x\|_{\beta,2}^2 = \|x\|_{L^2}^2$ vanishes over
\begin{align*}
   \ip{x}{F^-(x)}_\beta &= \ip{x}{F^-(x)} \geq \int rx^3K^{-2} - \int [c+r]x^2K^{-1} \\
   &\geq \inf rK^{-1}\|x\|_{L^3}^3 - [c+\sup r]\|x\|_{L^2}^2 \,,
\end{align*}
because $\|x\|_{L^2}^2 \lesssim \|x\|_{L^3}^2$ vanishes over $\|x\|_{L^3}^3$.

Then we want to apply \Cref{robust-SPDE-thm} with $\theta_0 = 0$, which requires verifying the assumptions of \Cref{main-SPDE-thm}. In particular, we need to check that \eqref{SPDE-Lambda} is strictly negative when $\epsilon = 0$ in \eqref{reac-diff-domain}. First we need to find the set of all invariant measures on $\{v \in \cdom \mid \|v\|_{L^1} = 1\}$ (we omit $\Theta$ because we are fixing $\epsilon = 0$) for $v_t = u_t/\|u_t\|_{L^1}$, where $u_t$ solves \eqref{reac-diff-domain-linearized} (observe that \eqref{reac-diff-domain-linearized} is really \eqref{eq:general-linearized-SPDE}, since we are only dealing with one species, and therefore there is no $z$).

By Krein-Rutman theorem and the strong maximum principle, the eigenfunction $e_1$ corresponding to the largest eigenvalue of the operator $u \mapsto d\Delta(u K^{-1}) + (r - E)u$ is strictly positive on $D$, and no other eigenfunction is nonnegative. Since the operator is self-adjoint (this was verified above), it is standard to check that the eigenfunctions form an orthonormal basis for $L^2(D,\Sigma,\mu)$. 
If $u_t$ is the solution to \eqref{reac-diff-domain-linearized} with $\epsilon = 0$ and 
$u_0 \in C(\overline{D})_{+} \setminus \{0\}$ (which implies $\ip{u_0}{e_1} > 0$), then 
 from uniqueness of $e_1$ (up to scaling) follows $\lim_{t \to \infty} u_t/\|u_t\|_{L^2} = e_1/\|e_1\|_{L^2}$ in $L^2$ for all $u_0 \in \cdom \setminus\{0\}$. 
Since $x \mapsto x/\|x\|_{L^1}$ is continuous on $\{x \in L^2 \mid \|x\|_{L^1} > 0\}$, for $x = u_t/\|u_t\|_{L^2}$
we conclude
$$\lim_{t \to \infty} v_t = \lim_{t \to \infty} u_t/\|u_t\|_{L^1} = e_1/\|e_1\|_{L^1}
\qquad \textrm{in } L^2 \,. 
$$
Thus, in \Cref{def:equiv-lambda-def} the only invariant measure $\nu$ is the dirac delta mass at $e_1/\|e_1\|_{L^1}$, and we obtain
$$
\Lambda = -\Big\langle d\Delta\Big(\frac{e_1}{\|e_1\|_{L^1}} K^{-1}\Big) + (r - E)\frac{e_1}{\|e_1\|_{L^1}}, 1\Big\rangle = -\frac{1}{\|e_1\|_{L^1}}\ip{\lambda e_1}{1} = -\lambda \,,$$
so that the condition $\Lambda < 0$ corresponds exactly to $\lambda > 0$ as assumed in \Cref{reac-diff-domain-thm}. Thus, the proof follows from \Cref{robust-SPDE-thm} (in particular, the last two assertions of \Cref{main-SPDE-thm}).

\end{proof}

We can also make conclusions about the persistence of the population for an arbitrary value of $\epsilon$, but it requires computing a stochastic version of the eigenvalue $-\lambda$, which we call the average Lyapunov exponent, defined as follows.
\begin{deff}\label{reac-diff-domain-avg-lyap-exp}
    Let $\epsilon > 0$ be arbitrary and $P_{inv}(\mcM_0)$ denote the set of all invariant measures of the process $v_t \coloneqq u_t/\|u_t\|_{L^1}$ on $\mcM_0 \coloneqq \{v \in \cdom \mid \|v\|_{L^1} = 1\}$, where $u_t$ solves \eqref{reac-diff-domain-linearized}. We define \begin{equation}\label{eq-reac-diff-domain-Lambda}
        \begin{aligned}
            \Lambda \coloneqq \sup_{\mu \in P_{inv}(\mcM_0)} \mu \Big(&-\int_D d\Delta\Big(vK^{-1}\Big)K^{-1} + (r - E)vK^{-1} + \frac{1}{2}\epsilon^2\sigma'(0)^2\Big[\int_D v\Big]^2 \Big) \,.
        \end{aligned}
    \end{equation}
\end{deff}
\begin{rem}
For non-smooth $v$, the Laplacian of $v$ is understood in the sense of distributions, 
    see \Cref{def:tilde-e} and our alternative definition for $\Lambda$ in \eqref{SPDE-Lambda}, which is used in the proofs.
\end{rem}

\begin{thm}
\label{reac-diff-domain-stoch-thm}
    Fix some $\epsilon > 0$ (not necessarily small) and suppose $\Lambda < 0$ (see \Cref{reac-diff-domain-avg-lyap-exp}). Then \eqref{reac-diff-domain} is stochastically persistent in the sense of \Cref{reac-diff-domain-thm} (the bullet points hold).
\end{thm}

\begin{proof}
    The proof follows immediately from \Cref{main-SPDE-thm}. Indeed, it was already shown in \Cref{reac-diff-domain-thm} that \eqref{reac-diff-domain} satisfies \Cref{sass-A}-\ref{SPDE-lyap}.

\end{proof}

\subsection{Fisher-KPP on a Torus with Space-Time White noise}\label{example-torus}
The following equation is a typical reaction-diffusion equation with applications to the study of turbulence \cite{space-time, physics-interpretation} and biology \cite{biology-waves}. We work in one spatial dimension so that we may add space-time white noise and the solutions are still functions. For simplicity, we consider the equation on a torus $\T$ (periodic boundary conditions):
\begin{equation}\label{torus-equation}
    du = [u_{xx} + u - u^2]dt + \epsilon \sigma(u)dW_t \,,
\end{equation}
where
$W_t$ denotes cylindrical Brownian motion on $L^2(\T)$ (with Lebesgue measure) so that $dW_t/dt$ is space-time white noise. We assume $\sigma: [0,\infty) \to \R$ is a differentiable function with bounded locally Lipschitz derivative which satisfies $\sigma(0) = 0$ and $\epsilon > 0$ is a constant representing the strength of the noise. The following result was already proven in \cite[Theorem 2.4]{space-time}, but we show how it follows from our method. 

\begin{thm}
\label{torus-thm}
    There is some $\overline \epsilon > 0$ such that for all $0 < \epsilon < \overline \epsilon$, \eqref{torus-equation} is stochastically persistent in the sense that: \begin{itemize}
        \item There is an invariant measure $\mu_+$ for \eqref{torus-equation} with $\mu_+(\{0\}) = 0$.
        \item For all nonnegative initial conditions $u_0 \neq 0$, almost surely all limit points $\mu$ of the measures $$\mu_t \coloneqq \frac{1}{t}\int_0^t \delta_{u_s}ds$$ as $t \to \infty$ satisfy $\mu(\{0\}) = 0$.
    
    \end{itemize}
    
\end{thm}

\begin{rem}
    Using \cite[Theorem 2.3]{space-time}, we could replace $\mu_+(\{0\}) = 0$ with $\mu_+(\{u \in \cdom \mid u > 0\}) = 1$, and same for $\mu$ in the second bullet point (see \Cref{rem:strong-maximum-principle}).
\end{rem}

\begin{rem}
    Using \Cref{main-SPDE-thm}, we obtain a similar result for arbitrary $\epsilon > 0$ not necessarily small if we assume $\Lambda < 0$, for $\Lambda$ as in \Cref{def:equiv-lambda-def}. 
    
    This relates to the open problem posed after \cite[Remark 2.5]{space-time} because if one could show that $\Lambda$ is increasing in $\epsilon$ (obviously this is true for fixed $\nu$, but it is not sufficient since the invariant measure $\nu$ depends on $\epsilon$), then a corresponding extinction theory to this paper (analogous to how \cite{extinction} complements \cite{persistence}) would verify the conjecture.
\end{rem}

\begin{rem}
    In \cite{space-time} the linearized version of the deterministic system is $d\psi = \psi_{xx} + \psi - F'(0)\psi$, so their condition (called (F2)) $\limsup_{x \downarrow 0} F'(x) < 1$ is equivalent to our condition $\Lambda < 0$ in \Cref{main-SPDE-thm}. Additionally, \cite[Lemma 2.1(2)]{space-time} is enough to guarantee our \Cref{SPDE-lyap}. Thus, \Cref{torus-thm} holds when $u^2$ is replaced by any $F(u)$ satisfying  (F1)-(F3) in \cite{space-time}, we just chose $u^2$ for concreteness.
\end{rem}

\begin{proof}
    
We use \Cref{robust-SPDE-thm} with $\Theta = [0,1]$ and $\theta_0 = 0$. As in the proof of \Cref{reac-diff-domain-thm}, we have that the value $\Lambda_0$ is equal to $-\lambda$, where $\lambda$ is the largest eigenvalue of the operator $u \mapsto u_{xx} + u$. $\lambda$ is easily seen to be $1$ (the eigenvector is $u \equiv 1$), so that the assumption $\Lambda_0 < 0$ is satisfied.

It remains to verify \Cref{sass-A}-\ref{SPDE-lyap}. We define $Au = u_{xx} - 2u$ and then \Cref{sass-A} follows from \Cref{rem:subtract-from-c} and the fact that the Laplacian is self-adjoint. \Cref{SPDE-F} holds true because $f_2 = 0$, $m=1$ (there is no $z$), and polynomials are smooth. For \Cref{SPDE-G-sass}, we have by definition of space-time white noise that $\U = l^2$ and $Hu = \sum u_ne_n$, where $e_n$ is the $n$th eigenfunction of the Laplacian on the torus. In other words, $a_n = 1$ for all $n$, so \Cref{SPDE-G-sass} is satisfied with $p = \infty$ because $\T$ is one-dimensional ($N=1$). Also, \eqref{eq:l-l1-bound} follows from
$$
\int_\T u_{xx} + u - u^2 \leq \int_\T u_{xx} + 1-u= 1 - \|u\|_{L^1}
$$
(indeed, $(u-1)^2 = u^2 - 2u + 1 \geq 0 \implies u - u^2 \leq 1 - u$).
To verify \Cref{SPDE-lyap}\ref{case-1}, we work with the Sobolev spaces $\B_{\beta,2} = W^{2\beta,2}(\T)$ (see \Cref{characterization-of-interpolation-spaces}) and we use the following auxilliary lemma.

\begin{lem}\label{l1-lemma}
    For all $\beta \in (-1/2,-1/4)$ we have: \begin{enumerate}[label=(\roman*)]
        \item \label{l1-in-h-beta} $L^1(\T) \subset \B_{\beta,2}$ and $\|u\|_\beta \lesssim \|u\|_{L^1}$ for all $u \in L^1(\T)$.
        \item \label{l1-times-l1} If $f,g \in \cdom$ then
        $$\|f\|_{L^1}\|g\|_{L^1} \lesssim \ip{f}{g}_\beta \lesssim \|f\|_{L^1}\|g\|_{L^1} \,.$$
    \end{enumerate}
\end{lem}

\begin{proof}
    The Sobolev embedding $\B_{-\beta,2} = W^{-2\beta,2} \subset L^\infty(\T)$ and duality imply 
    $$
    \|u\|_\beta = \sup_{v \in \B_{-\beta,2}, \|v\|_{-\beta,2} = 1} \ip{u}{v} \lesssim \sup_{v \in \B_{-\beta,2}, \|v\|_{-\beta,2} = 1} \|u\|_{L^1}\|v\|_{L^\infty} \lesssim \|u\|_{L^1} \,.
    $$
and \ref{l1-in-h-beta} follows.  
The  upper bound in \ref{l1-times-l1} is a consequence of Cauchy-Schwartz inequality and \ref{l1-in-h-beta}: 
 \begin{align}
 \ip{f}{g}_\beta &= \ip{A^\beta f}{A^\beta g} \leq \|f\|_{\beta, 2}\|g\|_{\beta, 2} \leq \|f\|_{L^1} \|g\|_{L^1} \,.
 \end{align}
To prove the lower bound, we use \eqref{eq:doapt}, the fact that $A$ is self-adjoint (and therefore $S$ is self-adjoint), the substitution $s + t = r$, 
and a substitution $s/r = q$ to obtain
\begin{align*}
   \ip{f}{g}_\beta &= \frac{1}{(\Gamma(-\beta))^2} 
   \int_{0}^{\infty}\int_{0}^{\infty} t^{-\beta -1} s^{-\beta - 1} 
   \ip{S(t)f}{S(s)g} ds dt \\
   &=  
   \frac{1}{(\Gamma(-\beta))^2} 
   \int_{0}^{\infty}\int_{0}^{\infty} (ts)^{-\beta - 1} 
   \ip{S(t+s)f}{g} ds dt
   \\
   &=
   \frac{1}{(\Gamma(-\beta))^2} 
   \int_{0}^{\infty}\int_{0}^{r} ((r-s)s)^{-\beta - 1} 
   \ip{S(r)f}{g} ds dr
   \\
   &=
   \frac{1}{(\Gamma(-\beta))^2} 
   \int_{0}^{\infty} \ip{S(r)f}{g} r^{-2\beta - 1}\int_{0}^{1} ((1 - q)q)^{-\beta - 1} 
    dq dr \,.
\end{align*}
Using properties of beta function, we have 
$$
\int_{0}^{1} ((1 - q)q)^{-\beta - 1} dq =
\mathcal{B}(-\beta, -\beta) = \frac{(\Gamma(-\beta))^2}{\Gamma(-2\beta)} \,,
$$
and consequently
$$
\ip{f}{g}_\beta = \frac{1}{\Gamma(-2\beta)} 
   \int_{0}^{\infty} \ip{S(r)f}{g} r^{-2\beta - 1} dr \,.
$$

 Let $\Pp_t$ be the heat semigroup, meaning the semigroup generated by $u \mapsto u_{xx} = Au + 2u$, so $\Pp_t = e^{2t}S(t)$. Then continuing the above we have \begin{equation}\label{gamma-inner-product}
    \begin{aligned}
        \ip{f}{g}_\beta &= \frac{1}{\Gamma(-2\beta)}\int_0^\infty t^{-2\beta-1}\ip{S(t)f}{g} dt
        \\
        &= \frac{1}{\Gamma(-2\beta)}\int_0^\infty t^{-2\beta-1}e^{-2t}\ip{\Pp_t f}{g} dt \,.
    \end{aligned}
    \end{equation}

   By Harnack inequality for $f \in \cdom$ we have $\|f\|_{L^1} \lesssim \Pp_tf$ uniformly in $t \in [1,2]$ (for example, use  that $\Pp_tf(x) = \E[f(x + \sqrt{2}B_t)]$,  where $B_t$ is a Brownian motion and the normal distribution projected to the torus has a strictly positive density which is jointly continuous in $t$ and $x$), so that for $t \in [1,2]$
    $$\|f\|_{L^1}\|g\|_{L^1} = \int_\T \|f\|_{L^1}g(x) dx \lesssim \int_\T (\Pp_tf)(x)g(x) dx = \ip{\Pp_t f}{g}$$
    (we used $g \geq 0$ in the first equality). Combined with \eqref{gamma-inner-product}, this proves \ref{l1-times-l1}.
\end{proof}

Next, we proceed with the proof of \Cref{SPDE-lyap}\ref{case-1}. For any $\beta \in (-\frac{1}{2},-\frac{1}{4})$ ($p = \infty$ and $N = 1$ in \Cref{SPDE-lyap}\ref{case-1}) we have  by \Cref{l1-lemma}\ref{l1-in-h-beta} that $\|u\|_{\beta,2}^2 \lesssim \|u\|_{L^1}^2$. Also, by \Cref{l1-lemma}\ref{l1-times-l1} we have
$$\|F^-(u)\|_{L^1}\|u\|_{L^1} \lesssim \ip{F^-(u)}{u}_\beta \,.$$
Since $F(u) = cu - u^2$ for some $c \in \R$ and $u$ vanishes over $u^2$, we have also that $u$ vanishes over $F^-(u) \geq u^2 - |c|u$, and thus for all $K > 0$ there is some $C > 0$ such that $\|F^-(u)\|_{L^1} \geq K\|u\|_{L^1} - C$, so
$$\|u\|_{L^1}\|F^-(u)\|_{L^1} \geq K\|u\|_{L^1}^2 - C\|u\|_{L^1} \geq \frac{K}{2}\|u\|_{L^1}^2 - C' \,,$$
which proves that $\|u\|_{L^1}^2$ (and thus $\|u\|_{\beta,2}^2$) vanishes over $\|F^-(u)\|_{L^1}\|u\|_{L^1}$, and thus $\ip{F^-(u)}{u}_\beta$.

\end{proof}

\subsection{SIR Epidemic Model}\label{example-SIR}
In this section we investigate a stochastic version of the classic SIR epidemic model analyzed numerically in \cite{stochastic-SIR}, which is similar to the one treated analytically in \cite{stochastic-SIR-analytical}.
We consider a disease infecting some portion of a population living in some domain $D \subset \R^N$ which is open, connected, bounded, and has smooth boundary. The function $S$ represents the density of the population which is susceptible and $I$ is the density of population which is currently infected. There are positive parameters $\lambda, \eta, \delta, \sigma$ representing  respectively the recruitment rate, natural death rate, death rate due to disease, recovery rate, as well as diffusion constants $d_1,d_2 > 0$ and $\beta, c_1, c_2 > 0$, $c_3 \geq 0$ governing the rate at which the disease spreads. In fact, we could allow all of these constants to be functions on $D$ and our results would still apply, just with slightly more complicated formulas. We add noise to the death rate $\eta$, so at each time $t$ we have a death rate of $\eta + \alpha dw/dt$, where $w^{(1)}, w^{(2)}$ are independent Brownian motions and $\alpha_1, \alpha_2 \geq 0$ govern the strength of the noise. (We could have made our example more complicated by choosing space-time white noise in one dimension like in \Cref{example-torus}, or the type of noise considered in \Cref{sec:lotka-volterra}, a generalization of the Hilbert-Schmidt noise considered in \cite{stochastic-SIR-NN}, but we already have examples with those types of noise, so we kept it simple for this one.)  The system is:

\begin{equation}\label{SIR}
    \begin{aligned}
        dS &= \Big[d_1\Delta S + \lambda - \eta S - \frac{\beta SI}{1 + c_1S + c_2I + c_3SI}\Big]dt + \alpha_1 Sdw^{(1)}_t \\
        dI &= \Big[d_2\Delta I - (\eta + \delta + \sigma) I + \frac{\beta SI}{1 + c_1S + c_2I + c_3SI}\Big]dt + \alpha_2 Idw^{(2)}_t
    \end{aligned}
\end{equation} with Neumann boundary conditions ($dS/dn = dI/dn = 0$ on $\partial D$) and independent standard Brownian motions $w^{(1)}, w^{(2)}$. We study conditions under which the disease persists.

In the deterministic setting ($\alpha_1 = \alpha_2 = 0$), it is known (\cite[Lemma 2.1]{deterministic-SIR}) that the system has an endemic equilibrium (one where $I \neq 0$) if the quantity
\begin{equation}\label{eq:tilde-lambda}
    \tilde \lambda \coloneqq \frac{\beta \lambda - (\eta + \delta + \sigma)(c_1 \lambda + \eta)}{c_1\lambda + \eta} > 0 \,.
\end{equation}
Via an analysis of the linearized system \begin{equation}\label{SIR-linearized}
    \begin{aligned}
        d\tilde S &= \Big[d_1\Delta \tilde S + \lambda - \eta \tilde S\Big]dt + \alpha_1 \tilde Sdw^{(1)}_t \\
        d\tilde I &= \Big[d_2\Delta \tilde I - (\eta + \delta + \sigma) \tilde I + \frac{\beta \tilde S\tilde I}{1 + c_1\tilde S}\Big]dt + \alpha_2 \tilde Idw^{(2)}_t \,,
    \end{aligned}
\end{equation} we can make similar statements about the stochastic system.

\begin{thm}
\label{SIR-thm}
    Suppose either: \begin{itemize}
        \item $\tilde \lambda > 0$ and $\alpha_1, \alpha_2$ are small enough.
        \item $\Lambda < 0$, where 
        \begin{equation}\label{eq:sir-lambda}
             \begin{aligned}
        \Lambda \coloneqq \sup_{\mu \in P_{inv}(\mcM_0)} \mu\Big(-\int_D \frac{\beta \tilde S}{1 + c_1 \tilde S}\tilde I\Big) + \eta + \delta + \sigma + \frac{1}{2}\alpha_2^2
    \end{aligned}
        \end{equation}
 and $P_{inv}(\mcM_0)$ denotes the set of all invariant measures of the process $(\tilde S_t, \tilde I_t / \|\tilde I_t\|_{L^1})$ on
    $$\mcM_0 \coloneqq \{(S,I) \in C(\overline D;\R^2) \mid \|I\|_{L^1} = 1, S \geq 0, I \geq 0\} \,,$$
    where $(\tilde S_t, \tilde I_t)$ solves \eqref{SIR-linearized}.
    \end{itemize} 

    Then the disease is stochastically persistent in the sense that
    \begin{itemize}
        \item There is an invariant measure $\mu_+$ for \eqref{SIR} with $\mu_+(\{I = 0\}) = 0$.
        \item For all $\delta > 0$ there are $0 < b < B < \infty$ such that for any nonnegative initial condition $(S_0,I_0)$ with $I_0 \neq 0$, $$
        \liminf_{T \to \infty} \frac{1}{T}m\Big(\Big\{t \leq T \mid b \leq \int_D I_t \leq B\Big\}\Big) \geq 1 - \delta$$ 
        almost surely, where $m$ is the Lebesgue measure.
    \end{itemize}
\end{thm}

\begin{rem}
    It should be noted that in \Cref{SIR-thm} we obtain persistence in the almost sure sense, as opposed to the statements about expectation in \cite[Theorem 4.1]{stochastic-SIR-NN}, \cite[Theorem 3.3]{stochastic-SEIRS}, \cite[Theorem 3.1]{stochastic-SIR-analytical}. Additionally, we believe our criteria for persistence are tighter than in 
    \cite{stochastic-SIR-NN, stochastic-SEIRS, stochastic-SIR-analytical} if $\beta,\eta,$ etc. are allowed to depend on the space variable. Indeed, in \cite{stochastic-SIR-NN, stochastic-SEIRS, stochastic-SIR-analytical} it is not obvious if the threshold for persistence is the same as the threshold for extinction (except in the case of constant $\beta,\eta,$ etc.). In the forthcoming project we will present a extinction  theory (as \cite{extinction} complements \cite{persistence}) showing that if the infimum of the expression appearing in the definition of $\Lambda$ (see \eqref{eq:sir-lambda}) is positive then extinction occurs. 
\end{rem}

\begin{proof}
    
    We start by verifying \Cref{sass-A}-\ref{SPDE-lyap}. We note that in the notation of \Cref{section-SPDE} we have $m = 2$, $d = 1$, $u = I$ denotes the species which persistence we investigate, and $z = S$ represents the auxiliary species. Since the Laplacian is self-adjoint, \Cref{sass-A} follows from \Cref{rem:subtract-from-c}. For \Cref{SPDE-F} we have (ignoring the aforementioned $c$) 
    $$
    f_1(I,S) = \Big(-(\eta+\delta+\sigma)+\frac{\beta S}{1 + c_1S + c_2I + c_3SI},-\eta-\frac{\beta I}{1 + c_1S + c_2I + c_3SI}\Big)
    $$ 
    and $f_2(I,S) = (0,\lambda)$, so \Cref{SPDE-F} is easily verified. \Cref{SPDE-G-sass}  also follows with $\sigma(I,S) = (I,S)$, $Hu_1 = (1,0)$, $Hu_2 = (0,1)$, and $Hu_n = 0$ for all other $n$. Indeed, \eqref{eq:an-en-assumption} holds with $p = 2$. Finally, all of our terms are of linear order or lower, so \Cref{SPDE-lyap}\ref{case-2} holds, and \eqref{eq:l-l1-bound} follows from
    Neumann boundary conditions and integration by parts:
    \begin{equation*}
        \ip{Ax + F(x)}{1} = \int_D d_1\Delta S + d_2\Delta I + \lambda - \eta S - (\eta + \delta +\sigma)I \leq \lambda - \eta\|(I,S)\|_{L^1} \,.
    \end{equation*}
Next, we claim that \eqref{eq:sir-lambda} is equal to $-\tilde \lambda$ when $\alpha_1 = \alpha_2 = 0$. 
Indeed, by \eqref{SIR-linearized}, the equation for $\tilde S$ is independent of $\tilde I$. If $\alpha_1 = \alpha_2 = 0$, then $\tilde S_E = \lambda/\eta$ is the only equilibrium of 
$$
\partial_t\tilde S = d_1 \tilde \Delta S + \lambda - \eta \tilde S \,.
$$ 
because, for example by the energy method, we obtain $\lim_{t \to \infty} \tilde S_t = \lambda / \eta$ (see also \Cref{rem:eigenvalue-for-SIR}). 
Thus, all invariant measures $\mu \in P_{inv}(\mcM_0)$ when $\alpha_1 = \alpha_2 = 0$ satisfy $\mu(\{\tilde S = \lambda/\eta\}) = 1$.

Replacing $\tilde S$ with $\lambda /\eta$ in \eqref{eq:sir-lambda} and using $\int_D \tilde I = \|\tilde I\|_{L^1} = 1$ when $(\tilde S, \tilde I) \sim \mu \in P_{inv}(\mcM_0)$, we indeed obtain $\Lambda = -\tilde \lambda$ when $\alpha_1 = \alpha_2 = 0$.

Finally, the claim follows from \Cref{main-SPDE-thm} and \Cref{robust-SPDE-thm} (applied to \Cref{main-SPDE-thm}).
        
\end{proof}
\begin{rem}\label{rem:eigenvalue-for-SIR}
    Notice that the value of $\Lambda$ when the noise terms are dropped is the smallest eigenvalue (the only eigenvalue with nonnegative eigenvector) of the operator $-[d_2\Delta - (\eta + \delta + \sigma) + \frac{\beta \lambda}{\eta + c_1 \lambda}]$. See also the proof of \Cref{reac-diff-domain-thm}, where this phenomena is illustrated directly, and \Cref{rem:rayleigh-quotient}, which gives an intuitive explanation.
\end{rem}

\begin{rem}
    Observe that if our ``functional response" $\beta SI/(1+c_1S+c_2I+c_3SI)$ was replaced with $\beta SI/(S+I)$, which is the one considered in \cite{stochastic-SIR-NN}, then technically \Cref{SPDE-F} would not hold because $xy/(x+y)$ is not continuously differentiable. In other words, we cannot linearize the problem. However, we believe our results should still apply in some sense. For example, consider \cite[(2.4)]{stochastic-SIR-NN}. Introduce a parameter $\epsilon \geq 0$ and consider 
    \begin{equation}\label{eq:sir-perturbed-example}
            \begin{aligned}
        dS(t) &= \Big[A_1S(t) + \Lambda - \mu_1S(t) - \frac{\alpha S(t)I(t)}{S(t) + I(t)}\Big]dt + S(t)dW_1(t) \\
        dI(t) &= \Big[A_2I(t) - \mu_2I(t) + \frac{\alpha S(t)I(t)}{S(t)+I(t)+\epsilon}\Big]dt + I(t)dW_2(t)
    \end{aligned}
    \end{equation}

    The corresponding process $X_t = (S_t,I_t,\epsilon)$ satisfies \Cref{as1}-\ref{as-W}, \Cref{as-U}\ref{5.1}-\ref{5.3}, and \Cref{as-compact} on the state space $\mcM = C(\overline D;\R^2)_+ \times [0,1]$. Thus, its set of invariant measures is compact, its empirical occupation measures are almost surely tight, it is Feller, etc. If $\sup_{\epsilon \in (0,\overline \epsilon)} \Lambda_\epsilon < 0$, then \eqref{eq:sir-perturbed-example} is stochastically persistent with respect to $\mcM_0 = C(\overline D;\R)_+ \times \{0\}$ for all fixed $\epsilon \in (0,\overline \epsilon)$. We believe this should imply that it is also persistent for $\epsilon = 0$. Such statement  requires proving an extra `uniform' tightness of $\mu_t$ on $\inv$. It would be interesting to further explore this property, because in general a sort of uniform tightness of $\mu_t$ should also give us a rate of convergence of $\mu_t$ to $P_{inv}(\mcM_0)$ (see for example \cite[Theorem 4.11, Theorem 4.13]{persistence} for a result in the finite-dimensional case).
\end{rem}

\subsection{Stochastic Competitive Lotka-Volterra}\label{sec:lotka-volterra}
Consider the population of $m$ species $u_1,\dots,u_m$ evolving on some domain $D \subset \R^N$ which is connected, open, bounded, and has smooth boundary. Assume $a_{ii}: \overline D \to (0,\infty)$, $a_{ij}: \overline D \to [0,\infty)$, $m_i: \overline D \to \R$, and $\sigma_i: [0,\infty) \to \R$ are all differentiable with bounded Lipschitz derivative and additionally, $\sigma_i(0) = 0$. Fix constants $d_i > 0$ and suppose that the driving noise $W_i(t)$ is given by independent cylindrical $Q$-Wiener processes ($Q$ is the covariance operator):
$$W_i(t) = \sum_{k = 1}^\infty \lambda_{ik}e_kw_{ik}(t) \,,$$
where $w_{ik}(t)$ are independent standard (one-dimensional) Brownian motions, the functions $\{e_k\}_{k = 1}^\infty$ form an orthonormal basis for $L^2(D)$ consisting of eigenfunctions of $\Delta$, and either $N>1$ and $\lambda_{ik} \geq 0$ satisfies
$$\sum_{i=1}^m\sum_{k=1}^\infty \lambda_{ik}^p\|e_k\|_{L^\infty}^2 < \infty$$
for some $p \in [2,2N/(N-2))$, or $N=1$ and $\sup_{ik} |\lambda_{ik}| < \infty$.

\begin{rem}
If $p = 2$ then $W_i$ is a finite-trace Q-Wiener process (it is a Hilbert-Schmidt embedding of a cylindrical Wiener process). If $N = 1$ and $\lambda_{ik} = 1$, then $W_i$ is a cylindrical Wiener process, so $dW_i(t)/dt$ represents space-time white noise. Otherwise, we have something in between. See \cite[Section 8]{stochastic-lotka-volterra-SPDE} for further discussion about the choice of noise, and notice that our assumption is weaker than \cite[Hypothesis 8.1]{stochastic-lotka-volterra-SPDE} because we do not assume that $\{e_{ik}\}$ is equibounded in $L^\infty$.
\end{rem}

Our SPDE is as follows:

\begin{equation}\label{eq:lotka-vol}
    dx_i = \Big[d_i\Delta x_i + x_i\Big(m_i - a_{ii}x_i - \sum_{j \neq i}a_{ij} x_j\Big)\Big]dt + \sigma_i(x_i)dW_i(t) \,,
\end{equation}
where $x_i$ satisfies Neumann boundary conditions ($dx_i/dn = 0$ on $\partial D$, where $n$ is the outward pointing normal vector to $\partial D$).

Then \eqref{eq:lotka-vol} induces a Markov process on $\mcM \coloneqq C(\overline D;\R^m)_+$, the set of nonnegative continuous $\R^d$-valued functions on $\overline D$ (this is a consequence of \Cref{well-posedness-of-SPDE}). We investigate when the species can coexist, meaning that the following set is a ``repeller":
$$\mcM_0 \coloneqq \{x \in \mcM \mid x_i = 0 \text{ for some } i = 1,\dots,m\} \,.$$

As in the finite-dimensional cases like SDEs and discrete time Markov chains (see \cite{ecologicalContinuous, ecologicalDiscrete, ecologicalGeneral}), this question can be answered by considering the ``invasion rates" $r_i(\mu)$ defined below. Intuitively, we imagine that a species $i$ is extinct and other species are in a steady state $\mu$. Then, 
the $i$th invasion rate measures the average exponential rate of growth of the $i$th population if a small amount of species $i$ is introduced into the ecosystem.
The exponential rate of growth can be obtained by formally applying It\^{o}'s formula to the function $\ln{\int_D x_i}$, where $\int_D x_i = \|x_i\|_{L^1}$ is the total population of species $i$, and then to obtain $r_i(\mu)$ we take average according to the measure $\mu$. If the species was not extinct we define the invasion rate to be zero. We arrive at the following definition, which should be compared to \cite[(2.3) and Proposition 3.1]{ecologicalGeneral}.

\begin{deff}\label{def:invasion-rates}
    Fix $\mu \in P_{inv}(\mcM_0)$ (see \Cref{inv-meas}), $k \in \{1,\dots,m\}$, and suppose $\mu(\{x_k = 0\}) = 1$ 
    Let $\tilde{x}$ be a process satisfying
    \begin{align*}
        d\tilde x_k &= [d_k\Delta \tilde x_k + \tilde x_k(m_k -\sum_{j\neq k}a_{kj}\tilde x_j)]dt + \sigma_k'(0)\tilde x_kdW_k(t) \\
        d\tilde x_i &= [d_i\Delta \tilde x_i + \tilde x_i(m_i - a_{ii}\tilde x_i -\sum_{j \neq i,k}a_{ij} \tilde x_j)]dt + \sigma_i(\tilde x_i)dW_i(t) \qquad i \neq k \,. 
    \end{align*}
    Define 
    \begin{equation}\label{eq:lotka-vol-invasion-rate}
        r_k(\mu) \coloneqq \inf_{\nu \in P_k(\mu)} \nu\Big(\int_D v_k\Big(m_k - \sum_{j\neq k} a_{kj}\tilde x_j\Big) - \frac{1}{2} \sigma_k'(0)^2\sum_{i=1}^\infty \lambda_{ki}^2\Big(\int_D e_{ki} v_k\Big)^2\Big) \,,
    \end{equation}
    where $P_k(\mu)$ denotes the set of all invariant measures (see  \Cref{inv-meas}) $\nu$ for the process $(v_k,\tilde x_{-k}) \coloneqq (\tilde x_k/\|\tilde x_k\|_{L^1}, \tilde x_{-k})$ satisfying $\phi^*\nu = \mu$. Here, $\phi$ is given by $\phi(v,x_{-k}) = (x_1,\dots,x_{k-1},0,x_{k+1},\dots,x_m)$, $\phi^*\nu(A) = \nu(\phi^{-1}(A))$, and  $x_{-k} = ( x_1,\dots, x_{k-1}, x_{k+1},\dots, x_m)$.

    For a general $\mu \in P_{inv}(\mcM_0)$, by invariance there is $p \in [0,1]$ and $\mu_1, \mu_2 \in P_{inv}(\mcM_0)$ such that $\mu_1(\{x_k = 0\}) = 1$, $\mu_2(\{x_k = 0\}) = 0$, and $\mu = p\mu_1 + (1-p)\mu_2$. Then we 
    we define $r_k(\mu) = pr_k(\mu_1)$. 
\end{deff}

\begin{rem}
    If $\mu$ is ergodic then $\mu(\{x_k = 0\}) \in \{0,1\}$ because $\{x_k = 0\}$ is an invariant set for \eqref{eq:lotka-vol}. This is intuitively clear and follows rigorously from \Cref{bigthm}\ref{bigthm:invt-set} with $R=1$. Indeed, we will prove below that \Cref{sass-A}-\ref{SPDE-G-sass} hold for \eqref{eq:lotka-vol}. Thus, we could also define $r_k(\mu)$ for ergodic $\mu$ first and then in general via ergodic decomposition.
\end{rem}

\begin{rem}\label{rem:on-inv-rates}
    To provide intuition for \Cref{def:invasion-rates}, note that the expression inside of $\nu()$ in \eqref{eq:lotka-vol-invasion-rate} is identical to $\Ll \ln \|x_k\|_{L^1}$ when $x_k = 0$, modulo the tildes, provided we identify $v_k$ with $x_k/\|x_k\|_{L^1}$.
    
    This is a valid identification because  writing the equation for $x_k/\|x_k\|_{L^1}$ and then setting $x_k = 0$,  yields the same result as if we first linearized the equation for $x_k$. Hence, formally when $x_k = 0$ we have $x_k/\|x_k\|_{L^1} = \tilde x_k/\|\tilde x_k\|_{L^1}$ and the equations for $\tilde x_i$ are obtained simply by setting $x_k = 0$. 
    
    Note that we are averaging over $\nu \in P_k(\mu)$ as opposed to $\mu$. The reason is that before we were looking at an invariant measure $\mu$ where $x_k$ had only one possible value ($0$). However, the expression we need to average contains $v_k$, and since $\|x_k\|_{L^1} = 0$ the polar decomposition $x_k = \|x_k\|_{L^1}v_k$ leaves many possible options for $v_k$ (any $v_k \in C(\overline D;\R)_+$ with $\|v_k\|_{L^1} = 1$). Thus, we must consider all possible stationary behaviors for $v_k$ given that the other species are still distributed according to $\mu$. This detail was not encountered in the SDE or discrete-time Markov chain case where $x_k/x_k$ has only one possible value ($1$) since each $x_k$ lives in the one-dimensional space $\R$ as opposed to the infinite-dimensional space $C(\overline D;\R)_+$.
\end{rem}

\begin{rem}
Below we list some common cases for the noise which simplify the formula \eqref{eq:lotka-vol-invasion-rate}.
When $\lambda_{ki} = 1$ (space-time white noise), by Plancherel's theorem
$$
\sum_{i=1}^\infty \lambda_{ki}^2\Big(\int_D e_{ki} v_k\Big)^2 = \|v_k\|_{L^2}^2 \,.
$$  
Similarly if $\lambda_{ki} = \lambda_i^{\alpha/2}$,  where $\lambda_i$ is the $i$th eigenvalue of the Laplacian (cylindrical Brownian motion on the Sobolev space $H^\alpha(D)$),  then 
$$\sum_{i=1}^\infty \lambda_{ki}^2\Big(\int_D e_{ki} v_k\Big)^2 = \|v_k\|_{H^\alpha}^2 \,,
$$
where $H^\alpha$ is the $\alpha$-Sobolev norm. If the noise terms are standard Brownian motions, meaning $\lambda_{ki} = 1$ for $i = 1$ and $\lambda_{ki} = 0$ otherwise, then $e_{k1} = 1$ for all $k$ and 
we have 
$$
\sum_{i=1}^\infty \lambda_{ki}^2\Big(\int_D e_{ki} v_k\Big)^2 = 1 \,,
$$ 
because $\|v\|_{L^1} = 1$.
\end{rem}

\begin{rem}
    In the SDE case there is no $\inf$ in \eqref{eq:lotka-vol-invasion-rate} because $L^1(D) \cong \R$ when $D$ is a singleton (no spatial dependence) which forces $v_k = 1$. Since we have spatial dependence, the simplex in $L^1(D)$ is a much richer space and the set $P_k(\mu)$ is needed. However, as mentioned in the introduction, it may be possible to use a contraction argument as in \cite{BenaimLobrySari25, Bushell86, cohenfausti2024} to show that the $\nu$ in \eqref{eq:lotka-vol-invasion-rate} is unique ($P_k(\mu)$ is singleton), and then 
    $\inf$ could be removed. 
\end{rem}

\begin{thm}\label{lotka-volterra-main}
Suppose that $\max_{i=1,\dots,m} r_i(\mu) > 0$ for all $\mu \in P_{inv}(\mcM_0)$. Then coexistence (persistence) occurs in \eqref{eq:lotka-vol}. Specifically,
\begin{itemize}
    \item There is an invariant measure $\mu_+$ for \eqref{eq:lotka-vol} with $\mu_+(\mcM_0) = 0$.
    \item For all $\delta > 0$ there are $0 < b < B < \infty$ such that for any nonnegative initial condition $u \notin \mcM_0$,
    $$\liminf_{T \to \infty} \frac{1}{T}m(\{t \leq T \mid b \leq \int_D x_i(t) \leq B \text{ for all } i\}) \geq 1 - \delta$$ almost surely.
\end{itemize}
\end{thm}

\begin{rem}
    It follows from \Cref{robust-SPDE-thm} that the invasion rates are continuous in $m_i,a_{ij},\sigma_i$, so in fact coexistence still occurs if we slightly perturb these functions (compare to \cite[Theorem 2.8]{ecologicalGeneral}). Also, in the deterministic case ($\sigma_i = 0$) it may be possible to give more explicit descriptions of the invasion rates, for example in terms of solutions to eigenvalue problems. Specifically, if $\sigma_k = 0$ then any $\nu \in P_k(\mu)$ (see \Cref{def:invasion-rates})  satisfies $\nu(\{v_k = v^*\}) = 1$, where $v^*$ is the nonnegative eigenvector of a linear PDE for $\tilde x_k$. Then one could use similar methods to \cite[Section 5.5.]{ecologicalGeneral} to calculate the invasion rates, at least when $m_i,a_{ij}$ are constants. It would be interesting to see how this aligns  with the results in \cite{deterministic-lotka-volterra} for corresponding deterministic models since, combined with the continuous dependence on $\sigma_i$, the results would extend to the case of small noise. We already illustrated such an argument in \Cref{SIR-thm} for the SIR epidemic model, so we skip the details here, but we encourage an intersted reader to calculate and/or bound $r_k(\mu)$ for  specific examples.
\end{rem}

\begin{proof}
    To use \Cref{main-SPDE-thm-2}, we first verify \Cref{sass-A}-\ref{SPDE-lyap}. Since there is no parameter $\theta$ we do not include it in the notation. \Cref{sass-A} follows from the Laplacian being self-adjoint (see \Cref{rem:subtract-from-c}) and $A_i = d_i\Delta - c$, where $c > 0$ is large enough. For \Cref{SPDE-F} we have $f_2 = 0$ and $(f_1(\cdot,x))_i = c + m_i(\cdot) - a_{ii}(\cdot)x_i - \sum_{j\neq i} a_{ij}(\cdot)x_j$, so \Cref{SPDE-F} is satisfied for $x_t = (x_i(t),x_{-i}(t))$ for any $i \in \{1,\dots,m\}$. \Cref{SPDE-G-sass} is satisfied with $\sigma(\cdot,x) = (\sigma_1(x_1),\dots,\sigma_m(x_m))$, $\U = (l^2)^m$, and $Hu_{ik} = \lambda_{ik}\tilde e_{ik}$, where $\{\{u_{ik}\}_{i=1,\dots,m}\}_{k=1}^\infty$ is an orthonormal basis for $(l^2)^m$ and $\tilde e_{ik}$ is the function in $C(\overline D;\R^m)$ with $i$th coordinate equal to $e_k$ and all other coordinates equal to $0$ (indeed, $\{\tilde e_{ik}\}$ forms an orthonormal basis for $(L^2(D))^m$).

    It remains to verify \Cref{SPDE-lyap}. First \eqref{eq:l-l1-bound} follows from 
    \begin{align*}
        \ip{Ax + F(x)}{1} &= \ip{xf_1(x)}{1} \leq K - \sum_{i} a_{ii} \ip{x_i^2}{1} - \sum_{i,j} a_{ij} \ip{x_ix_j}{1} \leq K - \tilde c\|x\|_{L^2}^2 \\
        &\leq K - \tilde c\|x\|_{L^1} \,,
    \end{align*}
    where we used $a_{ii} > 0$ on $\overline{D}$, $x_i, a_{i, j} \geq 0$, and Young's inequality. 
    
    To check \Cref{SPDE-lyap}\ref{case-1} we use ideas from \Cref{l1-lemma}, in particular, recall \eqref{gamma-inner-product} (the derivation did not use any conditions specific to that example). 
    As in \eqref{gamma-inner-product}, we have 
    $$
    \ip{f}{g}_\beta = C(\beta)\int_0^\infty t^{-2\beta-1}e^{-ct}\ip{\Pp_t f}{g} dt \,,
    $$
    where $C(\beta) > 0$ is a constant and $\Pp_t$ is the (contraction) semigroup generated by $Ax = (d_1\Delta x_1,\dots,d_m \Delta x_m)$ on $C(\overline D;\R^m)$ with Neumann boundary conditions ($\Pp_t x$ is the solution to $\frac{\partial}{\partial t} \Pp_t x = A\Pp_t x$). As $\Pp_t$ is positivity preserving, we conclude that $\ip{\cdot}{\cdot}_\beta$ is monotone in each argument (for nonnegative $f,g$). Thus, for every $\epsilon > 0$ we have
    \begin{align*}
        \ip{x_i}{F_i^-(x)}_\beta &\geq \ip{x_i}{-cx_i - m_ix_i + a_{ii}x_i^2}_\beta \\
        &\geq \ip{x_i}{\epsilon^{-1} x_i - K}_{\beta,2} \\
        &\geq \epsilon^{-1} \|x_i\|_{\beta,2}^2 - K\|x_i\|_{\beta,2} \\
        &\geq \frac{\epsilon^{-1}}{2}\|x_i\|_{\beta,2}^2 - K
    \end{align*}
    for constants $K = K(\epsilon) > 0$ which are allowed to change from line to line, and therefore $\|x_i\|_{\beta,2}^2$ vanishes over $\ip{x_i}{F_i^-(x)}_\beta$. Summing over $i$ proves \Cref{SPDE-lyap}\ref{case-1}.

Finally, we need to reconcile \Cref{def:invasion-rates} with \eqref{eq:inv-rates-general}. If $\mu(\{x_k = 0\}) = 1$, then \Cref{def:invasion-rates}  agrees with \eqref{eq:inv-rates-general} by exactly the same reasoning as why \Cref{def:equiv-lambda-def} is the same as \eqref{SPDE-Lambda} (see the paragraph right below \eqref{SPDE-Lambda}). 

When $\mu(\{x_k = 0\}) = 0$, then $r_k(\mu) = 0$ in \Cref{def:invasion-rates} and 
in \eqref{eq:inv-rates-general} we have 
$r_k(\mu) = \mu H = 0$ if $\mu \in P_{inv}(\inv)$, where we used \Cref{muH-neg} applied in the setting of \Cref{def:rv-process} with $x_t = (x_k(t),x_{-k}(t))$ ($V$ is as in the proof of \Cref{main-SPDE-thm}, in which \Cref{as1}-\ref{as-compact} are verified, so we may indeed apply \Cref{muH-neg}).
    
\end{proof}

\begin{rem}
     Our results also apply to predator-prey systems with minor adjustments. For example, if we considered an SPDE version of \cite[(5.20)]{ecologicalGeneral}, then \Cref{SPDE-lyap}\ref{case-1} technically does not hold because $\ip{x}{F^{-1}(x)}_\beta$ contains terms like $\ip{x_3-x_1}{a_{13}x_1x_3}_\beta$. This can be bypassed by using $\|\sum_{i=1}^m x_i\|_{\beta,2}^2$ instead of $\|x\|_{\beta,2}^2 = \sum_{i=1}^m \|x_i\|_{\beta,2}^2$. Indeed, in this case $\ip{x}{F^{-1}(x)}_\beta$ is replaced by $\ip{\sum_{i=1}^m x_i}{\sum_{i=1}^m F_i^-(x)}_\beta$, and then terms $a_{13}x_1x_3$ and $-a_{13}x_1x_3$ in $\sum_{i=1}^m F_i^-(x)$ cancel (as in the calculation of $\ip{F(x)}{1}$). This also applies to the SDE case; we would 
     use $(x_1+x_2+x_3)^2$ as 
     a Lyapunov function instead of $x_1^2+x_2^2+x_3^2$ in \cite[(5.20)]{ecologicalGeneral}. We decided, for the sake of brevity and clarity, not to include a third case to \Cref{SPDE-lyap}, which would have been quite similar to \Cref{SPDE-lyap}\ref{case-1}. We are confident that all of the arguments are still valid with this minor change and if we were to write this paper again we would use $\|\sum_{i=1}^m x_i\|_{\beta,2}^2$ in place of $\|x\|_{\beta,2}^2$ for \Cref{SPDE-lyap}\ref{case-1}. 
\end{rem}

\begin{rem}
    Our results will also apply when $d_i$ are allowed to be functions on $D$. However, we avoided this technicality because then we would have to use a measure (different in each component, see \Cref{example-logistic}) other than Lebesgue measure on $D$ to make the operator $d_i\Delta$ self-adjoint. Our general results assume each operator $d_i\Delta$ is self-adjoint with respect to the same measure, but all results in \Cref{section-SPDE} follow without a change if we use spaces like $L^p(D,\Sigma,\mu_1) \oplus \dots \oplus L^p(D,\Sigma,\mu_m)$ as opposed to $L^p(D,\Sigma,\mu;\R^m)$.
\end{rem}

\addcontentsline{toc}{section}{References}

\newcommand{\etalchar}[1]{$^{#1}$}
\providecommand{\bysame}{\leavevmode\hbox to3em{\hrulefill}\thinspace}
\providecommand{\MR}{\relax\ifhmode\unskip\space\fi MR }
\providecommand{\MRhref}[2]{%
  \href{http://www.ams.org/mathscinet-getitem?mr=#1}{#2}
}
\providecommand{\href}[2]{#2}

\end{document}